\documentclass{article}
\usepackage{graphicx} % Required for inserting images
\usepackage{amsmath} %lots of math symbols
\usepackage[colorlinks=true,bookmarks=false,linkcolor=blue,urlcolor=blue,citecolor=blue,breaklinks=true]{hyperref}
\usepackage[T1]{fontenc} % gets rid of "Font shape `OMS/cmtt/m/n' undefined" warning
\usepackage{breakcites} %--JOE: I commented this out because it was clashing with hyperref

\usepackage{doi}
\usepackage{graphicx}
\usepackage{amssymb} % \ll
\usepackage{amsfonts} %\mathbb
\usepackage{amsthm} %proof environment %--JOE: I commented this out because of an apparent clash with SIAM.
\usepackage[shortlabels]{enumitem} %resume enumerate
\usepackage{comment} %comment environment
\usepackage{bbm}
\usepackage{mathtools}
\usepackage{algorithm}
\usepackage{algpseudocode}
\usepackage[margin=1in]{geometry} %NB: this definitely goes against SIAM format
\usepackage{subcaption}
\usepackage{multirow} % combine rows vertically in a table
\usepackage{pifont} % needed for check and x marks
\usepackage{booktabs} % better horizontal lines for tables
\usepackage{mathrsfs}
\usepackage{tikz-cd}
\usetikzlibrary{calc, positioning}
\usepackage[numbers,sort&compress]{natbib}

\usepackage{multicol}

\makeatletter
\renewcommand{\paragraph}{%
  \@startsection{paragraph}{4}%
  {\z@}{1ex \@plus 1ex \@minus .2ex}{-.5em}%
  {\normalfont\normalsize\bfseries}%
}
\makeatother

\newtheorem{theorem}{Theorem}[section]
\newtheorem{lemma}[theorem]{Lemma}
\newtheorem{proposition}[theorem]{Proposition}
\newtheorem{corollary}[theorem]{Corollary}
\newtheorem{example}[theorem]{Example}
\newtheorem{definition}[theorem]{Definition}
\newtheorem{remark}[theorem]{Remark}

\newtheorem{conjecture}[theorem]{Conjecture}

\newcommand{\be}{\begin{equation}}
\newcommand{\ee}{\end{equation}}
\newcommand{\bthm}{\begin{theorem}}
\newcommand{\ethm}{\end{theorem}}
\newcommand{\blem}{\begin{lemma}}
\newcommand{\elem}{\end{lemma}}
\newcommand{\bpof}{\begin{proof}}
\newcommand{\epof}{\end{proof}}
\newcommand{\bcor}{\begin{corollary}}
\newcommand{\ecor}{\end{corollary}}
\newcommand{\bprop}{\begin{proposition}}
\newcommand{\eprop}{\end{proposition}}

\newcommand{\diag}{\mathrm{diag}}

\newcommand{\mc}[1]{\mathcal{#1}}
\newcommand{\mbb}[1]{\mathbb{#1}}

\newcommand{\msf}[1]{\mathsf{#1}}

\newcommand{\RR}{\mathbb{R}}

\newcommand{\NN}{\mathbb{N}}

\title{Limits of Weighted Graphs via Random Quotients}
\author{Eitan Levin$^\dag$ and Venkat Chandrasekaran$^{\dag,\ddag}$ \thanks{Email: \texttt{\{eitanl, venkatc\}@caltech.edu}} \vspace{0.25in} \\ $^\dag$ Department of Computing and Mathematical Sciences\\ $^\ddag$ Department of Electrical Engineering \\ California Institute of Technology \\ Pasadena, CA 91125}
\date{March 20, 2026}

\begin{document}

\maketitle

\begin{abstract}
% We present a new notion of limits of weighted directed graphs of growing size based on convergence of their random quotients.  These limits are specified in terms of random exchangeable measures on the unit square.  We call our limits \emph{grapheurs} and show that these are dual to graphons in a precise sense. Grapheurs are well-suited to modeling hubs and connections between them in large graphs; previous notions of graph limits based on subgraph densities fail to adequately model such global structures as subgraphs are inherently local.  Using our framework, we present an edge-based sampling approach for testing properties pertaining to hubs in large graphs.  This method relies on an edge-based analog of the Szemer\'edi regularity lemma, whereby we show that sampling a small number of edges from a large graph approximately preserves its quotients. Finally, we observe that the random quotients of a graph are related to each other by equipartitions, and we conclude with a characterization of such random graph models.\\

We present a new notion of limits of weighted directed graphs of growing size based on convergence of their random quotients.  These limits are specified in terms of random exchangeable measures on the unit square.  We call our limits \emph{grapheurs} and show that these are dual to graphons in a precise sense. Grapheurs are well-suited to modeling global structure in large graphs such as hubs and connections between them; previous notions of graph limits based on subgraph densities fail to adequately model such global structure as subgraphs are inherently local.  Using our framework, we characterize properties of large graphs that are continuous with respect to our limits and present an edge-based sampling approach for testing them.  This method relies on an edge-based analog of the Szemer\'edi regularity lemma, whereby we show that sampling a constant number of edges from an arbitrarily-large graph approximately preserves its quotients. Finally, we observe that the random quotients of a graph are related to each other by equipartitions, and we conclude with a characterization of such random graph models.\\

% We consider summaries of large weighted directed graphs via random quotients.  A sequence of graphs of growing size is then said to \emph{quotient-converge} if the associated quotients converge.

% Grapheurs model the asymptotic distribution of edge weights in a sequence of graphs, in particular the presence of global attributes such as ; previous notions of graph limits based on subgraph densities fail to adequately model such global structures as subgraphs are inherently local.

% whereby we show that any graph can be summarized by one with a small number of edges such that random quotients of the large graph are close to those of their smaller summary.

% We use our framework to test properties pertaining to hubs.  We do so  we present an edge-based sampling approach that leverages 

% We present an edge-based analog of the Szemer\'edi regularity lemma in which any graph can be summarized by one with a small number of edges so that random quotients of the large graph are close to those of their smaller summary.  This result yields an edge-based sampling approach for testing properties pertaining to hubs in large graphs.

% random graphs obtained by taking quotients of large graphs are related 

% Finally, we conclude with a discussion of random graph models in which the distributions on graphs of different sizes are related by equipartitions, .\\


\end{abstract} 
%\tableofcontents
\section{Introduction}\label{sec:intro}
Analyzing the structure of large graphs is a challenge that arises in numerous problem domains such as bioinformatics, network tomography, operations research, and the social sciences~\cite{newman_networks,network_tomography,social_network,JACKSON2002265,bassett2017network}. 
The graphs arising in these domains are, in general, weighted and directed. 
For example, in traffic networks vertices correspond to locations, and edge weights quantify the amount of traffic between two locations in each direction, see Figure~\ref{fig:gold_coast}. In functional brain connectome graphs, vertices correspond to regions of the brain, and edge weights quantify the connectivity between these regions in each direction, see Figure~\ref{fig:Celegans}.
A particularly important global structural aspect of such graphs is the presence of hub vertices and connections between them~\cite{OKELLY1998171,hub_and_spoke,power2013evidence,hubs_in_drugs,van2013network}.
For graphs with nonnegative edge weights and possible self-loops, a \emph{hub} is a vertex with a significant fraction of the total edge weight incident to it, such as a large city in a traffic network, a well-connected region of the brain, or a particularly interactive protein.  As a consequence of their high connectivity, hubs play a central role in the structure and function of large networks.  The structure of such large graphs has been studied using suitable notions of graph limits, most prominently graphons~\cite{lovasz2012large}; see Section~\ref{sec:related_work} for related work. However, these previous approaches are either limited to simple graphs, or summarize large graphs by their subgraph densities.  Subgraph densities fail to properly account for global structures such as hubs.  Indeed, a network can only contain a relatively small number of hubs, and therefore most subgraphs of a network do not contain any hubs.
%and connections between them, which are commonly encountered in real-world networks~\cite{OKELLY1998171,hub_and_spoke,power2013evidence,hubs_in_drugs,van2013network}. 
%For example, most subgraphs of a star, the prototypical example of a hub, do not contain any edges.

% \VC{Let's only use hubs, and formally say what these are.  Let's not say clusters.  Once we define a hub, we can motivate quotients are summaries of large graphs that preserve hub structure.}
% A graph may simple, or unweighted, or weighted, corresponding to the edge-weights belong to the sets $\{0,1\}$, $\mathbb{N}$, or $\RR_+$ respectively.  Unless specified, a graph is generally taken to be weighted.

Our central thesis is that hub structure in large graphs can be studied using \emph{quotients}.  Formally, a quotient of size $k$ of a graph on $n$ vertices is defined by an ordered partition $f_{k,n} : [n] \rightarrow [k]$, which partitions the vertex set $[n]$ into $k$ parts that are given by the fibers $f^{-1}_{k,n}(j)$ for $j \in [k]$.  In particular, for a graph on $n$ vertices with adjacency matrix $G \in \RR^{n \times n}_+$, its quotient by the ordered partition $f_{k,n}$ is given by the graph $\rho(f_{k,n}) G \in \RR^{k \times k}_+$ on $k$ vertices with edge weights:
\begin{equation}\label{eq:quotient_graph}
[\rho(f_{k,n})G]_{i,j} = \sum_{\substack{v\in f_{k,n}^{-1}(i)\\ u\in f_{k,n}^{-1}(j)}} G_{v,u}.
\end{equation}
Here we view $\rho(f_{k,n}) : \RR^{n \times n} \rightarrow \RR^{k \times k}$ as a `linear representation' of $f_{k,n}$.  In words, the vertices of the quotient $\rho(f_{k,n})G$ correspond to the parts in the partition defined by $f_{k,n}$, and the edge weight between parts $i,j\in[k]$ is given by the sum of all edge weights in $G$ between vertices in $f_{k,n}^{-1}(i)$ and $f_{k,n}^{-1}(j)$.  Thus, a quotient of $G$ represents the edge weights between each pair of parts in a vertex partition.

As an illustration, we give in Figures~\ref{fig:gold_coast}-\ref{fig:Celegans} examples of small quotients of two networks. The first, depicted in Figure~\ref{fig:gold_coast}, is a traffic network representing the fraction of trips starting and ending at each zone of the Gold Coast, Australia~\cite{traffic_networks}. 
The second, depicted in Figure~\ref{fig:Celegans}, is a brain connectome, the network of connections between neurons of an adult male \emph{C.~elegans} roundworm~\cite{cook2019whole,worms_brain}.
The structure of each quotient indicates some of the features of the underlying large network. As we will see in Section~\ref{sec:prop_test}, the presence of hubs in the original network leads to nonuniform distribution of edge weight in their random quotients.

% For example, in Figure~\ref{fig:gold_coast} we see that comparable levels of traffic arrive, depart, and stay within each random part of the partition, whereas in Figure~\ref{fig:philly} a significant fraction of traffic stays within each part.  

% For example, in Figure~\ref{fig:gold_coast} there is nontrivial traffic between nodes in the quotient, suggesting a more `global' pattern to the traffic, while in Figure[REF] much of the traffic between between nodes in the quotient is ... within zone.  some of the nodes in the quotient 

% Furthermore, the structure of the quotients indicates some of the features of the larger network from which it was obtained. For example, the blue node in the quotient of Figure~\ref{fig:gold_coast} has a relatively small fraction of traffic that stays within it but a relatively large fraction of traffic leaving or arriving at it, likely representing suburbs, whereas the orange vertex has a large fraction of its traffic staying inside of it. In contrast, in the Philadelphia network, we see significant traffic staying within each random part of its partition, signifying substantial self-loops in the original traffic network. 

% In Figure[RE] we see that more traffice 
%  \VC{Give a quick numerical illustration here?}

\begin{figure*}[h!]
    \centering
    \begin{subfigure}[t]{0.45\textwidth}
        \centering
        \includegraphics[width=\linewidth]{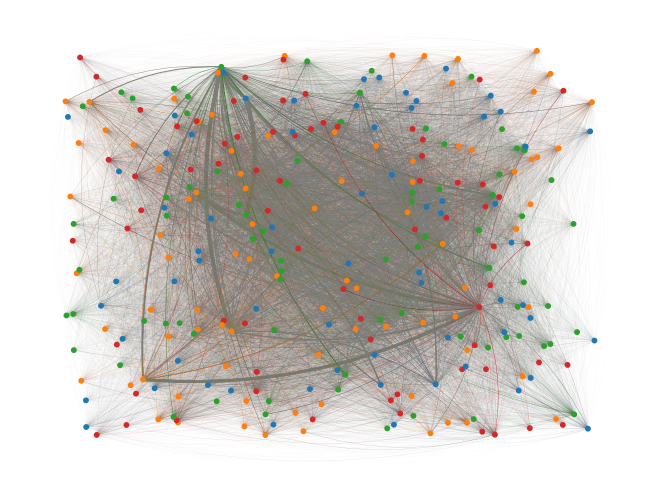}
    \end{subfigure}
    ~
    \begin{subfigure}[t]{0.45\textwidth}
        \centering
        \includegraphics[width=\linewidth]{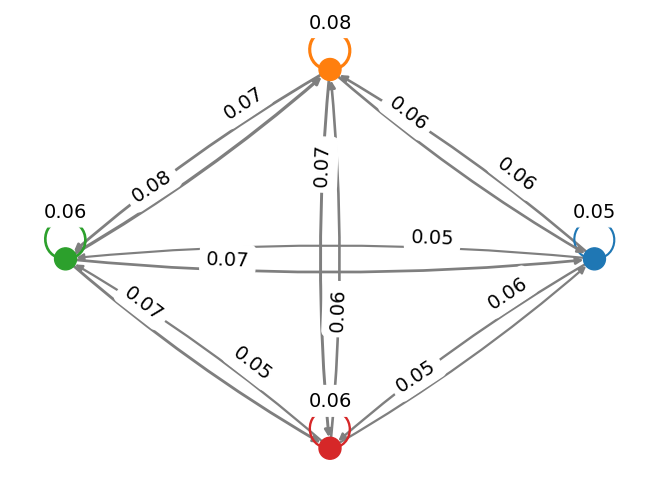}
    \end{subfigure}
    \caption{Traffic network for the Gold Coast, Australia, and a realization of a random quotient. The city is divided into 1068 zones, and the weight of an edge is the fraction of trips from a dataset with given start and end zones~\cite{traffic_networks}. Here we plot the first 300 of these zones and their interconnections, displaying the width of an edge proportionally to its weight. We color the vertices according to their image in the quotient, and color an edge connecting two vertices with the same image by the color of that image. The nonuniformity in the edge weights of the quotient indicates the presence of hubs that we can see in the plot of the entire network. 
    }\label{fig:gold_coast}
\end{figure*}

% \begin{figure*}[h!]
%     \centering
%     \begin{subfigure}[t]{0.45\textwidth}
%         \centering
%         \includegraphics[width=\linewidth]{figures/philly_full.png}
%     \end{subfigure}
%     ~
%     \begin{subfigure}[t]{0.45\textwidth}
%         \centering
%         \includegraphics[width=\linewidth]{figures/philly.png}
%     \end{subfigure}
%     \caption{Philadelphia traffic network and its random quotient. The city is divided into 1525 zones, and we plot the first 300 of these zones and the edges between them. In contrast to Figure~\ref{fig:gold_coast}, here much of the traffic stays within each part of the partition, indicating significant self-loops that we can see in the original network.}\label{fig:philly}
% \end{figure*}

\begin{figure*}[h!]
    \centering
    \begin{subfigure}[t]{0.45\textwidth}
        \centering
        \includegraphics[width=\linewidth]{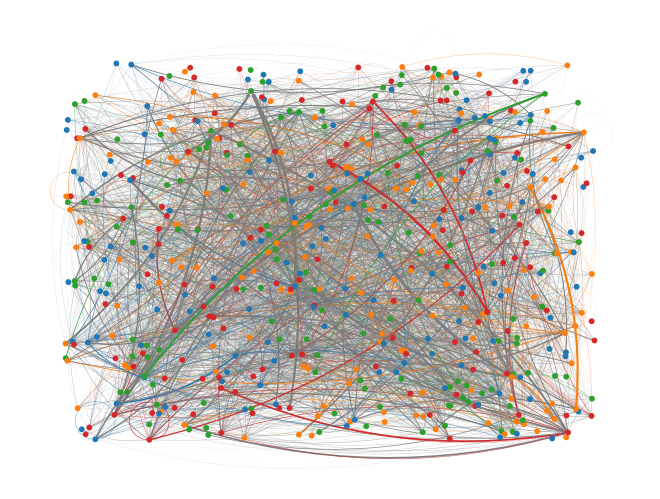}
    \end{subfigure}
    ~
    \begin{subfigure}[t]{0.45\textwidth}
        \centering
        \includegraphics[width=\linewidth]{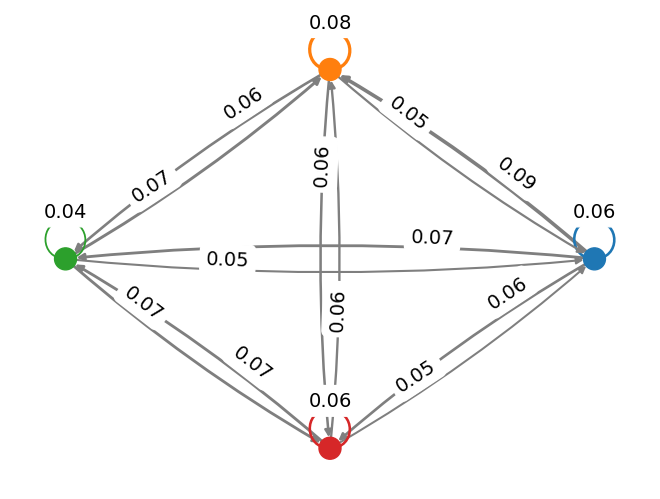}
    \end{subfigure}
    \caption{Network of chemical connections among the 575 neurons of adult male C.~elegans (a species of roundworm), and a realization of a  random quotient. Edges are directed from pre-synaptic to post-synaptic cells, and their weight is the total number of EM serial sections of connectivity~\cite{cook2019whole,worms_brain}. The nonuniform distribution of self-loops and the unequal in-degrees and out-degrees among the four vertices in the quotient again indicate the presence of hubs.}\label{fig:Celegans}
\end{figure*}

%\VC{In case it is not too much work, I wonder if you could include a figure from another application context (aside from transportation) either to replace one of these two figures or in addition to these two figures?  It is fine if the conceptual points being made are similar.  It would just be nice to illustrate that the scope of this paper is broad, not just restricted to transportation networks.}

In this paper, we define a new notion of convergence for sequences of growing-sized graphs based on convergence of their random quotients.  In analogy with graphons~\cite[\S1]{lovasz2012large}, we define suitable limit objects in terms of random measures for any quotient-convergent sequence of graphs.  Finally, we show that finite (random) graphs sampled from our limit objects are quotient-convergent to the limit almost surely.  The key concepts in our paper are dual to those in the literature on graphons, and we emphasize this duality throughout our development. 
In particular, we shall see that random quotients and random subgraphs, in terms of which graphon convergence is defined, are random linear maps which are adjoints of each other.
%Specifically, if $f_{n,k}\colon[k]\to[n]$ and $G \in \RR^{n \times n}_+$ is a graph, we observe that $\rho(f_{n,k})^\star G \in \mbb R^{k\times k}$ is the induced subgraph on the vertices $(f_{n,k}(1),\ldots,f_{n,k}(k))$. We recall from~\cite[\S11]{lovasz2012large} that a sequence of graphs $(G_n)$ converges to a graphon precisely when $(\rho(F$ 
Further, a dual analog of the cut metric plays a prominent role in our development, and we derive an edge-based analog of the Szemer\'edi regularity lemma as a consequence of our results.  We describe next these contributions in more detail.

% In a precise sense, the key notions in our paper are dual to those in the literature on graphons.  As an example, the dual cut metric plays a prominent role in our developments, and we derive a dual Szemer\'edi lemma as a consequence of results.  We describe next these contributions in more detail.

% Our development parallels three of the main threads in the literature on graphons.  In that body of work, a first key idea is a local notion of convergence of sequences of growing-sized graphs in terms of convergence of densities of small subgraphs.  A second central perspective is analytic in nature in which suitable limit objects called graphons are defined for sequences of growing-sized graphs, with the \emph{cut metric} playing a prominent role.  Finally, a third viewpoint considers finite (random) graphs sampled from a graphon and demonstrates that these graphs converge almost surely to the limiting graphon; this result is also closely related to Szmer\'edi's regularity lemma in which ...

% In this paper, we begin by defining a notion of convergence in which a sequence of growing-sized graphs is said to converge is their typical quotients converge.  Next, we 

\subsection{Our Contributions}
To set the stage, we define a random quotient of a graph on $n$ vertices in terms of a uniformly random map $F_{k,n}\colon [n]\to[k]$ sending each $i\in[n]$ independently to a uniformly-random element $F_{k,n}(i)\in [k]$. Equivalently, this is a uniformly random ordered partition of $[n]$ into $k$ parts, given by the sequence of (random) fibers $F_{k,n}^{-1}(j)$ for $j\in[k]$.  We denote by $\Delta^{n \times n}$ the set of $n \times n$ matrices with nonnegative entries that sum to one.

\begin{definition}\label{def:quotient_convergence}
A sequence of weighted graphs $(G_{n}\in\Delta^{n\times n})_{n=1}^{\infty}$ with edge weights that sum to one \emph{quotient-converges} if for each $k\in\NN$ the sequence of random graphs $(\rho(F_{k,n})G_{n}\in\Delta^{k\times k})_{n=1}^{\infty}$ converges weakly.  A general sequence of nonzero weighted graphs $(G_{n}\in \RR^{n\times n}_+)_n$ \emph{quotient-converges} if $(G_{n}/\mathbbm{1}^\top G_{n}\mathbbm{1} \in\Delta^{n\times n})_n$ quotient-converges.
\end{definition}

%The restriction to graphs with total edge weight equal to one may essentially be made without loss of generality.  Specifically, for any sequence of weighted graphs $(G_n)$ (with edge weights not necessarily summing to one), the sequence $(\rho(F_{1,n})G_{n} \in \RR_+)_n$ is equal to the sequence $(\mathbbm{1}^\top G_{n}\mathbbm{1})_n$ of edge weights.  As we require this sequence to converge (along with all the other typical quotients), we may normalize the total edge weight to equal one.  
Normalizing our graphs to have unit total edge weight does not change the structure of its hubs, which were defined above in terms of the \emph{fraction} of edge weight incident on them. Unless specified otherwise, we will assume throughout that the total edge weight of our graphs equals one.

% \VC{Are we going to make this standing assumption going forward?}
Denoting by $\rho(F_{n,k})^\star\colon\RR^{n\times n}\to\RR^{k\times k}$ the adjoint of $\rho(F_{n,k})$ with respect to the Frobenius inner product, note that $\rho(F_{n,k})^\star G_n$ is the induced subgraph on $k$ vertices sampled with replacement from the vertex set $[n]$ of $G_n$, in terms of which graphon convergence is defined~\cite[\S9]{lovasz2012large}. This is a first sense in which our theory is dual to that of graphons.
We remark that there is an alternative view of graphon limits via convergence of quotients~\cite[\S12]{lovasz2012large}, but these latter quotients differ from ours.  In particular, these latter quotients cannot detect hubs in a large graph, as we explain in Section~\ref{sec:related_work} below.

% To highlight the duality with graphons, note that $\rho(F_{n,k})^\star\colon\mbb R^{n\times n}\to\mbb R^{n\times n}$ samples $k$ vertices with replacement from a graph on $n$ vertices, and extracts the induced subgraph on them.  Hence, In the context of graphons, a sequence of graphs $(G_{n}\in\mbb R^{n\times n}_+)$ converges if $(\rho(F_{n_i,k})^\star G_{n_i}\in\mbb R^{k\times k})$ converges weakly for each $k\in\NN$, see~\cite[\S11]{lovasz2012large}. 

    %\item Note that relabelling the vertices of $G_n$ or appending any number of isolated vertices to it does not change the distribution of its quotients $\rho(F_{k,n})G_n$. 
    %Therefore, we seek to characterize the completion of $\bigsqcup_n\Delta^{n\times n}/\sim$ under quotient-convergence, where we identify graphs that differ by a relabelling of their vertices and isolated vertices.

\subsubsection{Combinatorial Aspects of Quotient Convergence (Section~\ref{sec:hom_convergence})}
As our first contribution, we analyze combinatorial aspects of quotient convergence in Section~\ref{sec:hom_convergence}. Whereas previous notions of graph convergence were based on convergence of local ``motifs'' in the form of homomorphism densities (appropriately normalized), we show that our quotient convergence is equivalent to convergence of homomorphism \emph{numbers}.  More precisely, for a  multigraph $H\in\NN^{k\times k}$ possibly containing parallel edges and self-loops, and a general weighted graph $G\in\RR^{n\times n}$ on $n$ vertices, the homomorphism number and homomorphism density of $H$ in $G$ are given by~\cite[Eq.~(5.3)]{lovasz2012large}
\begin{equation}\label{eq:hom_nums}
    \mathrm{hom}(H;G) = \sum_{f\colon[k]\to[n]}\prod_{i,j=1}^k G_{f(i),f(j)}^{H_{i,j}},\qquad t(H;G)=\frac{1}{n^k}\mathrm{hom}(H;G).
\end{equation}
The terminology comes from the fact that when $H$ and $G$ are simple graphs, then $\mathrm{hom}(H;G)$ is the number of graph homomorphisms from $H$ to $G$ and $t(H;G)$ counts the average number of times $H$ occurs in a uniformly random subgraph of $G$.  In particular, convergence of the homomorphism densities $t(H;\cdot)$ is used to define limits of dense graphs~\cite{LOVASZ2006933,convergent_seqs1}.

In analogy, we define the \emph{quotient density} $t_Q(H;G)$ of a multigraph $H\in\NN^{k\times k}$ in a general weighted graph $G\in\RR^{n\times n}$ as the average number of times $H$ occurs in a uniformly random quotient of $G$:
\begin{equation}\label{eq:quotient_densities}
    t_Q(H;G) = \frac{1}{k^n}\sum_{f\colon[n]\to[k]}\prod_{i,j=1}^k(\rho(f)G)_{i,j}^{H_{i,j}}.
\end{equation}
The role of homomorphism densities in graphon theory is played by quotient densities in our framework.  We prove that a sequence of graphs is quotient-convergent if all of their quotient densities converge.
\begin{theorem}\label{thm:hom_num_conv}
For a sequence of graphs $(G_{n} \in \Delta^{n \times n})_n$, the following are equivalent:
\begin{enumerate}
    \item The sequence $(G_{n})$ is quotient-convergent.
    \item For each multigraph $H\in\NN^{k\times k}$, the sequence of quotient densities $(t_Q(H;G_{n}))_n$ converges.
    \item For each multigraph $H\in\NN^{k\times k}$, the sequence of homomorphism numbers $(\mathrm{hom}(H;G_{n}))_n$ converges.
\end{enumerate}
\end{theorem}
    % Theorem~\ref{thm:hom_num_conv} can equivalently be stated in terms of injective homomorphism numbers and subgraph counts.
The proof of this result is given in Section~\ref{sec:hom_convergence}, and is based on expressing quotient densities in terms of (injective) homomorphism numbers. This expression involves a partial order on the collection of multigraphs induced by the quotient operation.

For simple graphs $(G_n\in\mbb S^n_{\mathrm{sim}})_n$, i.e., undirected graphs with weights $(G_n)_{i,j}\in\{0,1\}$ and with no self-loops so $\diag(G_n)=0$, we conclude that $(G_n/2|E(G_n)|)_n$ is quotient-convergent if and only if $\left(\frac{\mathrm{hom}(H;G_n)}{|E(G_n)|^{|E(H)|}}\right)_n$ converges for all simple graphs $H\in\mbb S^k_{\mathrm{sim}}$ (see Corollary~\ref{cor:dense_graphs1}). Thus, quotient-convergence of simple graphs is equivalent to convergence of a homomorphism-density-like quantity, but one that is differently normalized than usual.  To the best of our knowledge, such a normalization based on the number of edges has not appeared previously in the literature, as normalizations based on the numbers of vertices $k$ and $n$ are used (see Section~\ref{sec:related_work}).  As a consequence, we will see below in our discussion of limit objects (Corollary~\ref{cor:undirected_graphs}) that quotient-convergent sequences of both dense and degree-bounded sequences of simple graphs converge to the same (uninteresting) limit.  In fact, quotient limits of simple graphs can only detect stars, which are vertices on which a positive fraction of edges are incident.  More broadly, hub structure in weighted and directed graphs can consist not only of (weighted and directed) stars, but also of individual edges carrying a significant fraction of the edge weight in the graph. Such structure is modeled by our limit objects for quotient-convergent sequences, which we proceed to describe in detail.

\subsubsection{Limit Objects for Quotient Convergence (Section~\ref{sec:limits})}

Our second contribution consists of constructing analytic limit objects for quotient-convergent sequences together with a notion of distance between them that allows us to compare graphs of different sizes, akin to graphons and the cut metric for dense graph sequences.  Note that appending isolated vertices to a graph $G \in \Delta^{n \times n}$ or relabeling its vertices does not change the distribution of its quotients $\rho(F_{k,n})G_n$ for each $k\in\NN$. Therefore, the limit object we associate to $G$ must be the same as that associated to any other graph differing from it by extra isolated vertices and vertex relabeling.  We describe our limit objects in terms of certain random exchangeable measures on the square, which we define next.

\begin{definition}[Random exchangeable measures] A \emph{random exchangeable probability measure} on $[0,1]^2$ is a random measure\footnote{We endow the space of probability measures on $[0,1]^2$ with the weak topology and consider Borel random variables on it.} $\msf M$ satisfying $\msf M\overset{d}{=}\msf M\circ(\sigma,\sigma)$ for all (Lebesgue-)measure-preserving bijections $\sigma\colon[0,1]\to[0,1]$. 
We say that a sequence of random measures $(\msf M_n)$ converges if it converges weakly viewed as a sequence of random variables taking values in the space of probability measures on $[0,1]^2$ (see Section~\ref{sec:notation} for details).
\end{definition}

Any finite graph $G\in\Delta^{n\times n}$ can be associated with the random exchangeable measure:
\begin{equation}\label{eq:step_measure}
    \msf M_G = \sum_{i,j=1}^nG_{i,j}\delta_{(T_i,T_j)},\quad \textrm{where } T_1,\ldots,T_n\overset{iid}{\sim}\mathrm{Unif}([0,1]).
\end{equation}
Note that relabeling the vertices of $G$ or appending isolated vertices to it yields the same random measure in this way.  More generally, Kallenberg showed that all random exchangeable measures on $[0,1]^2$ have the form~\cite{kallenberg1990exchangeable}:
\begin{equation}\label{eq:kallenberg_char}
    \msf M = \sum_{i,j\in\NN} E_{i,j}\delta_{(T_i,T_j)} + \sum_{i\in\NN}[\sigma_i(\delta_{T_i}\otimes\lambda) + \varsigma_i(\lambda\otimes \delta_{T_i})] + \theta\lambda^2+\vartheta\lambda_D,
\end{equation}
where $(i)$ $\lambda$ is the uniform measure on $[0,1]$,  $\lambda^2$ is the uniform measure on $[0,1]^2$, and $\lambda_D$ is the uniform measure on the diagonal $\{(x,x):x\in[0,1]\}$; $(ii)$ $T_1,T_2,\ldots\overset{iid}{\sim}\mathrm{Unif}([0,1])$; and $(iii)$ $E_{i,j},\sigma_i,\varsigma_i,\theta,\vartheta\geq0$ may potentially be random but must be independent of the $(T_i)$ and must satisfy $\sum_{i,j}E_{i,j}+\sum_i(\sigma_i+\varsigma_i)+\theta+\vartheta = 1$ almost surely.  Moreover, the collection of extremal random exchangeable measures, i.e., those that are not mixtures of other such random measures, are precisely the ones for which the coefficients $E,\sigma,\varsigma,\theta,\vartheta$ in~\eqref{eq:kallenberg_char} can be taken to be deterministic. 
We call such extremal random measures \emph{grapheurs}, and show that these are precisely limits of quotient-convergent graph sequences.
%\VC{In the above, $b, b'$ is slightly confusing notation.  I wonder if we could say $\iota$ and $\omega$, which are Greek for i and o, meaning incoming and outgoing.  The following is not essential, but if you feel upto it -- I would try to change $\theta$ and $\vartheta$ to something similar, for example to $\beta$ and $\sigma$ to denote `blip' and `sand/dust'; similarly, something like $E$ instead of $A$ to denote edges would also be appropriate.}

\begin{definition}\label{def:M0}
A \emph{grapheur} (graph edge measure) is an extremal random exchangeable measure, i.e., one of the form~\eqref{eq:kallenberg_char} with deterministic coefficients $(E,\sigma,\varsigma,\theta,\vartheta)$. 
We denote the collection of grapheurs by
\begin{equation}\label{eq:grapheur_space}
    \mc M = \left\{\msf M \textrm{ of the form~\eqref{eq:kallenberg_char}}: E\in \RR^{\NN\times \NN}_{\geq 0},~ \sigma,\varsigma\in \RR^{\NN}_{\geq0},~ \theta,\vartheta\in\RR_{\geq0} \textrm{ s.t. } \sum\nolimits_{i,j}E_{i,j} + \sum\nolimits_i(\sigma_i+\varsigma_i)+\theta+\vartheta=1\right\}.
\end{equation}
The space of grapheurs inherits the weak topology from the ambient space of random exchangeable measures.
\end{definition}

The coefficients $(E,\sigma,\varsigma,\theta,\vartheta)$ describing a fixed grapheur $\msf M\in\mc M$ are essentially unique. Specifically, note that for any permutation $\pi\colon\NN\to\NN$, the coefficients $(E,\sigma,\varsigma,\theta,\vartheta)$ and $(\pi E\pi^\top, \pi \sigma, \pi \varsigma, \theta,\vartheta)$ define the same random exchangeable distribution via~\eqref{eq:kallenberg_char}. Further, adding zero rows to $E$ and corresponding zero entries to $\sigma$, or zero columns to $E$ and corresponding zero entries to $\varsigma$, does not change the resulting random measure. Kallenberg showed in~\cite[Prop.~3]{kallenberg1990exchangeable} that these are the only possible ambiguities in the coefficients.
In particular, note that $\msf M_G\in\mc M$ is a grapheur for any finite graph $G$, and that $\msf M_G=\msf M_{G'}$ for two such graphs $G,G'$ if and only if they differ by isolated vertices and vertex relabeling.

In analogy with the cut metric between graphons, we endow $\mc M$ with the following metric between grapheurs that metrizes their convergence:
    % \begin{equation}\label{eq:W1_metric}
    %     \delta_1(\msf M_1,\msf M_2) = \inf_{\substack{\textrm{random } (\msf M_1',\msf M_2')\\ \msf M_i'\overset{d}{=}\msf M_i \textrm{ for } i=1,2} } \mbb E W_1(\msf M_1',\msf M_2'),
    % \end{equation}
% \VC{I have several comments here:
% \begin{itemize}
%     \item Do we want to say how this is related to the cut metric?  Should we give this metric a name?  Something like the ``Wasserstein cut distance'' or ``Wasserstein cut metric''?  I say `Wasserstein' as the infimum seems to be over couplings.  If you agree with this perspective, we should change notation to $W_\delta$ instead of $\delta$.
%     \item In the inf below, should we replace ``random $(\msf M_1',\msf M_2')$'' with ``$\msf M_1', \msf M_2' \in \mc M$''?
%     \item Finally, for expository reasons, it would be good to clarify what equality in distribution means.  Related to this, is there a connection to couplings?  May be useful to comment on this.
% \end{itemize}}
\begin{equation}\label{eq:dual_cut_metric}
    W_{\square}(\msf M_1,\msf M_2) = \inf_{\substack{\textrm{random } (\msf M_1',\msf M_2')\\ \msf M_i'\overset{d}{=}\msf M_i \textrm{ for } i=1,2} }\ \sup_{\substack{S,T\subseteq[0,1]\\ \textrm{intervals}}}\mbb E|\msf M_1'(S\times T)-\msf M_2'(S\times T)|,
\end{equation}
where the infimum is over couplings $(\msf M_1',\msf M_2')$ of the two random measures $\msf M_1$ and $\msf M_2$. Here $\msf M_i'\overset{d}{=}\msf M_i$ means that the distribution of $\msf M_i'$ is equal to that of $\msf M_i$, see Section~\ref{sec:notation} for more detail.
We show that grapheurs are the limits of quotient-convergent sequences of graphs in the following precise sense.
\begin{theorem}\label{thm:limits_formal}
A sequence of grapheurs $(\msf M_n)\subseteq\mc M$ converges weakly to a grapheur $\msf M\in\mc M$ if and only if $\lim_nW_{\square}(\msf M_n,\msf M)=0$. Furthermore, the following statements hold.
\begin{enumerate}
    \item A sequence of graphs $(G_n\in\Delta^{n\times n})$ is quotient-convergent if and only if there is a grapheur $\msf M\in\mc M$ satisfying $\lim_nW_{\square}(\msf M,\msf M_{G_n})=0$.

    \item The space of grapheurs $\mc M$ is compact. In particular, any sequence of graphs $(G_n\in\Delta^{n\times n})$ has a quotient-convergent subsequence.

    \item For every grapheur $\msf M\in\mc M$ there is a sequence $(G_n\in\Delta^{n\times n})$ satisfying $\lim_nW_{\square}(\msf M,\msf M_{G_n})=0$.
\end{enumerate}
\end{theorem}
In view of Theorem~\ref{thm:limits_formal}, we say that a sequence of graphs $(G_n)$ converges to a grapheur $\msf M\in\mc M$ if $\lim_n\msf M_{G_n}=\msf M$ in $\mc M$. 
Theorem~\ref{thm:limits_formal} states that any quotient-convergent sequence of graphs converges to a grapheur, and conversely each grapheur is the limit of a quotient-convergent sequence of graphs.
Moreover, this theorem shows that the space of grapheurs endowed with the weak topology is compact, and that this compact topology is metrized by our $W_{\square}$ metric~\eqref{eq:dual_cut_metric}. 
%\VC{This preceding phrase is confusing to me for two reasons.  First, superficially it seems you can only conclude sequential compactness from the second part of the above result, right?  Is it then evident that the space is in fact compact?  (Again, I'm thinking from the perspective of a combinatorialist reading this.)  Second, should we say the metric space $(\mc M, W_\square)$ is compact rather than just ``the space of grapheurs''?}

To prove Theorem~\ref{thm:limits_formal}, we relate random quotients of finite graphs to their associated grapheurs~\eqref{eq:step_measure}. 
For each grapheur $\msf M\in\mc M$ and $k\in\NN$, define the random graph $\msf G_k[\msf M]\in\Delta^{k\times k}$ whose entries are given by:
\begin{equation} \label{eq:graph_from_measure}
    (\msf G_k[\msf M])_{i,j} = \msf M(I_i^{(n)}\times I_j^{(n)})
\end{equation} 
for $i,j\in[n]$, where $I_i^{(n)}=[(i-1)/n,i/n)$ for $i\in[n]$.\footnote{Note that $\sum_{i,j}(\msf G_k[\msf M])_{i,j}=\msf M([0,1)^2)=1$ almost surely by~\eqref{eq:kallenberg_char}.}
In words, we partition the square $[0,1]^2$ into a $k\times k$ uniform grid, and let the entries of $\msf G_k[\msf M]$ be the (random) measure under $\msf M$ of each square in this grid. We show in Section~\ref{sec:characterizing_lims} that $\rho(F_{k,n})G\overset{d}{=}\msf G_k[\msf M_{G}]$ for any finite graph $G\in\Delta^{n\times n}$, and that a sequence of graphs $(G_n)$ converges to the grapheur $\msf M$ precisely when the random quotients $(\rho(F_{k,n})G_n)_k$ converge weakly to $\msf G_k[\msf M]$ for each $k\in\NN$. Similarly, we can extend quotient densities to grapheurs by setting
\begin{equation}\label{eq:grapheur_density}
    t_Q(H;\msf M) = \mbb E\prod_{i,j=1}^k(\msf G_k[\msf M])_{i,j}^{H_{i,j}} = \int\prod_{i,j=1}^k\mu(I_i^{(k)}\times I_j^{(k)})^{H_{i,j}}\, d\msf M(\mu),
\end{equation}
for any multigraph $H\in\mbb N^{k\times k}$, analogously to the extension of homomorphism densities to graphons~\cite[\S7.2]{lovasz2012large}. In terms of quotient densities, the sequence $(G_n)$ converges to the grapheur $\msf M$ if and only if $t_Q(H;G_n)\to t_Q(H;\msf M)$ for all multigraphs $H$, see Section~\ref{sec:characterizing_lims}. 
%\VC{Give forward reference to where this result is proved?}

%\TODO{Define quotients of grapheurs and state that they give the limits of quotients of convergent graphs. Moreover, quotient densities can be defined for grapheurs and they give limits of the densities of graphs in the sequence. quotient densities vs.\ homomorphism densities}

Having characterized our limit objects, we interpret in Section~\ref{sec:structure_of_lims} each of the parameters $E, \sigma, \varsigma, \theta, \vartheta$ describing a grapheur~\eqref{eq:kallenberg_char} in terms of certain features of the large graphs converging to it. Informally, the array $E\in\RR_{\geq 0}^{\NN\times \NN}$ specifies the connections between a discrete collection of hubs, with each entry corresponding to an edge with an asymptotically positive fraction of the total edge weight. The arrays $\sigma,\varsigma\in\RR_{\geq 0}^{\NN}$ specify directed stars centered at hubs, with each entry corresponding to the fraction of out-degree or in-degree of a hub from all the edges incident to it with asymptotically vanishing edge weights. Lastly, the parameters $\theta$ and $\vartheta$ specify the remaining edge weight that comes from all edges between vertices with asymptotically vanishing degrees and their self-loops. In particular, we say that a grapheur $\msf M$ does not have hubs if $E=0$ and $\sigma=\varsigma=0$, and we say that a sequence $(G_n)$ of graphs does not contain hubs asymptotically if it quotient-converges to such a grapheur.

Finally, we expand on the duality between graphons and grapheurs. We have seen previously that the sampling operators used to define each notion of graph limit, namely, those extracting random subgraphs and random quotients, respectively, are duals of each other. We further show in Section~\ref{sec:dual_graphons} that the space of graphons endowed with the cut metric is dual to the space $\mc M$ endowed with the $W_{\square}$ metric~\eqref{eq:dual_cut_metric} as follows. 
%For the following informal theorem, we let $\mc M_{\mathrm{sym}}$ be the space of limits of undirected graphs (characterized in Section~\ref{sec:structure_of_lims}), and let $\mc W$ be the space of graphons endowed with the cut metric. See Section~\ref{sec:dual_graphons} for precise definitions of both spaces.
%Specifically, suppose $\mc W$ is the space of bounded graphons (see Section[REF] for a precise definition) and suppose $\mc M_{\mathrm{sym}}$ is the space of symmetric random exchangeable measures in $\mc M$, i.e., elements of $\mc M$ with symmetric $A$ and $b = b'$.  Then we have the following result.
\begin{theorem}[Informal; see Theorem~\ref{thm:graphon_duality}]\label{thm:graphon_duality_intro}
    There is a random pairing $\langle \msf M,W\rangle$ associating a scalar random variable to each grapheur $\msf M$ and graphon $W$ satisfying the following properties. We have $\lim_n\msf M_n=\msf M$ in the $W_{\square}$ metric if and only if $\langle \msf M_n, W_G\rangle\to\langle\msf M,W_G\rangle$ weakly for each step graphon $W_G$ associated to a finite simple graph $G$.
    We also have $\lim_nW_n\to W$ in the cut metric if and only if $\langle \msf M_G,W_n\rangle\to\langle\msf M_G,W\rangle$ weakly for each grapheur $\msf M_G$ associated to a finite undirected graph $G$ with unit total edge weight and no self-loops.
\end{theorem}

\subsubsection{Edge Sampling and Property Testing (Section~\ref{sec:sampling})}

%\VC{Mention here or somewhere else in this section that subscript is for number of vertices and superscript is for number of edges.}

As our next contribution, we apply our framework to test properties of large graphs based on small summaries of them.  
Specifically, we focus on testing properties given by graph parameters that are continuous with respect to quotient convergence, including quotient densities, homomorphism numbers, clustering and centrality coefficients, see Section~\ref{sec:prop_test}.
While it would be natural to test such properties using random quotients, forming these requires access to the entire large graph, which limits their utility in practice. 
Motivated by this challenge, we develop an edge-based sampling perspective on quotient-convergence in Section~\ref{sec:sampling_edges}.   
Specifically, we sample edges from a graph proportionally to their edge weights to form random graphs $\msf G^{(n)}$ with $n$ edges, and extend this sampling procedure to any grapheur $\msf M \in \mc M$.  We detail the general sampling procedure in Section~\ref{sec:sampling_edges}. 
%When $\msf M=\msf M_G$ is the random measure corresponding to a finite graph $G$, our sampling procedure amounts to sampling the edges of $G$ uniformly with replacement and randomly labelling the resulting subgraph.  \VC{More generally, ....} We detail the general sampling procedure in Section~\ref{sec:sampling}.

We prove the following convergence results for the sequence of random graphs $(\msf G^{(n)})$ sampled from a grapheur $\msf M$ back to $\msf M$, which are edge-based analogs of the second sampling lemma and Szemer\'edi's regularity lemma for graphons~\cite[Lemma~10.16, Lemma~9.15]{lovasz2012large}. Here and throughout the paper, a superscript $G^{(n)}$ denotes a graph on $n$ edges (and hence at most $2n$ vertices) while a subscript $G_n$ denotes a graph on $n$ vertices.

\begin{theorem}\label{thm:sampling_intro}
Fix any grapheur $\msf M\in \mc M$.
\begin{enumerate}[align=left, font=\emph]
        \item[(Edge sampling lemma)] For each $n\in\NN$, let $\msf G^{(n)}$ be the random graph containing $n$ edges sampled from $\msf M$ as in Section~\ref{sec:sampling}. With probability at least $1-e^{-2\epsilon^2}$, we have that $W_{\square}(\msf M_{\msf G^{(n)}}, \msf M)\leq \frac{174+\epsilon}{\sqrt{n}}$. 

        \item[(Edge-based Szemer\'edi regularity)] For any $\epsilon>0$, there exists a graph $G^{(k(\epsilon))}$ on $k(\epsilon)=\lceil (174/\epsilon)^2\rceil$ edges satisfying $W_{\square}(\msf M,\msf M_{G^{(k(\epsilon))}})\leq \epsilon$.
    \end{enumerate}
\end{theorem}
%Note that the usual Szemer\'edi regularity lemma approximates a graphon by a step graphon obtained from a finite graph on $n$ vertices to error $O(1/\sqrt{\log n})$ in cut metric, which decays much slower than our $O(1/\sqrt{n})$ rate in the $W_\square$-metric.  

Theorem~\ref{thm:sampling_intro} allows us to test properties of large graphs by sampling a few of their edges at random. We formally define the graph parameters that can be tested using our framework in Section~\ref{sec:prop_test}. These consist of graph parameters that are continuous with respect to quotient convergence, and include all polynomial graph parameters that do not change when appending isolated vertices to their inputs (Proposition~\ref{prop:polys_are_testable}). 
In particular, the graph parameter $f(G)=\|G\mathbbm{1}\|_2^2+\|G^\top\mathbbm{1}\|_2^2$ is quotient-testable in our framework, and can be used to test for the presence of hubs in a growing sequence of graphs. Specifically, if $(G_n)$ is a quotient-convergent sequence, we show in Example~\ref{ex:hub_test} that $f(G_n)\to 0$ precisely when $(G_n)$ does not have hubs asymptotically, i.e., it converges to a limiting grapheur~\eqref{eq:kallenberg_char} with $E=0$ and $\sigma=\varsigma=0$. We then describe an edge-based sampling procedure for testing this graph parameter. In contrast, this parameter cannot be tested in the graphons framework; see Section~\ref{sec:prop_test}.

\subsubsection{Equipartition-Consistent Random Graph Models (Section~\ref{sec:eqp_models})}\label{sec:eqp_intro}

As our final contribution, we observe that random quotients $(\msf G_k[\msf M])_k$ of any grapheur $\msf M$ are related to each other by equipartitions of their vertex sets. 
Specifically, we observe that $\rho(d_{k,nk})\msf G_{nk}[\msf M]\overset{d}{=}\msf G_k[\msf M]$ for any equipartition $d_{k,nk}\colon[nk]\to[k]$, i.e., a map with equal-sized fibers. 
We call such sequences of random graphs \emph{equipartition-consistent random graph models}, which play a role in our framework analogous to that played by consistent random graph models in graphon theory~\cite{lovasz2012random}. 
\begin{definition}[Equipartition-consistent models]\label{def:eqp_models}
    A sequence of probability measures $(\mu_k\in\mc P(\Delta^{k\times k}))_{k\in\NN}$ is an \emph{equipartition-consistent random graph model} if $\rho(d_{k,nk})\msf G_{nk}\overset{d}{=}\msf G_k$ whenever $\msf G_{nk}\sim\mu_{nk}$ and $\msf G_k\sim\mu_k$ and for any equipartition $d_{k,nk}\colon[nk]\to[n]$. 
\end{definition}
%Since any permutation is an equipartition (all its fibers are singletons), an equipartition-consistent model is in particular exchangeable. 
Further, we show that any equipartition-consistent random graph model is derived from a mixture of grapheurs. 
\begin{theorem}\label{thm:eqp_models}
    For any equipartition-consistent random graph model $(\mu_k\in\mc P(\Delta^{k\times k}))_k$, there is a unique random exchangeable measure $\msf M$ on $[0,1]^2$ satisfying $\mu_k=\mathrm{Law}(\msf G_k[\msf M])$ for all $k\in\NN$. 
    % \VC{Do you mean to say a unique grapheur or a unique random grapheur here?  The discussion below this theorem statement suggests `random grapheur', but it would be good to clarify in this theorem statement.} \TODO{I mean a unique random exchangeable measure, which is a (not necessarily unique) mixture of grapheurs}
\end{theorem}
Recall that grapheurs are extremal random exchangeable measures by Definition~\ref{def:M0}. By Choquet's theorem, any random exchangeable measure is thus a mixture of grapheurs.  Finally, we show in Proposition~\ref{prop:eqp_as_convergence} that the sequence $(\msf G_k[\msf M])$ of quotients of a grapheur $\msf M$ converges back to $\msf M$ in a certain sense.

\subsection{Related Work}\label{sec:related_work}
There are several well-studied notions of graph limits and random graph models in the literature. We review this body of work and discuss similarities and differences with the present paper.

\paragraph{Dense and degree-bounded graph limits.} The earliest two types of graph limits pertain to sequences of simple graphs that are dense, whose limits are given by graphons~\cite{LOVASZ2006933,convergent_seqs1}, and degree-bounded, whose limits are given by graphings~\cite{benjamini_schramm,aldous2004objective,elek2007note}. We refer the reader to~\cite{lovasz2012large} for a survey of both limits. Both of these notions of graph limits are incomparable to ours, as both dense and degree-bounded sequences of simple graphs converge to the same limit in our framework, see Corollary~\ref{cor:undirected_graphs}. Further, dense graph limits can be characterized in terms of quotients~\cite[\S12]{lovasz2012large}, called ``right convergence'' of graphons.  However, these quotients are differently normalized compared to ours, and it is the set of quotients that converges in the graphon framework rather than their distribution as in our framework. 
%Neither notion of graph limit is able to capture stars, the prototypical example of hubs, as these are neither dense nor degree-bounded. 
Despite these distinctions, we highlight parallels between our development and the theory of graphons throughout our paper; in particular, grapheurs can be viewed as dual to graphons in a precise sense, see Section~\ref{sec:dual_graphons}.
%Nevertheless, our theory was inspired in part by the theory of graphons and can be viewed as dual to it in a precise sense, see Section~\ref{sec:dual_graphons}. \VC{This last sentence feels a bit weak.  How about just saying ``Despite these distinctions, we highlight parallels between our development and the theor of graphons; in particular, grapheurs can be viewed as dual to graphons in a precise sense, see Section~\ref{sec:dual_graphons}.''}

\paragraph{Edge-exchangeable random graphs.} A class of random graph models closely related to the ones we obtain by sampling edges from a grapheur is called edge-exchangeable~\cite{crane2016edge,Crane_Dempsey_2019,NIPS2016_1a0a283b, exchangeable_traits, janson2018edge}. These random graphs are obtained by adding random edges to a fixed infinite vertex set (often with repetition, yielding multigraphs), and can yield both sparse and dense random graphs~\cite{janson2018edge}. 
As explained in~\cite[Rmk.~4.4]{janson2010graphons}, every edge-exchangeable multigraph model can be obtained by sampling edges from a random exchangeable measure on $[0,1]^2$. 
In fact, the random graphs we obtain by sampling edges in Section~\ref{sec:sampling} are precisely these edge-exchangeable multigraphs, normalized to have unit edge weight, see Remark~\ref{rmk:edge_exch}.
From this perspective, convergence of grapheurs can be interpreted as convergence of the edge-exchangeable multigraph models associated to them, an interpretation we make precise in Proposition~\ref{prop:edge_sample_view}.
%Edge-exchangeable multigraphs are closely related to the random graphs we obtain in Section~\ref{sec:sampling} by sampling edges from a grapheur. see Remark~\ref{rmk:edge_exch}.
%The random graphs we obtain by sampling edges from a grapheur are edge-exchangeable multigraphs, normalized to have unit total edge weight. Conversely every edge-exchangeable random graph can be obtained by sampling edges from a mixture of grapheurs (in our terminology) as explained in~\cite[Rmk.~4.4]{janson2010graphons}, see also Remark~\ref{rmk:edge_exch}. 
%In fact, quotient convergence can be characterized in terms of a certain notion of convergence of the edge-exchangeable models sampled from a graph sequence, as we show in Proposition~\ref{prop:edge_sample_view}. 
%\VC{The stuff below this comment is fine.  However, the sentences before this comment are a bit confusing, alternating rapidly between edge-exchangeability and your framework.  Is there a clearer way to states things?}\TODO{I modified the above text accordingly}
Nevertheless, to our knowledge this notion of convergence of edge-exchangeable models has not been previously studied in the literature. 
In particular, the only metric between edge-exchangeable models we are aware of is an $\ell_1$ distance between edge probabilities~\cite[Eq.~(8)]{Crane_Dempsey_2019}. Convergence in this metric is too restrictive for our purposes, see Proposition~\ref{prop:tv_convergence}.
Further, our development of convergence via random quotients, which are not edge-exchangeable, allows us to take limits of general weighted graphs, to relate it to convergence of homomorphism densities and other graph parameters, and to test their values. 
%\VC{New oganization reads well to me!}

\paragraph{Graphex processes and limits.}
The closest line of work to ours is the study of graphex processes and their associated graph limits and random graph models~\cite{caron_fox,veitch2015class,jmlr_graphon_proc,borgs2020identifiability,borgs2019sampling,multigraphex,JANSON2022103549}.
Graphexes form a subset of random exchangeable measures on $\RR^2_{\geq0}$, and admit a characterization due to Kallenberg similarly to our grapheurs~\cite{kallenberg1990exchangeable}. 
Graphexes give rise to random graph models obtained by sampling edges from them~\cite{veitch2015class,borgs2019sampling,veitch2019sampling}, analogously to our edge sampling procedure in Section~\ref{sec:sampling_edges}. 
There are also several notions of limits for simple graphs and multigraphs, whereby such graphs converge to a limiting graphex~\cite{jmlr_graphon_proc,multigraphex}. 
These notions of limits are conjectured to be inequivalent~\cite{JANSON2022103549}, and the space of graphexes with either topology is not compact. 
Some of these limits can be characterized by sampling random subgraphs (whose number of vertices is also random)~\cite{borgs2019sampling}, while others can be characterized by convergence of differently-normalized homomorphism numbers~\cite[\S2.5]{jmlr_graphon_proc} (the authors of~\cite{jmlr_graphon_proc} normalize $\mathrm{hom}(H;G)$ by $(2|E(G)|)^{|V(H)|/2}$, whereas we normalize by $(2|E(G)|)^{|E(H)|}$). 
Graphex-based graph limits are again incomparable to ours. On the one hand, for a sequence of simple graphs our grapheur-based limits can only distinguish stars from a uniform component (Corollary~\ref{cor:undirected_graphs}), whereas graphex-based limits can capture more general limiting structures (see~\cite[\S9]{JANSON2022103549} or~\cite{multigraphex}). On the other hand, grapheur-based graph limits are defined for general sequences of nonnegatively-weighted graphs, for which graphex-based limits have not been defined, allowing us to test properties pertaining to the distribution of edge weight in the graph (see Section~\ref{sec:prop_test}). 
In fact, grapheurs are particularly well-suited for property testing, since the space of grapheurs we study here is compact, in contrast to the space of graphexes; to our knowledge, there is no characterization of properties that can be tested using graphex-based limits.  Another consequence of the compactness of the space of grapheurs is that our notion of graph limits can be characterized in several equivalent ways: by convergence of random quotients, convergence of random edge-sampled graphs, and convergence of homomorphism numbers and related graph parameters. 
%\VC{Should we add a sentence here saying something like ``Similar characterizations are unavailable for graphexes due to the lack of compactness.''}\TODO{I'm wary of saying this because the ``compactness iff equivalence'' seems to only be an intuition, and various compact subspaces of graphexes have been previously studied} \VC{Sounds good, let's not make such a statement.}

\subsection{Notation and Preliminaries}\label{sec:notation}
% \VC{It would be good to have a notation / terminology section here.  Some thoughts based on going through the paper:
% \begin{itemize}
%     \item Distinguish between simple vs. weighted vs. unweighted again?

%     \item Discuss weak topology

%     \item Discuss distinction between equality of random variables in distribution and weakly.

%     \item Subscript for vertices, superscripts for edges
% \end{itemize}}

%\VC{Introduce notation somewhere for $\msf S_\infty$ and $\msf S_n$}

We denote by $\NN=\{0,1,2,\ldots\}$ the set of nonnegative integers. If $n\in\NN$, then $[n]=\{1,\ldots,n\}$. If $a\leq b$ are real numbers then $[a,b]=\{x\in\RR:a\leq x\leq b\}$ and $[a,b)=[a,b]\setminus\{b\}$. 
We denote by $\lambda$ the uniform measure on $[0,1]$, by $\lambda^2$ the uniform measure on $[0,1]^2$, and by $\lambda_D$ the uniform measure on the diagonal $D=\{(x,x):x\in[0,1]\}$. We denote by $\mathbbm{1}_n$ the vector of all-1's of length $n$, and we omit the subscript $n$ when the size (possibly infinite) of the all-1's vector is clear from context. We denote by $I_n$ the $n\times n$ identity matrix and by $\mathbbm{1}_{n\times n}=\mathbbm{1}_n\mathbbm{1}_n^\top$ the $n\times n$ all-1's matrix.

We identify a (directed, weighted) graph with its adjacency matrix $G\in\RR^{n\times n}_{\geq0}$, where $G_{i,j}$ is the weight of the edge $i\to j$. The collection of all adjacency matrices of graphs with unit total edge weight is denoted by $\Delta^{n\times n}$. In this paper, all edge weights are nonnegative and sum to one unless stated otherwise. 
A graph is undirected if its adjacency matrix is symmetric. We denote the space of $n\times n$ symmetric matrices by $\mbb S^n$, and the set of such matrices with nonnegative entries, which is the set of undirected graphs, by $\mbb S^n_{\geq0}$. An undirected graph $G$ is simple if  $\diag(G)=0$ and $G_{i,j}\in\{0,1\}$ for all $i,j$, and we denote the set of simple graphs by $\mbb S^n_{\mathrm{sim}}$. 
A graph $H$ is called a \emph{multigraph} if it has integer edge weights $H\in\NN^{k\times k}$, where we view $H_{i,j}$ as representing the number of parallel edges from $i$ to $j$. Note that a simple graph is a multigraph. 
We denote by $|E(G)|=\mathbbm{1}^\top G\mathbbm{1}/2$ the number of edges in a \emph{simple} graph $G$. If $H$ is a multigraph, we write $\mathbbm{1}^\top H\mathbbm{1}=\|H\|_1$ to be the number of (directed) edges in $H$.
We assume that multigraphs have no isolated vertices, defined as vertices with no self-loops and no edges incident on them. In other words, if $H\in\NN^{k\times k}$ is a multigraph, then for every $i\in[k]$ there is a $j\in[k]$ such that $H_{i,j}>0$ or $H_{j,i}>0$.
If $G\in\RR^{n\times n}_{\geq0}$ and $H\in\NN^{k\times k}$ with $k\leq n$, we denote by $G^H=\prod_{i,j=1}^kG_{i,j}^{H_{i,j}}$ the monomial over $\RR[G_{i,j}]$ defined by $H$. In this notation, we have $t_Q(H;G)=\mbb E(\rho(F_{k,n})G)^H$ for a multigraph $H\in\NN^{k\times k}$ and graph $G\in\Delta^{n\times n}$.

We denote by $\msf S_n$ the group of permutations of $[n]$, and by $\msf S_{\infty}$ the group of permutations of $\NN$ that fixes all but finitely-many integers. If $\pi\in\msf S_n$ and $G\in\RR^{n\times n}_{\geq0}$ is a graph, we denote by $\pi G\in\RR^{n\times n}_{\geq0}$ the permuted graph with entries $(\pi G)_{i,j}=G_{\pi^{-1}(i), \pi^{-1}(j)}$. This corresponds to relabeling the vertices of $G$ according to $\pi$. 
%Note that if we identify $\pi$ with its corresponding $n\times n$ permutation matrix, then $\pi G=\pi G\pi^\top$ corresponds to conjugating $G$ by this permutation matrix. \VC{This preceding sentence is potentially confusing due to $\pi G = \pi G \pi^\top$.  How about just deleting it?  You already say what is needed in the sentence ``This corresponds to relabeling ...'' a couple of lines above.} 
Similarly, if $x\in\RR^n$ is a vector we denote by $\pi x\in\RR^n$ the permuted vector with entries $(\pi x)_i=x_{\pi^{-1}(i)}$. If $G\in\RR^{\NN\times \NN}_{\geq9}$ and $x\in\RR^{\NN}$ are infinite matrices and vectors, and $\pi\in\msf S_{\infty}$, we define $\pi G$ and $\pi x$ analogously.

If $\mu\in\mc P(S)$ is a probability measure on a compact Polish space $S$, we write $X\sim\mu$ and $\mathrm{Law}(X)=\mu$ to denote a random variable distributed according to $\mu$. 
If $X,Y$ are two random variables on $S$, we write $X\overset{d}{=}Y$ to denote equality in distribution, i.e., the equality $\mathrm{Law}(X)=\mathrm{Law}(Y)$. 
If $f$ is a function, we denote $\mbb E_{\mu}f=\mbb E_{X\sim\mu}f(X)$ its expectation with respect to $\mu$. 
If $(\mu_n)$ is a sequence of measures in $\mc P(S)$, we say $(\mu_n)$ converges weakly to $\mu$ and write $\mu_n\to\mu$ if $\mbb E_{\mu_n}f\to \mbb E_{\mu}f$ for all continuous $f$. 
All measures in this paper are supported in a single ambient compact set. In particular, all sequences of measures are tight. 

A coupling of random variables $X_1,\ldots,X_k$ taking values in Polish spaces $S_1,\ldots,S_k$, respectively, is a probability distribution $\Gamma\in\mc P(S_1\times\cdots\times S_k)$ such that if $(X_1',\ldots,X_k')\sim\Gamma$ then $X_i'\overset{d}{=}X_i$ for all $i\in[k]$. In this paper we refer to a coupling of $X_1,\ldots,X_k$ simply by the tuple of jointly-distributed random variables $(X_1',\ldots,X_k')$.
If $X_1,X_2$ are random graphs taking values in $\RR^{k\times k}$, we define the Wasserstein-1 distance between them by
\begin{equation*}
    W_1(X_1,X_2)=\inf_{\substack{\textrm{random } (X_1',X_2')\\ X_i'\overset{d}{=}X_i}}\mbb E\|X_1'-X_2'\|_1 = \sup_{\substack{f\colon\RR^{k\times k}\to\RR\\ \textrm{1-Lipschitz}}}|\mbb Ef(X_1)-\mbb Ef(X_2)|,
\end{equation*}
where the infimum is over couplings $(X_1',X_2')$ of $X_1$ and $X_2$, the norm $\|\cdot\|_1$ is the entrywise $\ell_1$ norm, and the supremum is over 1-Lipschitz functions with respect to this norm. Note that a sequence of random graphs $(X_n)\subseteq\Delta^{k\times k}$ converges weakly to the random graph $X\in\Delta^{k\times k}$ if and only if $\lim_nW_1(X_n,X)=0$ by~\cite[Thm.~6.9]{villani2008optimal}.

For a compact Polish space $S$, the space $\mc P(S)$ endowed with the weak topology is itself compact and Polish. If $\msf M$ is a random measure, that is, a random variable taking values in the Polish space $\mc P(S)$, we write $\mbb E_{\msf M}f$ for the scalar random variable obtained by integrating a deterministic function $f$ against the random measure $\msf M$. In particular, we write $\msf M(B)=\mbb E_{\msf M}\mathbbm{1}_B$ to denote the scalar random variable giving the (random) measure of a measurable set $B$, where $\mathbbm{1}_B(x)=1$ if $x\in B$ and zero otherwise. A sequence of random measures $(\msf M_n)$ converges weakly to $\msf M$ if $\mathrm{Law}(\msf M_n)\to\mathrm{Law}(\msf M)$ weakly, where these laws are probability measures over the space of probability measures. Equivalently, the sequence $(\msf M_n)$ converges weakly to $\msf M$ if the sequence of scalar random variables $\mbb E_{\msf M_n}f$ converges weakly to $\mbb E_{\msf M}f$ for all continuous functions $f$~\cite[Thm.~A2.3]{kallenberg2005probabilistic}.

\section{Quotient Convergence: Global and Local Views}\label{sec:hom_convergence}

% We begin our study of quotient convergence (Definition~\ref{def:quotient_convergence}) by giving a few examples of convergent sequences of graphs, and then prove Theorem~\ref{thm:hom_num_conv} relating quotient convergence with convergence of homomorphism numbers and quotient densities.

In this section we prove Theorem~\ref{thm:hom_num_conv} relating quotient convergence to convergence of homomorphism numbers and quotient densities.  Before presenting this proof, we give a few examples illustrating sequences of graphs that are quotient-convergent.  As our first example, we show that the sequence of complete graphs is quotient-convergent.

\begin{example}[Convergence of complete graphs]\label{ex:complete_graphs}
To show that the sequence $(\mathbbm{1}_{n\times n})_{n\in\NN}$ of complete graphs converges, consider the associated normalized graphs $(G_n=\mathbbm{1}_{n\times n}/n^2\in\Delta^{n\times n})$ and observe that the random quotients are given as $\rho(F_{k,n})G_n = \frac{1}{n^2}NN^\top$ for a multinomial random variable $N=(N_1,\ldots,N_k)\sim\mathrm{Multinom}(n,k,\mathbbm{1}_k/k)$.  Since $\mbb E\rho(F_{k,n})G_n = \frac{1}{k^2}\mathbbm{1}_{k\times k} + O(1/n)$ and $\mathrm{Var}[\rho(F_{k,n})G_n]_{i,j}=O(1/n)$ for all $i,j\in[k]$, we conclude that $\rho(F_{k,n})G_n$ converges weakly to the deterministic matrix $\frac{1}{k^2}\mathbbm{1}_{k\times k}$ as $n\to\infty$. Since this holds for all $k\in\NN$, we conclude that $(G_n)$ is quotient-convergent.
Note that the quotient densities and homomorphism numbers of this sequence likewise converge. Indeed, for any $H\in\NN^{k\times k}$ we have 
$t_Q(H;G_n)=\mbb E(\rho(F_{k,n})G_n)^H\xrightarrow{n\to\infty} (\frac{1}{k^2}\mathbbm{1}_k\mathbbm{1}_k^\top)^H=\frac{1}{k^{2\|H\|_1}}$. 
Similarly,
$\mathrm{hom}(H;G_n)=\frac{n^k}{n^{2\|H\|_1}}$, which converges to 1 if $H$ is a disjoint union of edges and to 0 otherwise.  An analogous set of arguments show that the sequence $(I_n)$ of $n$ self-loops is quotient-convergent.
\end{example}

As our next example, we consider sequences of (directed) star graphs and we show that these are quotient-convergent as well.  Note that stars are sparse graphs, as their edge density converges to zero.  Nevertheless, we will see in Section~\ref{sec:limits} that growing-sized stars have nontrivial limits in our framework, whereas the limits of stars in the context of graphons yield trivial limiting objects. Similarly, stars are not degree-bounded, as the degree of the star center increases with the size of the star, hence they do not have a limit in the framework of graphings. Stars do have limits in the graphex sense, though they differ from ours, see Section~\ref{sec:related_work}.

\begin{example}[Convergence of stars]\label{ex:stars}
The sequence of directed star graphs all of whose edges point away from the center may be expressed as $(e_1\mathbbm{1}_n^\top)$ for $e_1=(1,0,\ldots,0)^\top$. This sequence quotient-converges, since random quotients $\rho(F_{k,n})G_n$ of the associated normalized graphs $G_n=\frac{1}{n}e_1\mathbbm{1}_n^\top$ are given by $\rho(F_{k,n})G_n = \frac{1}{n}e_IN^\top$ for a multinomial vector $N\sim\mathrm{Multinom}(n,k,\mathbbm{1}_k/k)$ and a uniformly random index $I\sim\mathrm{Unif}([k])$. 
Since $N/n$ weakly converges to the deterministic vector $\mathbbm{1}_k/k$, we conclude that $\rho(F_{k,n})G_n$ converges weakly to the randomly-labeled star $\frac{1}{k}e_I\mathbbm{1}_k^\top$, and hence that $(G_n)$ is quotient-convergent.
Similarly, the quotient densities and homomorphism numbers of this sequence of stars converge, sinces
\begin{equation*}
    \lim_{n\to\infty}t_Q(H;G_n) = \mbb E\left(\frac{1}{k}e_I\mathbbm{1}_k^\top\right)^H=\begin{cases}
        \frac{1}{k^{1+\sum_jm_j}} & \textrm{if } H=e_im^\top \textrm{ for } m\in\NN^k \textrm{ and } i\in[k],\\ 0 & \textrm{otherwise},
    \end{cases}
\end{equation*}
and 
\begin{equation*}
    \lim_{n\to\infty}\mathrm{hom}(H;G_n) = \begin{cases} \lim_{n\to\infty}\frac{n^{k-1}}{n^{\sum_jm_j}} & \textrm{if }  H=e_im^\top,\\ 0 & \textrm{otherwise}, \end{cases} = \begin{cases}
        1 & \textrm{if } H= e_im^\top \textrm{ with } \sum_jm_j=k-1,\\ 0 & \textrm{otherwise}.
    \end{cases}
\end{equation*}
Analogous arguments show that the sequence of directed star graphs all of whose edges point towards the center quotient-converges, and so does the sequence of undirected stars.
\end{example}

As our final example of quotient-convergence, we consider an infinite weighted graph with edge weights summing to one and we show that any sequence of finite subgraphs obtained from this infinite graph is quotient-convergent.

\begin{example}[Subgraphs of an infinite weighted graph]
Suppose $G_{\infty}=(G_{i,j})_{i,j\in\NN}$ is an infinite two-dimensional array with $G_{i,j}\geq0$ and $\sum_{i,j}G_{i,j}=1$, and consider the sequence of normalized truncations $G_n=\frac{1}{\sum_{i,j=1}^nG_{i,j}}(G_{i,j})_{i,j\in[n]}$.  This sequence of subgraphs $(G_n)$ is quotient-convergent. Indeed, let $F_k\colon\NN\to[k]$ be a uniformly random map (whose images $F_k(i)$ are iid uniformly distributed over $[k]$), and consider the random graph $L_k=\rho(F_k)G_{\infty}$ whose entries are given by~\eqref{eq:quotient_graph}. We claim that the sequence $(\rho(F_{k,n})G_n)_n$ converges weakly to $L_k$. 
To this end, view $G_n$ as an infinite array by zero-padding it, and note that $(\rho(F_k)G_n,\rho(F_k)G_{\infty})$ is a coupling of $(\rho(F_{k,n})G_n, L_k)$. 
%\VC{I introduced $\tilde{G}_n$ notation in the previous line to not confuse $\rho(F_k)G_n$ and $\rho(F_{k,n})G_n$ -- please take a look.} 
Using this coupling, we obtain the bound
\begin{equation*}
W_1(\rho(F_{k,n})G_n, L_k)\leq \mbb E\|\rho(F_k)(G_n-G_{\infty})\|_1\leq \|G_n-G_{\infty}\|_1\xrightarrow{n\to\infty}0,
\end{equation*} 
where the second inequality follows from~\eqref{eq:quotient_graph}.
Thus, we conclude that $(G_n)$ is quotient-convergent, as claimed.
Quotient densities and homomorphism numbers of this sequence again converge, with limits $\lim_nt_Q(H;G_n)=\mbb EL_k^H$ and $\lim_{n}\mathrm{hom}(H;G_n)=\sum_{f\colon[k]\to\NN}\prod_{i,j=1}^kG_{f(i),f(j)}^{H_{i,j}}$, which is finite since $(G_{i,j})$ is summable.  
%The latter infinite sum converges since $\sum_{i,j}G_{i,j}=1$.
\end{example}
Finally, we remark that convex combinations of quotient-convergent sequences also converge.
\begin{example}[Combinations of graphs]\label{ex:combinations}
%\VC{One minor issue with this example is the notation -- we're using superscripts to denote number of edges elsewhere in the paper.  I would suggest taking the convex combination of two graph sequences $(G_n)$ and $(G'_n)$ rather than $m$ of these.  You could then say that the same process iterated a few times allows for arbitrary convex combinations.  (I know you need to take convex combinations later of more than two sequences.)}
Suppose $(G_n)$ and $(G_n')$ are quotient-convergent sequences of graphs. Then for any $\theta\in[0,1]$ the convex combination $(\theta G_n+(1-\theta)G_n')$ of these graphs is quotient-convergent as well. Indeed, for each $k\in\NN$, let $L_k$ and $L_k'$ be the weak limits of $\rho(F_{k,n})G_n$ and $\rho(F_{k,n})G_n'$, respectively. For each $n\in\NN$, let $(\widetilde L_k,\widetilde F_{k,n})$ and $(\widetilde L_k',\widetilde F_{k,n}')$ be couplings of $(L_k,F_{k,n})$ and $(L_k',F_{k,n})$, respectively, attaining the Wasserstein-1 distances $W_1(L_k,\rho(F_{k,n})G_n)$ and $W_1(L_k',\rho(F_{k,n})G_n)$. We define a coupling of $\theta L_k+(1-\theta)L_k'$ and $\rho(F_{k,n})(\theta G_n+(1-\theta)G_n')$ as follows. Draw $F_{k,n}$ uniformly at random, and draw $\widetilde L_k$ from the coupling $(\widetilde L_k,\widetilde F_{k,n})$ conditioned on $\widetilde F_{k,n}=F_{k,n}$. Draw $\widetilde L_k'$ similarly and independently. Then
\begin{equation*}\begin{aligned}
    &W_1(\theta L_k+(1-\theta)L_k', \rho(F_{k,n})(\theta G_n+(1-\theta)G_n')) \leq \mbb E\|\theta\widetilde L_k + (1-\theta)\widetilde L_k' - \rho(F_{k,n})(\theta G_n+(1-\theta G_n'))\|_1\\&\quad \leq \theta W_1(L_k,\rho(F_{k,n})G_n)+(1-\theta)W_1(L_k',\rho(F_{k,n})G_n')\xrightarrow{n\to\infty}0.
\end{aligned}\end{equation*}
This proves the claim. An analogous argument shows that a convex combination of any finitely-many quotient-convergent sequences of graphs is also quotient-convergent.
    %Suppose $(G_n^{(1)}),\ldots,(G_n^{(m)})$ are $m$ quotient-convergent sequences of graphs. Then for any $\theta\in\Delta^m$, the combination of these graphs $(\sum_{i=1}^m\theta_iG_n^{(i)})_n$ is quotient-convergent as well. Indeed, for each $k\in\NN$ and $i\in[m]$ let $L_k^{(i)}$ be the weak limit of $\rho(F_{k,n})G_n^{(i)}$. For each $n\in\NN$ let $(\widetilde L_k^{(i)}, \widetilde F_{k,n})$ be an optimal coupling. Draw $F_{k,n}$ uniformly at random and draw $\widetilde L_{k^{(i)}}$ from the coupling $\mathrm{Law}(\widetilde L_k^{(i)},\widetilde F_{k,n})$ conditioned on $\widetilde F_{k,n}=F_{k,n}$, independently for different $i\in[m]$. Then $(\widetilde L_{k}^{(i)},F_{k,n})\sim\mathrm{Law}(\widetilde L_k^{(i)},\widetilde F_{k,n})$ for all $i\in[m]$ and
    % \begin{equation*}\begin{aligned}
    %     W_1\left(\sum\nolimits_i\theta_iL_k^{(i)}, \rho(F_{k,n})\sum\nolimits_i\theta_iG_n^{(i)}\right) &\leq \mbb E\left\|\sum\nolimits_i\theta_i\widetilde L_k^{(i)} - \sum\nolimits_i\theta_i\rho(F_{k,n})G_n^{(i)}\right\|_1\leq \sum\nolimits_i\theta_i\mbb E\|\widetilde L_k^{(i)}-\rho(F_{k,n})G_n^{(i)}\|_1\\ &= \sum\nolimits_i\theta_i\mbb E\|\widetilde L_k^{(i)}-\rho(\widetilde F_{k,n})G_n^{(i)}\|_1 = \sum\nolimits_i\theta_iW_1(L_k^{(i)}, \rho(F_{k,n})G_n^{(i)})\xrightarrow{n\to\infty}0.
    % \end{aligned}\end{equation*}
    % This proves the claim.
\end{example}

Having seen a few examples of quotient convergence, we proceed to prove Theorem~\ref{thm:hom_num_conv}.  We begin by showing the first part of that theorem on the equivalence between quotient convergence and convergence of quotient densities.

\begin{proposition} \label{prop:quotient-part1}
    A sequence of graphs $(G_n\in\Delta^{n\times n})$ is quotient-convergent if and only if $(t_Q(H;G_{n}))_n$ converges for each $H\in\NN^{k\times k}$ and $k\in\NN$. 
\end{proposition}
\begin{proof}
    By definition, the sequence $(G_n\in\Delta^{n\times n})$ is quotient-convergent when the sequence$(\rho(F_{k,n})G_n\in\Delta^{k\times k})_n$ converges weakly for each $k\in\NN$. Since $\Delta^{k\times k}$ is compact, the latter sequence converges weakly if and only if all its associated sequences of moments converge. Now observe that the monomials in $\RR[X_{i,j}]_{i,j\in[k]}$ all have the form $X^H = \prod_{i,j}X_{i,j}^{H_{i,j}}$ for some $H\in\NN^{k\times k}$, and that the corresponding moment of $\rho(F_{k,n})G_n$ is precisely $\mbb E(\rho(F_{k,n})G_n)^{H} = t_Q(H;G_n)$ by definition of $t_Q$ in~\eqref{eq:quotient_densities}. 
    This yields the desired conclusion.
\end{proof}

This proves the first part of Theorem~\ref{thm:hom_num_conv}.
To show the second part relating quotient convergence with convergence of homomorphism numbers, we express homomorphism numbers in terms of quotient densities and vice-versa. To this end, we begin by expressing quotient densities in terms of injective homomorphism numbers using~\cite[Prop.~6.8]{levin2025deFin}.  This expression is stated in terms of a partial order on multigraphs induced by quotients. Specifically, for two multigraphs $K\in\NN^{n\times n},H\in\NN^{m\times m}$ we say that $K$ is a \emph{refinement} of $H$, denoted $K\leq_R H$, if $H$ is a quotient of $K$, i.e., if there exists a surjective map $f\colon [n]\to [m]$ such that $H=\rho(f)K$.  We further let $R_{K,H}$ be the number of such maps $f$. Note that $R_{H,H}=|\mathrm{Aut}(H)|$ for any multigraph $H$.  The poset of all multigraphs with the refinement order is a disjoint union of posets, one for each number of edges. The poset of all multigraphs with $n$ edges has a minimal element, namely, the disjoint union of $n$ edges, and a maximal element, namely, a single vertex with $n$ self-loops. We illustrate the poset of all multigraphs with two edges in Figure~\ref{fig:refinements}. 
\begin{figure}[h]
    \centering
    \includegraphics[width=0.33\linewidth]{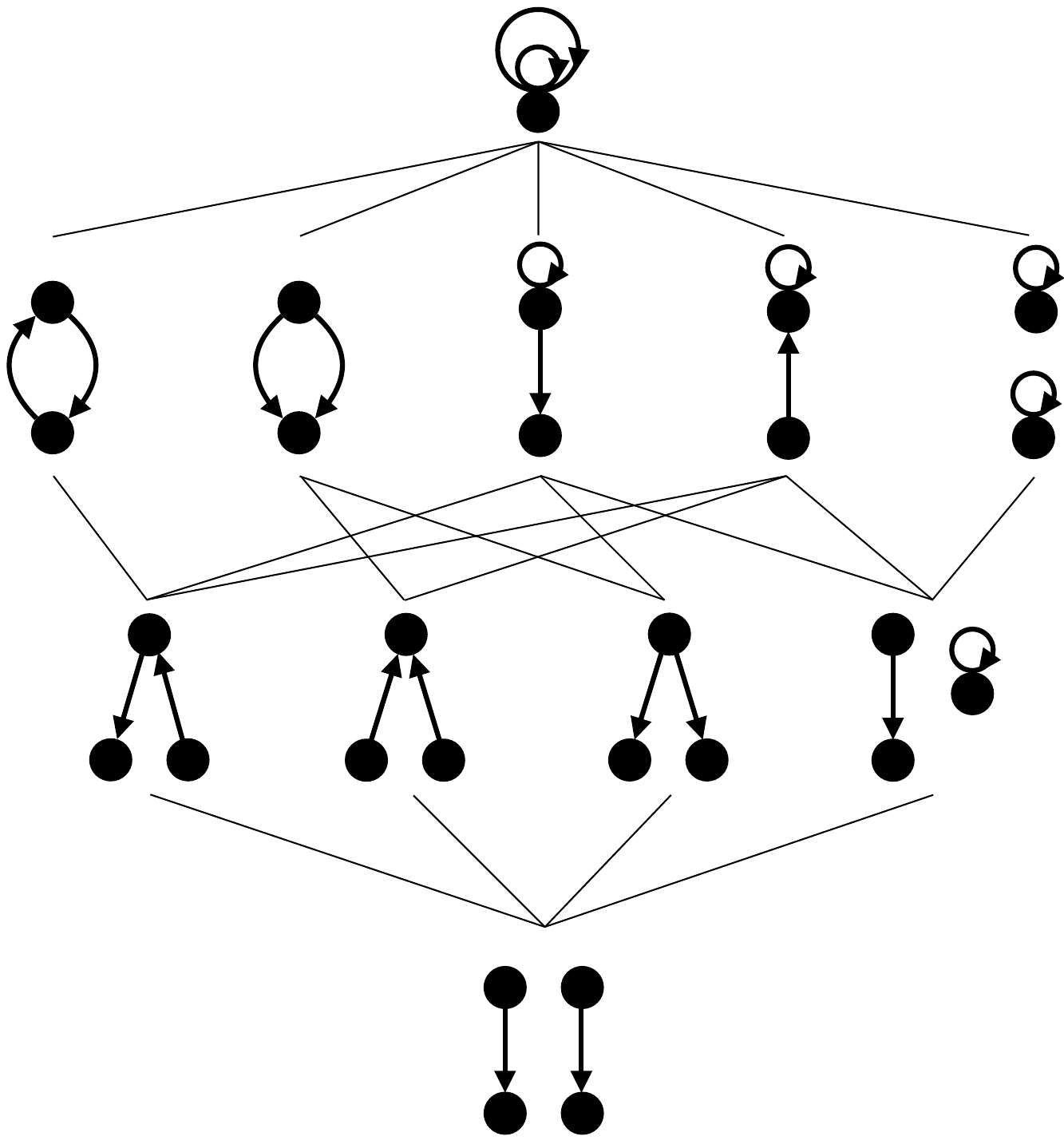}
    \caption{Hasse diagram for the refinement poset consisting of multigraphs with two edges. Here there is an upward path from $K$ to $H$ precisely when $K\leq_R H$.}
    \label{fig:refinements}
\end{figure}

% \VC{In this paragraph, what does `appear mean'?  Do you mean as a subgraph?  It seems that there is some switch between unweighted graphs when defining the refinement order, and then weighted graphs below when considering quotient densities and homomorphism numbers.  It would be good to be clear about this.}
% Note that $H$ is a subgraph of a quotient of $G$, i.e., we have $(\rho(f)G)^H>0$ for some surjective map $f$, precisely when some refinement of $H$ is a subgraph of $G$ itself. \VC{Preceding sentence is a bit quick.} We can further obtain the number of times $H$ appears in a quotient of $G$ from the number of times a refinement of $H$ appears in $G$. This was done in~\cite[Prop.~6.8]{levin2025deFin}, where it was shown that
The quotient density $t_Q(H;G)$ can be expressed in terms of a linear combination of injective homomorphism numbers over all refinements of $H$~\cite[Prop.~6.8]{levin2025deFin}:
\begin{equation}\label{eq:tQ_in_inj}
    t_Q(H;G) = \sum_{K\leq_R H}|V(H)|^{-|V(K)|}\frac{H!}{K!}\frac{R_{K,H}}{R_{K,K}R_{H,H}}\mathrm{inj}(K;G),
\end{equation}
where $H!=\prod_{i,j=1}^k(H_{i,j}!)$ and similarly for $K!$.
%\VC{What is $H!$ and $K!$ in the preceding expression?  Also, do you mean to say $\mathrm{inj}(K;G)$ rather than $\mathrm{inj}(H;G)$?}
%where the sum is over multigraphs $K$ with no isolated vertices that are refinements of $H$ 
%\VC{Should this condition be put in the sum explicitly?  More broadly, the introduction of this refinement order is a bit quick and some more points could be made.  For example, it may be useful to mention above that there is no further refinement of a set of disjoint edges -- this would help clarify that the above sum should be finite.} (so this is a finite sum), 
Here $\mathrm{inj}(K;G)$ is defined similarly to $\mathrm{hom}(K;G)$ in~\eqref{eq:hom_nums} but with a sum only over injective maps.
The expression~\eqref{eq:tQ_in_inj} states that $H$ is a subgraph of some quotient of $G$, i.e., $t_Q(H;G)>0$, if and only if some refinement of $H$ is a subgraph of $G$ itself, i.e., $\mathrm{inj}(K;G)>0$. 

We are now ready to prove the second part of Theorem~\ref{thm:hom_num_conv}. The following argument is a simple modification of the one from~\cite[Prop.~6.8]{levin2025deFin}.
\begin{proposition}\label{prop:hom_conv}
    A sequence of graphs $(G_n\in\Delta^{n\times n})$ is quotient-convergent if and only if $(\mathrm{inj}(H;G_n))$ converges for each $H\in\NN^{k\times k}$ and $k\in\NN$, which is equivalent to $(\mathrm{hom}(H;G_n))$ converging for each $H\in\NN^{k\times k}$ and $k\in\NN$. 
\end{proposition}
\begin{proof}
    For each $m\in\NN$, define the matrix $M$ indexed by all multigraphs containing $m$ edges whose entries are given by $M_{H,K} = |V(H)|^{-|V(K)|}\frac{H!}{K!}\frac{R_{K,H}}{R_{K,K}R_{H,H}}$ if $K\leq_R H$ and $M_{H,K}=0$ otherwise. 
    This is a lower-triangular matrix with respect to the refinement partial order, and has positive diagonal entries $M_{H,H}=(|V(H)|^{|V(H)|}|\mathrm{Aut}(H)|)^{-1}$, hence it is invertible. Thus, we can write 
    \begin{equation}\label{eq:inj_in_tQ}
        \mathrm{inj}(H;G) = \sum_{H\leq_R K}(M^{-1})_{H,K}t_Q(K;G).
    \end{equation}
    Based on~\eqref{eq:inj_in_tQ} and~\eqref{eq:tQ_in_inj}, we conclude that $(t_Q(H;G_n))$ converges for all multigraphs $H$ if and only if $(\mathrm{inj}(H;G_n))$ converges for all multigraphs $H$. 
    Injective and non-injective homomorphism numbers can be similarly related, with one of the quantities being expressed in terms of linear combinations of the other as in~\cite[\S5.2.3]{lovasz2012large}. For example, we have~\cite[Eq.~(73)]{levin2025deFin} that 
    \begin{equation}\label{eq:hom_in_inj}
        \mathrm{hom}(H;G)=\sum_{H\leq_R K}\frac{R_{H,K}}{R_{K,K}}\mathrm{inj}(K;G),
    \end{equation}
    where the sum is over multigraphs that $H$ refines. This expression can be inverted as above to obtain injective homomorphism numbers in terms of non-injective homomorphism numbers.
    This allows us to conclude that convergence of $(\mathrm{inj}(H;G_n))$ for all multigraphs $H$ is equivalent to convergence of $(\mathrm{hom}(H;G_n))$ for all multigraphs $H$.
\end{proof}

Combined together, Propositions~\ref{prop:quotient-part1} and~\ref{prop:hom_conv} complete the proof of Theorem~\ref{thm:hom_num_conv}.  We illustrate the utility of these results next by expressing quotient convergence of simple graphs in terms of appropriately-normalized homomorphism numbers.  The resulting expression allows us to conclude that all dense simple graphs are quotient-convergent; in fact, we prove that all the homomorphism numbers of dense simple graphs converge to the same limit (see also Corollary~\ref{cor:undirected_graphs} below). 
%As a corollary, we prove that all sequences of dense graphs $(G_n)$ are quotient-convergent. In fact, we prove that all the homomorphism numbers of graphs in such a sequence converge to the same limit, which allows us to prove in the next section that all such sequences converge to the same limit.
\begin{corollary}\label{cor:dense_graphs1}
    Let $(G_n\in\mbb S^n_{\mathrm{sim}})$ be a sequence of simple graphs.
    \begin{enumerate}[align=left, font=\emph]
        \item[(Convergence of simple graphs)] The sequence $(G_n)$ is quotient-convergent if and only if $\left(\frac{\mathrm{hom}(H;G_n)}{|E(G_n)|^{|E(H)|}}\right)$ converges for all simple and connected graphs $H\in\mbb S^k_{\mathrm{sim}}$. 

        \item[(Convergence of dense graphs)] Suppose $\liminf_{n\to\infty}\frac{|E(G_n)|}{n^2}>0$. Then $(G_n)$ is quotient-convergent. In fact, for a simple and connected $H\in\mbb S^k_{\mathrm{sim}}$, we have 
    \begin{equation*}
        \lim_{n\to\infty}\frac{\mathrm{hom}(H;G_n)}{|E(G_n)|^{|E(H)|}} = \begin{cases} 1 & \textrm{if } H=K_2,\\ 0 & \textrm{otherwise}\end{cases}
    \end{equation*}
    where $K_2$ denotes a graph with a single edge.
    \end{enumerate}
\end{corollary}
\begin{proof}
    For the first claim, if $H\in\mbb N^{k\times k}$ we define a strictly upper-triangular $\widetilde H\in\{0,1\}^{n\times n}$ by $\widetilde H_{i,j}=1$ if $i<j$ and either $H_{i,j}>0$ or $H_{j,i}>0$, and $\widetilde H_{i,j}=0$ otherwise.  
    Then $\mathrm{hom}(H;G_n)=\mathrm{hom}(\widetilde H,G_n)$ for all simple graphs $G_n$ of all sizes. 
    Therefore, we have
    \begin{equation*}
        \mathrm{hom}(H;\tfrac{G_n}{\mathbbm{1}^\top G_n\mathbbm{1}})=\mathrm{hom}(\widetilde H;\tfrac{G_n}{\mathbbm{1}^\top G_n\mathbbm{1}})\cdot (2|E(G_n)|)^{\mathbbm{1}^\top(\widetilde H- H)\mathbbm{1}}.
    \end{equation*}
    Since $\mathbbm{1}^\top \widetilde H\mathbbm{1}\leq \mathbbm{1}^\top H\mathbbm{1}$, we conclude that $\mathrm{hom}(H;\tfrac{G_n}{\mathbbm{1}^\top G_n\mathbbm{1}})$ converges if $\mathrm{hom}(\widetilde H;\tfrac{G_n}{\mathbbm{1}^\top G_n\mathbbm{1}})$ converges, since either $|E(G_n)|\to\infty$ and $\mathrm{hom}(H;\tfrac{G_n}{\mathbbm{1}^\top G_n\mathbbm{1}})\to0$ or $|E(G_n)|$ is constant for all large $n$. 
    In turn, defining the simple graph $\widetilde H_{\mathrm{sym}}=(\widetilde H + \widetilde H^\top)/2$, we have 
    \begin{equation*}
        \mathrm{hom}\left(\widetilde H;\frac{G_n}{\mathbbm{1}^\top G_n\mathbbm{1}}\right)=\frac{\mathrm{hom}(\widetilde H_{\mathrm{sym}}; G_n)}{(2|E(G_n)|)^{|E(\widetilde H_{\mathrm{sym}})|}}.
    \end{equation*}
    Thus, the sequence $(G_n)$ quotient-converges if and only if $(\frac{\mathrm{hom}(H;G_n)}{|E(G_n)|^{|E(H)|}})_n$ converges for all simple $H$.
    To see that it suffices to consider simple and connected graphs, note that if $H=H_1\sqcup\ldots\sqcup H_{\ell}$ is a disjoint union of connected graphs, then $\frac{\mathrm{hom}(H;G_n)}{|E(G_n)|^{|E(H)|}}=\prod_{i=1}^{\ell}\frac{\mathrm{hom}(H_i;G_n)}{|E(G_n)|^{|E(H_i)|}}$. This proves the first claim.
    
    For the second claim, if $H=K_2$ then $\mathrm{hom}(H;G_n)/|E(G_n)|=1$ by definition. 
    Otherwise, suppose $\frac{|E(G_n)|}{n^2}\geq c>0$ for all $n$. 
    Note that $\mathrm{hom}(H;G_n)\leq 2|E(G_n)|\cdot n^{|V(H)|-2}$, because the number of graph homomorphisms $V(H)\to V(G_n)$ mapping all edges of $H$ to edges of $G_n$ is upper-bounded by the number of maps $V(H)\to V(G_n)$ that send one particular edge in $H$ to an edge in $G_n$. 
    Since $H$ is connected and contains more than one edge, we have $|V(H)|<2|E(H)|$, in which case
    \begin{equation*}
        \frac{\mathrm{hom}(H;G_n)}{|E(G_n)|^{|E(H)|}} \leq \frac{2n^{|V(H)|-2}}{|E(G_n)|^{|E(H)|-1}}\leq \frac{2}{c^{|E(H)|-1}}\cdot n^{|V(H)|-2|E(H)|}\xrightarrow{n\to\infty}0.
    \end{equation*}
    This proves the quotient-convergence of $(G_n)$ by the first part of this corollary.
\end{proof}

%While this result specifically concerns limits of quotient-convergent sequences of deterministic simple graphs, 
We have now seen several examples and characterizations of quotient-convergent sequences. However, so far it is not clear to what limits such sequences converge. 
In the next section, we derive limits for quotient-convergent graphs in the form of random exchangeable measures on $[0,1]^2$.
%Having characterized quotient convergence both in terms of global quotients and local subgraph counts, we proceed to construct limit objects and study their structure.

\section{Limits of Quotient-Convergent Graph Sequences}\label{sec:limits}
This section is devoted to characterizing limits of quotient-convergent graph sequences and understanding their properties.
In Section~\ref{sec:characterizing_lims}, we prove Theorem~\ref{thm:limits_formal} characterizing limits of quotient-convergent graph sequences via grapheurs. In Section~\ref{sec:structure_of_lims} we interpret the different components of limiting grapheurs~\eqref{eq:kallenberg_char} in terms of combinatorial aspects of the sequence of graphs converging to them. 
In Section~\ref{sec:dual_graphons} we present a (random) duality pairing between grapheurs and graphons, and prove the formal version of Theorem~\ref{thm:graphon_duality_intro} given in Theorem~\ref{thm:graphon_duality}.
Finally, Section~\ref{sec:missing_proofs} contains some tedious but straightforward proofs omitted from Section~\ref{sec:characterizing_lims} for clarity of exposition.

%\VC{From taking a quick early look at the next section on sampling (which I recognize is still in progress as I'm writing this), I see that the phrase `a sequence of graphs of growing size quotient-converges to some limiting random exchangeable measure' is being used.  However, this phrase is not formally defined, and it seems to be a consequence of Corollary 4.4 from my understanding.  It would be good to state this usage after Corollary 4.4.}

\subsection{Characterizing Limits via Grapheurs}\label{sec:characterizing_lims}
%In this section, we prove Theorem~\ref{thm:limits_formal} identifying $\mc M$ as the compact space of limit objects for quotient-convergent graph sequences.

To prove Theorem~\ref{thm:limits_formal}, we show that the random quotients of size $k$ of a finite graph $G$ are equal in distribution to the random graph $\msf G_k[\msf M_G]$ of size $k$ obtained in~\eqref{eq:graph_from_measure} from the grapheur $\msf M_G$ associated in~\eqref{eq:step_measure} to $G$. 
\begin{proposition}\label{prop:induced_random_graph}
    For any $n,k\in\NN$ and $G\in\Delta^{n\times n}$, we have $\rho(F_{k,n})G\overset{d}{=}\msf G_k[\msf M_G]$. 
\end{proposition}
\begin{proof}
    Recall from~\eqref{eq:step_measure} that $\msf M_G = \sum_{i,j=1}^nG_{i,j}\delta_{(T_i,T_j)}$ where $T_1,\ldots,T_n$ are iid uniform on $[0,1]$. With probability one, for each $i\in[n]$ there is a unique $j_i\in[k]$ such that $T_i\in I_{j_i}^{(k)}$. Moreover, the index $j_i$ is uniformly distributed in $[k]$ by uniformity of $T_i$, and the indices $j_1,\ldots,j_n$ are independent of each other since the $T_i$ are independent. 
    Define the random map $F\colon[n]\to[k]$ sending $F(i)=j_i$ and observe that $F\overset{d}{=}F_{k,n}$ is uniformly random, since each $i\in[n]$ is mapped to an independent and uniformly random element of $[k]$. Further, note that
    \begin{equation*}
        (\msf G_k[\msf M_G])_{i,j} = \msf M_G(I_i^{(k)}\times I_j^{(k)}) = \sum_{u,v=1}^nG_{u,v}\mathbbm{1}[T_u\in I_i^{(k)}, T_v\in I_j^{(k)}] = \sum_{\substack{u\in F^{-1}(i) \\ v\in F^{-1}(j)}}G_{u,v} = (\rho(F)G)_{i,j}.
    \end{equation*}
    Hence, we conclude that $\msf G_k[\msf M_G] = \rho(F)G\overset{d}{=}\rho(F_{k,n})G$, as claimed.
\end{proof}
Consequently, we also have $t_Q(H;G)=t_Q(H;\msf M_G)$, where the latter is given by~\eqref{eq:grapheur_density}.
Proposition~\ref{prop:induced_random_graph} implies that a sequence $(G_n)$ is quotient-convergent if and only if the sequence of random graphs $(\msf G_k[\msf M_{G_n}])_n$ constructed from the associated grapheurs $(\msf M_{G_n})$ converge weakly for each $k$.
Since $\Delta^{k\times k}$ is compact, this is equivalent to convergence of $(\msf G_k[\msf M_{G_n}])_n$ in the Wasserstein-1 metric $W_1$ with respect to the entrywise $\ell_1$ norm on $\RR^{k\times k}$ for each $k$ (see Section~\ref{sec:notation}).
Motivated by this observation, we relate our $W_\square$-metric~\eqref{eq:dual_cut_metric} to this $W_1$-metric.
\begin{proposition}\label{prop:comparing_metrics}
    For any two grapheurs $\msf M,\msf N\in\mc M$ and any $k\in\NN$, we have
    \begin{equation*}
        W_1(\msf G_k[\msf M], \msf G_k[\msf N])\leq k^2W_{\square}(\msf M,\msf N)\quad \textrm{and } \quad W_{\square}(\msf M,\msf N)\leq W_1(\msf G_k[\msf M], \msf G_k[\msf N]) + \frac{4}{k}.
    \end{equation*}
\end{proposition}
The proof is straightforward but somewhat technical, as it is based on approximating any rectangle appearing in our expression~\eqref{eq:dual_cut_metric} for $W_{\square}$ by squares in a uniform grid. We defer the proof to Section~\ref{sec:missing_proofs}.
As an immediate consequence, we conclude that a sequence of graphs $(G_n)$ is quotient-convergent if and only if the associated sequence of grapheurs $(\msf M_{G_n})$ is Cauchy in the $W_\Box$-metric~\eqref{eq:dual_cut_metric}.  
We now observe that the metric space $(\mc M,W_{\square})$ is compact, because the metric $W_{\square}$ metrizes weak convergence in $\mc M$ and $\mc M$ is weakly compact; see Section~\ref{sec:notation} for a review of weak convergence.
\begin{proposition}\label{prop:M0_compact} 
    The space $\mc M$ of grapheurs~\eqref{eq:grapheur_space} is weakly compact. In particular, any Cauchy sequence $(\msf M_n)\subseteq\mc M$ has a limit $\msf M\in\mc M$.  Moreover, the following statements are equivalent for $(\msf M_n)\subseteq\mc M$ and $\msf M\in\mc M$.
    \begin{enumerate}
        \item We have $\lim_n\msf M_n=\msf M$ weakly.  
        \item We have that $\lim_nW_{\square}(\msf M_n,\msf M)=0$. 
        \item For each $k\in\NN$, we have that $\lim_n \msf G_k[\msf M_n]=\msf G_k[\msf M]$ weakly.
        \item For each multigraph $H\in\NN^{k\times k}$, we have $\lim_nt_Q(H;\msf M_n)=t_Q(H;\msf M)$.
    \end{enumerate}
\end{proposition}
In particular, we have $\lim_nW_{\square}(\msf M_{G_n},\msf M)=0$ for a sequence $(G_n\in\Delta^{n\times n})$ if and only if $\lim_nt_Q(H;G_n)=t_Q(H;\msf M)$ for all multigraphs $H\in\NN^{k\times k}$, as can be seen by combining Proposition~\ref{prop:induced_random_graph} and Proposition~\ref{prop:M0_compact}(4). 
Once again, the proof is straightforward but involved, and we defer it to Section~\ref{sec:missing_proofs}.
Combining the preceding propositions, we are ready to prove Theorem~\ref{thm:limits_formal}.
% obtain the following characterization of quotient convergence.
% \begin{corollary}
%     A sequence of graphs $(G_n\in\Delta^{n\times n})$ is quotient-convergent if and only if there exists $\msf M\in\mc M$ satisfying $\lim_nW_{\square}(\msf M,\msf M_{G_n})=0$. Moreover, any sequence of graphs has a quotient-convergent subsequence.
% \end{corollary}
% \begin{proof}
    % By Proposition~\ref{prop:sampling_metric_metrizes_quotient}, the sequence $(G_n)$ is quotient-convergent if and only if the sequence $(\msf M_{G_n})$ is Cauchy in the sampling metric. By Proposition~\ref{prop:comparing_metrics}, this happens if and only if $(\msf M_{G_n})$ is Cauchy in the metric $\delta$. By Proposition~\ref{prop:M0_compact}, this is further equivalent to the existence of a limit $\msf M\in\mc M$ for $(\msf M_{G_n})$ in this metric. The existence of a convergent subsequence likewise follows from compactness of $\mc M$.
% \end{proof}
% All that remains to finish the proof of Theorem~\ref{thm:limits_formal} is to show that any $\msf M\in\mc M$ is the limit of a quotient-convergent sequence of graphs.
% We do so explicitly, with the help of Proposition~\ref{prop:M0_compact} and the examples of Section~\ref{sec:hom_convergence}.
\begin{proof}[Proof (Theorem~\ref{thm:limits_formal}).]
    To prove the first statement, note that the sequence $(G_n)$ is quotient-convergent if and only if $(\msf G_k[\msf M_{G_n}])$ is Cauchy in $W_1$ for each $k\in\NN$ by Definition~\ref{def:quotient_convergence}. By Proposition~\ref{prop:comparing_metrics}, this is equivalent to $(\msf M_{G_n})$ being Cauchy in the $W_{\square}$-metric, as we remark previously. By Proposition~\ref{prop:M0_compact}, this is further equivalent to the existence of a limit $\msf M\in\mc M$ for $(\msf M_{G_n})$ in this metric, thus proving the first statement.

    The second statement follows from Proposition~\ref{prop:M0_compact}.

    Finally, we prove the third statement by explicitly constructing a convergent sequence $(G_n)$ of graphs for each grapheur $\msf M\in\mc M$. Suppose $\msf M$ is given by~\eqref{eq:kallenberg_char} with $E,\sigma,\varsigma,\theta,\vartheta\geq0$ and iid uniform $(T_i)_{i\in\NN}$. 
    %For each $k\in\NN$, define the random map $F_k\colon\NN\to[k]$ as in the proof of Proposition~\ref{prop:induced_random_graph}, by sending $F_k(i)=j$ if $T_i\in I_j^{(k)}$ (such $j$ exists with probability 1). As in the proof of that proposition, we observe that $F_k(i)\in[k]$ are iid uniformly random. Further, we have
    % \begin{equation}\label{eq:random_graph_from_kallenberg}
    %     (\msf G_k[\msf M])_{i,j} = \sum_{\substack{u\in F^{-1}(i)\\ v\in F^{-1}(j)}}E_{u,v} + \sum_{u\in F^{-1}(i)}\tfrac{1}{k}\sigma_u + \sum_{v\in F^{-1}(j)}\tfrac{1}{k}\varsigma_v + \tfrac{1}{k^2}\theta + \tfrac{1}{k}\vartheta \mathbbm{1}_{i=j}.
    % \end{equation}
    Define the sequence of graphs $(G_n\in\Delta^{n\times n})$ with entries
    \begin{equation}\label{eq:particular_sequence}
        (G_n)_{i,j} = \frac{1}{s_n}\left(E_{i,j} + \tfrac{1}{n}\sigma_i + \tfrac{1}{n}\varsigma_j + \tfrac{1}{n^2}\theta+\tfrac{1}{n}\vartheta\mathbbm{1}_{i=j}\right),\qquad s_n=\sum_{i,j=1}^nE_{i,j} + \sum_{i=1}^n(\sigma_i+\varsigma_i) + \theta+\vartheta,
    \end{equation}
    where we note that $s_n\nearrow1$ and we may assume $s_n>0$ for all $n$ by reordering the rows and columns of $E,\sigma,\varsigma$ if needed. 
    Each summand in~\eqref{eq:particular_sequence} quotient-converges to the corresponding summand in~\eqref{eq:kallenberg_char} as shown by combining Proposition~\ref{prop:M0_compact} with the examples of Section~\ref{sec:hom_convergence}. 
    Moreover, by Example~\ref{ex:combinations} the convex combination of these summands in~\eqref{eq:particular_sequence} converges to the convex combination of the summands in~\eqref{eq:kallenberg_char}, which is precisely the grapheur $\msf M$.
    % For each $k\in\NN$, we have $\rho(F_{k,n})G_n \overset{d}{=}\rho(F_k|_{[n]})G_n$ and the entries of the latter random graph are given by
    % \begin{equation*}
    %     (\rho(F_k|_{[n]})G_n)_{i,j} = \frac{1}{s_n}\left[\sum_{\substack{u\in F^{-1}(i)\cap[n]\\ v\in F^{-1}(j)\cap[n]}}E_{u,v} + \sum_{u\in F^{-1}(i)\cap[n]}\tfrac{n_i}{n}\sigma_u + \sum_{v\in F^{-1}(i)\cap[n]}\tfrac{n_j}{n}\varsigma_v + \tfrac{n_in_j}{n^2}\theta + \tfrac{n_i}{n}\vartheta \mathbbm{1}_{i=j}\right],
    % \end{equation*}
    % where $n_i = |F^{-1}(i)\cap[n]|$ for $i\in[k]$.
    % Observe that $(n_1,\ldots,n_k)\sim\mathrm{Multinom}(n,k,\mathbbm{1}_k/k)$, hence $\frac{n_k}{n}\to \frac{1}{k}$ and $\frac{n_in_j}{n^2}\to\frac{1}{k^2}$ weakly. Therefore, we have
    % \begin{equation*}\begin{aligned}
    %     &W_1(\rho(F_{k,n})G_n, \msf G_k[\msf M])\leq \mbb E\|\rho(F_k|_{[n]})G_n-\msf G_k[\msf M]\|_1\leq (s_n^{-1}-1)\sum_{i,j=1}^nE_{i,j} + \left(\sum_{i=1}^k\mbb E|\tfrac{1}{k}-\tfrac{n_i}{n}s_n^{-1}|\right)\left(\sum_{i=1}^n(\sigma_i+\varsigma_i)+\vartheta\right)\\ &+ \left(\sum_{i,j=1}^k\mbb E|\tfrac{1}{k^2}-\tfrac{n_in_j}{n^2}s_n^{-1}|\right)\theta + \sum_{i>n \textrm{ or } j>n}E_{i,j} + \sum_{i>n}(\sigma_i+\varsigma_i)\xrightarrow{n\to\infty}0.
    % \end{aligned}\end{equation*}
    % This proves that $(\rho(F_{k,n})G_n)_n$ converges weakly to $\msf G_k[\msf M]$, and hence that $\lim_nW_{\square}(\msf M_{G_n},\msf M)=0$ by Proposition~\ref{prop:M0_compact}. 
    Thus, the third statement in Theorem~\ref{thm:limits_formal} is proved.
\end{proof}
In particular, the sequence of limits of random quotients of a convergent graph sequence fully determines its limit.
\begin{corollary}\label{cor:homs_determine_limits}
    Two quotient-convergent sequences $(G_n^{(1)})$ and $(G_n^{(2)})$ have the same limit if and only if $\lim_n\rho(F_{k,n})G_n^{(1)}\overset{d}{=}\lim_n\rho(F_{k,n})G_n^{(2)}$ for all $k\in\NN$.
\end{corollary}
\begin{proof}
    From Proposition~\ref{prop:M0_compact}, the two sequences have the same limit $\msf M$ if and only if $\lim_n\rho(F_{k,n})G_n^{(i)}\overset{d}{=}\msf G_k[\msf M]$ for both $i=1,2$.
    %, or equivalently, if and only if $\lim_n\mathrm{hom}(H;G_n^{(1)})=\lim_n\mathrm{hom}(H;G_n^{(2)})$ for all multigraphs $H$. by the proofs of Propositions~\ref{prop:quotient-part1} and~\ref{prop:hom_conv}.
\end{proof}

We end this section by noting that valid polynomial inequalities in quotient densities precisely correspond to valid inequalities over grapheurs. We hope that grapheurs can be used to study such inequalities in future work, dually to the use of graphons in the study of inequalities in homomorphism densities~\cite[\S16]{lovasz2012large}.
\begin{corollary}[Inequalities in quotient densities]\label{cor:inequalities_in_quotients}
    Suppose $H_1,\ldots,H_k$ are multigraphs and $f\colon\RR^k\to\RR$ is a continuous function. We have $f(t_Q(H_1;G),\ldots,t_Q(H_k;G))\geq 0$ for all $G\in\Delta^{n\times n}$ and all $n\in\NN$ if and only if $f(t_Q(H_1;\msf M),\ldots,t_Q(H_k;\msf M))\geq 0$ for all grapheurs $\msf M\in\mc M$.
\end{corollary}
\begin{proof}
    Since $t_Q(H;G)=t_Q(H;\msf M_G)$, if $f(t_Q(H_1;\msf M),\ldots,t_Q(H_k;\msf M))\geq 0$ for all grapheurs $\msf M\in\mc M$ then setting $\msf M=\msf M_G$ for $G\in\Delta^{n\times n}$ yields $f(t_Q(H_1;G),\ldots,t_Q(H_k;G))\geq 0$. Conversely, suppose $f(t_Q(H_1;G),\ldots,\allowbreak t_Q(H_k;G))\geq 0$ for all $G\in\Delta^{n\times n}$ and all $n\in\NN$. For any $\msf M\in\mc M$, let $(G_n)$ be a sequence of finite graphs that quotient-converges to $\msf M$. Then $t_Q(H;G_n)\to t_Q(H;\msf M)$ for all multigraphs $H$, hence by continuity of $f$ we have $f(t_Q(H_1;\msf M),\ldots,t_Q(H_k;\msf M))=\lim_nf(t_Q(H_1;G_n),\ldots,t_Q(H_k;G_n))\geq0$.
\end{proof}
%Here quotient densities of grapheurs are defined in~\eqref{eq:grapheur_density}. 
%An analogous result holds for inequalities in homomorphism numbers by the proof of Proposition~\ref{prop:hom_conv}.
%\VC{I removed discussion of homomorphism numbers in the preceding corollary.}
 %, or if all their quotient densities converge to the same limits.
% Having characterized grapheurs as limits of quotient-convergent graph sequences, we proceed to relate the structure of a grapheur to properties of the graphs converging to it.
\subsection{Structure of Grapheurs}\label{sec:structure_of_lims}
Theorem~\ref{thm:limits_formal} shows that limits of quotient-convergent sequences correspond to grapheurs in $\mc M$. Each such grapheur is described by a tuple $(E,\sigma,\varsigma,\theta,\vartheta)$ as in~\eqref{eq:kallenberg_char}, and the goal of this section is to interpret each element of the tuple in terms of the combinatorics of the convergent sequence of graphs. 
Informally, the five parameters describing quotient limits correspond to the following properties.
\begin{itemize}
    \item The array $E\in\RR_{\geq 0}^{\NN\times \NN}$ describes a discrete collection of hubs and the fraction of edge weight between them. Specifically, each entry corresponds to an edge in a sequence of graphs that carries an asymptotically positive fraction of the total edge weight.

    \item The arrays $\sigma,\varsigma\in\RR_{\geq 0}^{\NN}$ specify directed stars. Specifically, each entry corresponds to the fraction of in-degree or out-degree, respectively, of a vertex that comes from all edges with asymptotically vanishing edge weights.

    \item The parameters $\theta$ and $\vartheta$ capture the remaining edge weight in the graph, coming from all edges between vertices with asymptotically vanishing degrees and their self-loops, respectively.
\end{itemize}
These properties can be seen in all the examples of Section~\ref{sec:hom_convergence}. The following proposition proves these properties in general.
\begin{proposition}\label{prop:structure_result}
    Suppose $(G_n \subseteq \Delta^{n \times n})$ quotient-converges to a grapheur $\msf M$ that is described by the tuple $(E,\sigma,\varsigma,\theta,\vartheta)$ as in~\eqref{eq:kallenberg_char}. View each $G_n$ as an element of $\RR^{\NN\times\NN}$ by zero-padding it with infinitely-many rows and columns. Then there exist permutations $\pi_n\in\msf S_{\infty}$ of $\NN$ fixing all but finitely-many integers such that $\lim_n\|E-\pi_nG_n\|_{\infty}=0$, $\lim_n\|E\mathbbm{1}+\sigma-\pi_nG_n\mathbbm{1}\|_{\infty}=0$, and $\lim_n\|E^\top\mathbbm{1}+\varsigma-\pi_nG_n^\top\mathbbm{1}\|_{\infty}=0$. Furthermore, we have $\vartheta=\lim_n\mathrm{Tr}(G_n)-\mathrm{Tr}(E)$.
\end{proposition}
Proposition~\ref{prop:structure_result} justifies the intuition outlined above: Up to a relabelling of the vertices of $G_n$, the entries of $E$ are entrywise limits of the edge weights of $G_n$, the entries of $\sigma,\varsigma$ are limits of the remaining degrees of vertices in $(G_n)$ after removing edges with positive limiting weights, and the parameters $\vartheta,\theta$ capture the remaining edge weights in self loops and the rest of the edges, respectively. 
We prove Proposition~\ref{prop:structure_result} in Section~\ref{sec:missing_proofs_sampling} using an edge-sampling perspective on quotient convergence and an edge-based Szemer\'edi regularity lemma.
An immediate consequence of Proposition~\ref{prop:structure_result} is that all dense and degree-bounded sequences of simple graphs have the same limit.
\begin{corollary}\label{cor:undirected_graphs}
    Suppose $(G_n\in\mbb S^n_{\geq 0})$ is a sequence of undirected graphs such that the sequence $(G_n/\mathbbm{1}^\top G_n\mathbbm{1})$ quotient-converges to a grapheur described by $(E,\sigma,\varsigma,\theta,\vartheta)$.
    \begin{enumerate}
        \item We have $\sigma=\varsigma$ and $E^\top = E$.
        
        \item If each $G_n$ has no self-loops, then $\diag(E)=\vartheta=0$.

        \item If $G_n\in\mbb S^n_{\mathrm{sim}}$ is a sequence of simple graphs with $|E(G_n)|\to\infty$, then $E=0$. If $(G_n)$ is either dense or degree-bounded, then the sequence $(\frac{G_n}{2|E(G_n)|})$ quotient converges to the grapheur $\lambda^2$.

        %\item If $G_n\in\mbb S^n_{\mathrm{sim}}$ is a sequence of simple graphs with $|E(G_n)|\to\infty$, and is either dense or degree-bounded, then the sequence $(\frac{G_n}{2|E(G_n)|})$ quotient converges to the grapheur $\lambda^2$.
    \end{enumerate}
\end{corollary}
\begin{proof}
    The first two claims follow directly from Proposition~\ref{prop:structure_result}, since for any permutation $\pi_n\in\msf S_{\infty}$ the matrix $\pi_nG_n$ is, respectively, symmetric and has zero diagonal if $G_n$ is symmetric and has no self-loops.
    For the third claim, since each $G_n$ is simple, we have $G_n/\mathbbm{1}^\top G_n\mathbbm{1}\in\{0,1/n\}$ which converges entrywise to zero. Thus $E=0$ by Proposition~\ref{prop:structure_result}. 
    Since $G_n$ is symmetric, we have $G_n^\top\mathbbm{1}=G_n\mathbbm{1}=\deg(G_n)$ is the vector of degrees of the vertices in $G_n$. If $(G_n)$ is dense, then $\frac{\max_i\deg(G_n)_i}{2|E(G_n)|}\leq \frac{c}{n}\xrightarrow{n\to\infty}0$ for some constant $c>0$. 
    If $(G_n)$ is degree-bounded, with all degrees bounded by $k$, then $\frac{\max_i\deg(G_n)_i}{2|E(G_n)|} \leq \frac{k}{2|E(G_n)|}\xrightarrow{n\to\infty}0$. In either case, Proposition~\ref{prop:structure_result} shows that $\lim_n\frac{G_n}{2|E(G_n)|}=\lambda^2$ since $\diag(G_n)=0$ for all $n$.
\end{proof}

We end this section by giving a sufficient condition for the limit of a quotient-convergent sequence $(G_n)$ to consist entirely of discrete hubs, i.e., such that $\sigma=\varsigma=0$ and $\theta=\vartheta=0$ in~\eqref{eq:kallenberg_char}. 
In particular, this sufficient condition implies that the random quotients $(\rho(F_{k,n})G_n)$ and the grapheurs $(\msf M_{G_n})$ converge not just weakly but also in total variation. Such total variation convergence was shown for graphs sampled both from graphons~\cite[Cor.~10.25]{lovasz2012large} and from graphexes~\cite[\S5]{borgs2019sampling}, and was the primary mode of convergence studied for edge-exchangeable random graph models~\cite{Crane_Dempsey_2019}.  However, this condition is too restrictive for our purposes in this paper; for example, it does not apply to the sequence of stars in Example~\ref{ex:stars}.
In the following result, if $\mu,\nu$ are two discrete measures on a space $S$ then their total variation distance is $\mathrm{dist}_{\mathrm{TV}}(\mu,\nu)=\frac{1}{2}\sum_{x\in S}|\mu(\{x\})-\nu(\{x\})|$.
\begin{proposition}\label{prop:tv_convergence}
    Suppose $(G_n\in\Delta^{n\times n})$ and $E\in\RR_{\geq0}^{\NN\times \NN}$ satisfy $\lim_n\|\pi_nG_n-E\|_1=0$ for some permutations $\pi_n\in\msf S_n$. Then $(G_n)$ is quotient-convergent to $\msf M=\sum_{i,j}E_{i,j}\delta_{(T_i,T_j)}$. 
\end{proposition}
\begin{proof}
    Draw iid uniform $(T_i)$ and consider the coupling $\msf M_{G_n}'=\sum_{i,j}(\pi_nG_n)_{i,j}\delta_{(T_i,T_j)}$, $\msf M'=\sum_{i,j}E_{i,j}\delta_{(T_i,T_j)}$, using the same locations for both. Then with probability 1 we have $\mathrm{dist}_{\mathrm{TV}}(\msf M',\msf M_{G_n}')\leq \|\pi_nG_n-E\|_1/2\to0$ since the atoms $(T_i,T_j)$ are distinct almost surely. By~\eqref{eq:graph_from_measure}, we then have $W_1(\msf G_k[\msf M],\msf G_k[\msf M_{G_n}])\leq \mbb E\|\msf G_k[\msf M']-\msf G_k[\msf M'_{G_n}]\|_1\leq k^2\mbb E\mathrm{dist}_{\mathrm{TV}}(\msf M',\msf M'_{G_n})\to0$ for each $k$, proving quotient convergence.  
\end{proof}

\subsection{Duality between Grapheurs and Graphons}\label{sec:dual_graphons}
In this section, we characterize convergence of grapheurs in terms of graphons and vice-versa, exhibiting a facet of the duality between the two graph limit theories. Since graphon theory is most simply expressed in the undirected case (see~\cite[\S8]{diaconis2007graph} for the directed case), we restrict ourselves to this setting in the present section.

We begin by recalling a few basic notions from graphon theory. 
Let $\widetilde{\mc W} = \{W\colon[0,1]^2\to [0,1] \textrm{ measurable},\allowbreak\textrm{ symmetric}\}$ be the space of $[0,1]$-valued kernels. 
We define two kernels $W$ and $U$ to be equivalent, denoted $W\sim U$, if there exist (Lebesgue-)measure-preserving maps $\varphi,\psi\colon[0,1]\to[0,1]$ satisfying $W\circ(\varphi,\varphi)=U\circ(\psi,\psi)$.
This is the weak isomorphism between kernels commonly defined in the dense graph limits literature~\cite[Cor.~10.35]{lovasz2012large}. We denote by $\mc W=\widetilde{\mc W}/\sim$ the quotient space under this equivalence relation. Elements of $\mc W$ are called graphons~\cite[\S7.1]{lovasz2012large}.
A graphon $W\in \mc W$ defines random simple graphs of each size by setting $\mathbb{G}_k[W]=(\msf B_{T_i,T_j}^{(W)})_{i,j\in[k]}$ where $\msf B_{T_i,T_j}^{(W)}\sim\mathrm{Ber}(W(T_i,T_j))$ are drawn independently for $i < j$ while $\msf B_{T_j,T_i}^{(W)} = \msf B_{T_i,T_j}^{(W)}$ for $i>j$ and $\msf B_{T_i,T_i}^{(W)}=0$. 
Here $T_1,\ldots,T_k\overset{iid}{\sim}\mathrm{Unif}([0,1])$ as before. 
A sequence $(W_n)\subseteq\mc W$ converges if for each $k$, the sequence of random graphs $(\mathbb{G}_k[W_n])_k$ converges weakly for each $k$.\footnote{Graphon convergence is usually defined by total variation convergence of these random graphs, which is equivalent to their weak convergence since they are supported on the finite set $\mbb S^k_{\mathrm{sim}}$.}
This is equivalent to convergence of homomorphism densities and to convergence in cut metric~\cite[\S10]{lovasz2012large}.
Finally, a finite graph $G\in\mbb S^n_{\mathrm{sim}}$ defines a step graphon $W_G(x,y)=G_{i,j}$ if $x\in I_i^{(n)}\times I_j^{(n)}$ where $I_i^{(n)}=[(i-1)/n,i/n)$.

In our context, let $\mc M_{\mathrm{sym}}=\{\msf M\in\mc M: E=E^\top, \sigma=\varsigma, \diag(E)=0, \vartheta=0, \theta \in \mbb R_+ \textrm{ in~\eqref{eq:kallenberg_char}}\}$ be the collection of grapheurs corresponding to limits of undirected graphs with unit edge weight and no self-loops.
For any $W\in\widetilde{\mc W}$ and $\msf M\in \mc M_{\mathrm{sym}}$, define the random variable
\begin{equation}\label{eq:graphon_dual_pairing}\begin{aligned}
    \langle \msf M,W\rangle = &\sum_{i,j\in\NN}E_{i,j}\msf B_{T_i,T_j}^{(W)} + 2\sum_{i\in\NN}\sigma_i\int_0^1W(T_i,y)\, dy + \theta\int_{[0,1]^2}W(x,y)\, dx\, dy.
\end{aligned}\end{equation}
Note that this random variable exists and is bounded by $|\langle \msf M,W\rangle|\leq 1$ because both $(E_{i,j})$ and $(\sigma_i)$ are summable. Also note that $\langle \msf M,W_G\rangle = \mbb E_{\msf M}W_G=\langle G,\msf G_k[\msf M]\rangle$ if $G\in\mbb S^k_{\mathrm{sim}}$, since $W_G$ takes values in $\{0,1\}$ so $\msf B_{T_i,T_j}^{(W_G)}=W_G(T_i,T_j)$. Moreover, for finite graphs $G_1\in\Delta^{n\times n}, ~ G_2\in\mbb S^m_{\mathrm{sim}}$, we have
\begin{equation*}
    \langle\msf M_{G_1},W_{G_2}\rangle=\langle \rho(F_{m,n})G_1,G_2\rangle=\langle G_1,\rho(F_{m,n})^\star G_2\rangle,
\end{equation*}
which can equivalently be viewed as comparing a random quotient of $G_1$ to $G_2$ or comparing a random subgraph of $G_2$ to $G_1$ (see Section~\ref{sec:intro}). 
We now prove that the random variable~\eqref{eq:graphon_dual_pairing} only depends on the equivalence class of a graphon $W$.
\begin{proposition}
    For any $W,U\in\widetilde{\mc W}$, if $W\sim U$ then $\langle \msf M,W\rangle\overset{d}{=}\langle \msf M,U\rangle$ for any $\msf M\in \mc M_{\mathrm{sym}}$.
\end{proposition}
\begin{proof}
    Suppose $\msf M\in\mc M_{\mathrm{sym}}$ is defined by $(E,\sigma=\varsigma,\theta,\vartheta=0)$ as in~\eqref{eq:kallenberg_char}. 
    Since $W\sim U$, there are measure-preserving maps $\varphi,\psi\colon[0,1]\to[0,1]$ satisfying $W\circ(\varphi,\varphi)=U\circ(\psi,\psi)$. 
    We now define a coupling of $\langle \msf M,W\rangle$ and $\langle \msf M,U\rangle$ under which the two are equal almost surely. 
    Sample locations $(T_i)$ iid uniformly on $[0,1]$, and note that both $(\varphi(T_i)),(\psi(T_i))\overset{iid}{\sim}\mathrm{Unif}([0,1])$ because $\varphi$ and $\psi$ are measure-preserving. 
    Note that, conditioned on the locations $(T_i)$ above, we have $\msf B_{\varphi(T_i),\varphi(T_j)}^{(W)}\overset{d}{=}\msf B_{\psi(T_i),\psi(T_j)}^{(U)}$, since they are both Bernoulli random variables with success probability $W(\varphi(T_i),\varphi(T_j))=U(\psi(T_i),\psi(T_j))$. Letting $\msf B_{i,j}$ be a sample from the above distribution independently for each $i<j$, setting $\msf B_{i,j}=\msf B_{j,i}$ for $i>j$, and $\msf B_{i,i}=0$ for all $i$, we define the coupling of $\langle \msf M,W\rangle$ and $\langle \msf M,U\rangle$ by using~\eqref{eq:graphon_dual_pairing} with these common samples $(\msf B_{i,j})$ for both, and note that under this coupling they are equal almost surely. Thus, we have $\langle \msf M,W\rangle\overset{d}{=}\langle \msf M,U\rangle$ as desired.
\end{proof}
Thus, the pairing~\eqref{eq:graphon_dual_pairing} defines a pairing from $\mc M_{\mathrm{sym}}\times\mc W$ to bounded scalar random variables. Importantly, we can characterize convergence in $\mc M_{\mathrm{sym}}$ in terms of graphons and characterize graphon convergence in $\mc W$ in terms of grapheurs. The following is the formal version of Theorem~\ref{thm:graphon_duality_intro} from Section~\ref{sec:intro}. 
\begin{theorem}\label{thm:graphon_duality}
    We have $\lim_n\msf M_n=\msf M$ in $\mc M_{\mathrm{sym}}$ if and only if $\langle \msf M_n, W_G\rangle\to\langle\msf M,W_G\rangle$ weakly for all simple graphs $G$.
    We also have $\lim_nW_n\to W$ in $\mc W$ if and only if $\langle \msf M_G,W_n\rangle\to\langle\msf M_G,W\rangle$ weakly for all undirected graphs $G$ with unit edge weight and no self-loops.
\end{theorem}
\begin{proof}
    For the first claim, suppose $\langle\msf M_n,W_G\rangle\to\langle\msf M, W_G\rangle$ for all $G\in\mbb S^k_{\mathrm{sim}}$ and all $k\in\NN$. Therefore, $\mbb E_{\msf M_n}W_G\to\mbb E_{\msf M}W_G$ for any such $G$. Taking linear combinations of such $W_G$, we get $\mbb E_{\msf M_n}W\to\mbb E_{\msf M}W$ for any $W\colon[0,1]^2\to\RR$ that is piecewise-constant on the $k\times k$ uniform grid and equals zero on the diagonal for some $k\in\NN$. For any continuous and symmetric $f\colon[0,1]^2\to\RR$, define $f_k=\sum_{i,j=1}^kf_{i,j}^{(k)}\mathbbm{1}_{I_i^{(k)}\times I_j^{(k)}}$ where if $i\neq j$ we set $f_{i,j}^{(k)}=f(\frac{i}{k}-\frac{1}{2k}, \frac{j}{k}-\frac{1}{2k})$ to the value of $f$ at the center of the rectangle $I_i^{(k)}\times I_j^{(k)}$, and if $i=j$ we set $f_{i,i}^{(k)}=0$. Since $f$ is continuous on a compact set, it is uniformly continuous, so for any $\epsilon>0$ we have $|f(x)-f_k(x)|\leq\epsilon$ for all $x\in[0,1]^2\setminus\bigcup_{i=1}^k(I_i^{(k)}\times I_i^{(k)})$. Therefore, for any grapheur $\msf M'\in\mc M_{\mathrm{sym}}$ we have
    \begin{equation*}
        \mbb E_{\msf M'}|f-f_k|\leq \epsilon + \mbb E\msf M'\left(\bigcup_{i=1}^kI_i^{(k)}\times I_i^{(k)}\right)=\epsilon + \frac{1}{k},
    \end{equation*}
    for all large $k$ since $\mbb E\msf M'=\lambda^2$ for $\msf M'\in\mc M_{\mathrm{sym}}$. Thus, we have
    \begin{equation*}
        W_1(\mbb E_{\msf M_n}f,\mbb E_{\msf M}f)\leq 2(\epsilon+k^{-1}) + W_1(\mbb E_{\msf M_n}f_k, \mbb E_{\msf M}f_k) \xrightarrow{n\to\infty} 2(\epsilon+k^{-1}),
    \end{equation*}
    since $f_k$ is piecewise-constant on the $k\times k$ uniform grid and has zero values on its diagonal. We conclude by taking $k\to\infty$ that $\mbb E_{\msf M_n}f\to\mbb E_{\msf M}f$ weakly for any continuous $f\colon[0,1]^2\to\RR$, and hence that $\msf M_n\to\msf M$ weakly.
    Conversely, suppose $\msf M_n\to\msf M$ weakly, and let $W=W_G$ for $G\in\mbb S^k_{\mathrm{sim}}$. Then 
    \begin{equation*}
        \langle \msf M_n, W_G\rangle=\mbb E_{\msf M_n}W_G=\langle G,\msf G_k[\msf M_n]\rangle \to \langle G,\msf G_k[\msf M]\rangle = \mbb E_{\msf M}W_G  =\langle \msf M,W_G\rangle,    
    \end{equation*}
    weakly, since $\msf G_k[\msf M_n]\to\msf G_k[\msf M]$ weakly for each $k\in\NN$. This proves the first claim.
    
    % For any $W\in\mc W$ and any $\epsilon>0$,  we can find a subset $B\subseteq[0,1]^2$ and a continuous function $f\colon[0,1]^2\to [0,1]$ such that $|f-W|\leq\epsilon$ on $[0,1]^2\setminus B$ and $\mbb E\msf M(B)<\epsilon$, so that $\mbb E\msf M_n(B)<\epsilon$ for all large $n$. Then
    % \begin{equation*}\begin{aligned}
    %     W_1(\langle \msf M_n,W\rangle,\langle\msf M,W\rangle)&\leq \mbb E|\langle\msf M_n,W-f\rangle| + \mbb E|\langle\msf M,W-f\rangle| + W_1(\mbb E_{\msf M_n}f,\mbb E_{\msf M}f)\\
    %     &\leq 4\epsilon + W_1(\mbb E_{\msf M_n}f,\mbb E_{\msf M}f)\xrightarrow{n\to\infty}4\epsilon,
    % \end{aligned}\end{equation*}
    % proving that $\langle\msf M_n,W\rangle\to\langle\msf M,W\rangle$ weakly.

    For the second claim, suppose $\langle \msf M_G,W_n\rangle\to\langle\msf M_G,W\rangle$ weakly for all graphs $G\in \mbb S^k\cap \Delta^{k\times k}$ with $\diag(G)=0$, i.e., undirected weighted graphs with no self-loops and total edge weight of 1. We then have 
    \begin{equation*}
        \langle G,\mathbb{G}_k[W_n]\rangle = \langle \msf M_G,W_n\rangle\to\langle \msf M_G,W\rangle=\langle G,\mathbb{G}_k[W]\rangle,
    \end{equation*}
    weakly. Therefore, we get $\langle G,\mathbb{G}_k[W_n]\rangle\to\langle G,\mathbb{G}_k[W]\rangle$ for all $G\in\mbb S^k$ by linearity and the fact that $\mathbb{G}_k[W_n]$ all have zero diagonal. This implies $\mathbb{G}_k[W_n]\to \mathbb{G}_k[W]$ weakly by Cram\'er--Wold~\cite[Cor.~6.5]{kallenberg1997foundations}. Since this holds for all $k\in\NN$, we conclude that $W_n\to W$.
    Conversely, suppose $W_n\to W$. Then $\langle \msf M_G,W_n\rangle = \langle G,\msf G_k[W_n]\rangle\to\langle G,\msf G_k[W]\rangle=\langle \msf M_G,W\rangle$ for any $G\in\mbb S^k\cap\Delta^{k\times k}$.
\end{proof}
We conjecture that the above duality pairing is, in fact, bi-continuous.

\begin{conjecture}\label{conj:bicont_pair}
    If $(\msf M_n)\subseteq\mc M_{\mathrm{sym}}$ and $(W_n)\subseteq\mc W$ converge to $\msf M$ and $W$, respectively, then $\langle\msf M_n,W_n\rangle\to\langle\msf M,W\rangle$ weakly.
\end{conjecture}
In particular, this conjecture implies that $\msf M_n\to\msf M$ if and only if $\langle \msf M_n,W\rangle\to\langle \msf M,W\rangle$ for all graphons $W\in\mc W$, and $W_n\to W$ if and only if $\langle \msf M,W_n\rangle\to\langle \msf M,W\rangle$ for all grapheurs $\msf M\in\mc M_{\mathrm{sym}}$, a strengthening of Theorem~\ref{thm:graphon_duality}.

Theorem~\ref{thm:graphon_duality} exhibits a concrete, formal duality between quotient convergence and graphon convergence. 
This duality can be extended to general grapheurs and more general kernels on $[0,1]^2$, but the extension of graphon convergence to this setting is more complicated~\cite[\S8]{diaconis2007graph}. We therefore leave this and other extensions of this duality for future work.
%The restriction to symmetric random measures without diagonal terms is not necessary if we also allow asymmetric kernels identified up to equality $(\lambda^2+\lambda_D)/2$ almost everywhere, but the limit of such asymmetric kernels might not be representable by another asymmetric kernel, see
\subsection{Missing Proofs from Section~\ref{sec:limits}}\label{sec:missing_proofs}
In this section, we give all the proofs deferred from Section~\ref{sec:characterizing_lims} and Section~\ref{sec:structure_of_lims}, namely, Propositions~\ref{prop:comparing_metrics} and~\ref{prop:M0_compact}.%, as well as Lemma~\ref{lem:D_limits}.
\subsubsection*{Proof of Proposition~\ref{prop:comparing_metrics}}
Many of the proofs in this section are based on standard approximations of rectangles, and continuous functions by step functions on a uniform grid and vice-versa. Such approximations allow us to compare the measures of these objects under our random exchangeable measures using the following lemma.

\begin{lemma}\label{lem:means}
Fix any grapheur $\msf M\in\mc M$.
\begin{enumerate}
    \item There is $\theta\in[0,1]$ satisfying $\mbb E\msf M = \theta\lambda^2+(1-\theta)\lambda_D$.

    \item For any Borel set $B\subseteq[0,1]^2$, we have $\mbb E\msf M(B)\leq\max\{\lambda^2(B),\lambda_D(B)\}$. In particular, if $\lambda^2(B)=\lambda_D(B)=0$, then $\msf M(B)=0$ almost surely.
\end{enumerate}
\end{lemma}
\begin{proof}
    For the first claim, since $\msf M$ is exchangeable we conclude that $\mu=\mbb E\msf M$ is a (deterministic) exchangeable measure on $[0,1]^2$, meaning that $\mu\circ(\sigma,\sigma)=\mu$ for all (Lebesgue)-measure preserving bijections $\sigma\colon[0,1]\to[0,1]$. The space of such measures is two-dimensional and is spanned by $\lambda^2,\lambda_D$ by~\cite[Thm.~9.12]{kallenberg2005probabilistic}, so $\mu=\theta\lambda+\vartheta\lambda_D$ for some $\theta,\vartheta\in\RR$. Since $\msf M$ is a probability measure almost surely, its mean $\mu$ is a probability measure as well, necessitating $\theta+\vartheta=1$ and $\theta,\vartheta\geq0$. This proves the first claim.

    The second claim immediately follows from the first, noting in particular that if $\lambda^2(B)=\lambda_D(B)=0$ then the scalar random variable $\msf M(B)$ is nonnegative with a vanishing mean.
\end{proof}
To prove Proposition~\ref{prop:comparing_metrics}, we also need to relate couplings of random measures to couplings of random graphs derived from them.
\begin{lemma}\label{lem:graph_couplings}
    For any $\msf M_1,\msf M_2\in\mc M$ and coupling $(\msf G_1, \msf G_2)$ of $(\msf G_k[\msf M_1], \msf G_k[\msf M_2])$ for some $k\in\NN$, there is a coupling $(\msf M_1',\msf M_2')$ of $(\msf M_1,\msf M_2)$ such that $(\msf G_k[\msf M_1'],\msf G_k[\msf M_2'])\overset{d}{=}(\msf G_1,\msf G_2)$.
\end{lemma}
\begin{proof}
    We apply the disintegration theorem~\cite[Thm.~A]{disintegration} to the continuous map $\pi_k\colon \mc P([0,1]^2)\to \Delta^{k\times k}$ between compact metric spaces 
    sending $\mu\mapsto (\mu(I_i^{(k)}\times I_j^{(k)}))_{i,j\in[k]}$. 
    The disintegration theorem yields for each $\msf M\in\mc M$ a Borel collection of probability measures $(\lambda_G^{(\msf M)})_{G\in \Delta^{k\times k}}$ for each $\msf M\in\mc M$ such that $\mathrm{supp}(\lambda_G^{(\msf M)})\subseteq\pi_k^{-1}(G)$ and such that $\mathrm{Law}(\msf M)=\mu_k^{(\msf M)}\otimes \lambda_G^{(\msf M)}$ in the sense that 
    \begin{equation*}
        \mbb P[\msf M\in B] = \int_{\Delta^{k\times k}}\int_{\pi_k^{-1}(G)} \mathbbm{1}_B(\mu)\, d\lambda_G^{(\msf M)}(\mu)\, d\mu_k^{(\msf M)}(s),
    \end{equation*}
    for any Borel $B\subseteq \mc P([0,1]^2)$.
    The measure $\lambda_G^{(\msf M)}$ is the distribution of $\msf M$ conditioned on $\pi_k(\msf M)=G$. 
    
    Let $\gamma=\mathrm{Law}((G_1,G_2))$. We construct the coupling $\Gamma=\mathrm{Law}((\msf M_1',\msf M_2'))$ explicitly by setting
    \begin{equation*}
        \Gamma(B) = \int_{\Delta^{k\times k}\times\Delta^{k\times k}}\int_{\pi_k^{-1}(G_2)}\int_{\pi_k^{-1}(G_1)}\mathbbm{1}_B(\mu,\nu) \, d\lambda_{G_1}^{(\msf M)}(\mu)\, d\lambda_{G_2}^{(\msf N)}(\nu)\, d\gamma(G_1,G_2).
    \end{equation*}
    for each Borel $B\subseteq\mc P([0,1]^2)\times\mc P([0,1]^2)$.
    This distribution corresponds to first sampling $(\msf G_1, \msf G_2)\sim\gamma$ and then sampling $\msf M_1'$ and $\msf M_2'$ conditioned on $\msf G_k[\msf M_1']=\msf G_1$ and $\msf G_k[\msf M_2']=\msf G_2$.
    It is routine to check that this is a coupling of $(\msf M_1,\msf M_2)$ and that its pushforward under $(\pi_k,\pi_k)$ is the given coupling $\gamma$ of $(\msf G_1, \msf G_2)$, which is the claim we sought to prove.
\end{proof}
We can now prove Proposition~\ref{prop:comparing_metrics}.
\begin{proof}[Proof (Proposition~\ref{prop:comparing_metrics}).]
    Define the intervals $I_i^{(n)}=[(i-1)/n,i/n)$ for $i\in[n]$ and let $(\msf M',\msf N')$ be a coupling of $(\msf M,\msf N)$ attaining $W_{\square}(\msf M,\msf N)$ in~\eqref{eq:dual_cut_metric}. Such a coupling exists since the infimum in~\eqref{eq:dual_cut_metric} consists of minimizing a lower-semicontinuous function of the coupling over the compact space of all possible couplings.
    For each $k\in\NN$, we have
    \begin{equation*}\begin{aligned}
        W_1(\msf G_k[\msf M],\msf G_k[\msf N]) &\leq \sum_{i,j=1}^k\mbb E|\msf M'(I_i^{(k)}\times I_j^{(k)}) - \msf N'(I_i^{(k)}\times I_j^{(k)})|\\&\leq k^2W_{\square}(\msf M,\msf N),
    \end{aligned}\end{equation*}
    giving the first claimed inequality.

    For the second inequality, consider arbitrary intervals $S=[a_1,a_2]$ and $T=[a_3,a_4]$ for $a,b,c,d\in[0,1]$. For each $k\in\NN$ and $j\in[4]$, let $i_j\in[k]$ be the unique index such that $a_j\in I_{i_j}^{(k)}$ (if $a_j=1$ set $i_j=k$). Note that $S\subseteq \bigcup_{i=i_1}^{i_2}I_i^{(k)}$ and $|\bigcup_{i=i_1}^{i_2}I_i^{(k)}\setminus S|\leq 2/k$, and similarly for $T$. 
    Therefore, we have
    \begin{equation*}
        \lambda^2\left(\bigcup_{i=i_1}^{i_2}\bigcup_{j=i_3}^{i_4}I_i^{(k)}\times I_j^{(k)}\setminus S\times T\right)\leq \frac{4}{k},\quad \lambda_D\left(\bigcup_{i=i_1}^{i_2}\bigcup_{j=i_3}^{i_4}I_i^{(k)}\times I_j^{(k)}\setminus S\times T\right)\leq \frac{2\sqrt{2}}{k},
    \end{equation*}
    hence by Lemma~\ref{lem:means} we have
    \begin{equation*}
        \mbb E|\msf M'(S\times T)-\msf N'(S\times T)|\leq \mbb E\|\msf G_k[\msf M']-\msf G_k[\msf N']\|_1 + \frac{4}{k}.
    \end{equation*}
    Taking the supremum over $S,T$, we get
    \begin{equation*}
        W_{\square}(\msf M',\msf N')\leq \frac{4}{k} + \inf_{\substack{\textrm{couplings}\\ (\msf M',\msf N')}}\mbb E\|\msf G_k[\msf M']-\msf G_k[\msf N']\|_1 = \frac{4}{k} + W_1(\msf G_k[\msf M], \msf G_k[\msf N]),
    \end{equation*}
    where the last equality follows from Lemma~\ref{lem:graph_couplings}. This is the second claimed inequality.
    % We conclude that
    % \begin{equation*}
    %     W_{\square}(\msf M,\msf N) \leq W_1(\msf G_k[\msf M],\msf G_k[\msf N]) + \frac{4\sqrt{2}}{k}\leq k^4\delta_{\mathrm{samp}}(\msf M,\msf N) + \frac{4\sqrt{2}}{k}.
    % \end{equation*}
    % Since this holds for any $k\in\NN$, we set $k=\lfloor (\sqrt{2}/\delta_{\mathrm{samp}}(\msf M,\msf N))^{1/5}\rfloor$ to obtain
    % \begin{equation*}
    %     W_{\square}(\msf M,\msf N) \leq 2^{1/10}\delta_{\mathrm{samp}}(\msf M,\msf N) + \frac{4\sqrt{2}}{2^{1/10}\delta_{\mathrm{samp}}(\msf M,\msf N)^{-1/5} - 1} \leq 110\cdot \delta_{\mathrm{samp}}(\msf M,\msf N)^{1/5},
    % \end{equation*}
    % where we used the fact that $\delta_{\mathrm{samp}}(\msf M,\msf N)\leq \sum_{k\geq 1}k^{-4}=\frac{\pi^4}{90}$ and that $\frac{1}{x-1}\leq \frac{x_0}{x_0-1}\cdot \frac{1}{x}$ for all $x\geq x_0$.
\end{proof}

\subsubsection*{Proof of Proposition~\ref{prop:M0_compact}}
Next, we turn to proving Proposition~\ref{prop:M0_compact}, which states that $\mc M$ is weakly compact and that the metric $W_{\square}$ metrizes this weak topology. 
To prove that $\mc M$ is weakly compact, we rely on the following lemma and the characterization of $\mc M$ from~\cite[\S9]{kallenberg2005probabilistic} as the space of ergodic random exhangeable measures.
\begin{lemma}\label{lem:weak_conv_measure_01_sets}
    If $(\mu_n)$ is a weakly convergent sequence of probability distributions on a Polish space $S$ with limit $\mu$, and if $B\subseteq S$ is a measurable subset satisfying $\mu_n(B)=1$ for all $n$, then $\mu(B)=1$.
\end{lemma}
\begin{proof}
    By Urysohn's lemma, for any $\epsilon>0$ we can find a subset $B_{\epsilon}\supseteq B$ with $\mu(B_{\epsilon}\setminus B)\leq \epsilon$ and a continuous function $f\colon S\to[0,1]$ satisfying $f|_B=1$ and $f|_{S\setminus B_{\epsilon}}=0$. Since $f$ is continuous, we have $\mbb E_{\mu}f=\lim_n\mbb E_{\mu_n}f\geq \lim_n\mu_n(S)=1$. On the other hand, since $f\leq \mathbbm{1}_{B_{\epsilon}\setminus B} + \mathbbm{1}_B$ pointwise, we have $\mbb E_{\mu}f\leq \mu(B)+\epsilon$. We conclude that $\mu(B)\geq 1$, and since $\mu$ is a probability measure, that $\mu(B)=1$, as desired.
\end{proof}
We are now to prove Proposition~\ref{prop:M0_compact}.
\begin{proof}[Proof (Proposition~\ref{prop:M0_compact}).]
    We break down the proof into several steps.
    
    \paragraph{Compactness of $\mc M$:}
    The space $\mc P([0,1]^2)$ of probability measures on $[0,1]^2$ is compact in the weak topology. Therefore, the space $\mc P(\mc P([0,1]^2))$ of random such measures is again compact in its own weak topology (see Section~\ref{sec:notation}). 
    Thus, it suffices to prove that $\mc M\subseteq\mc P(\mc P([0,1]^2))$ is closed in this weak topology.
    
    By~\cite[Prop.~9.1]{kallenberg2005probabilistic}, the space of random exchangeable measures is precisely the collection of random measures on $[0,1]^2$ invariant under the countable group $\widetilde{\msf S}_{\infty}$ consisting of permutations of the intervals $(I_i^{(n)})_{i\in[n]}$ for all $n\in\NN$. By~\cite[Lemma~A1.2]{kallenberg2005probabilistic}, we can characterize its extreme points by
    \begin{equation*}
        \mc M = \left\{\msf M \textrm{ random exchangeable}: \mbb P[\msf M\in B]\in\{0,1\} \textrm{ whenever } B\subseteq\mc P([0,1]^2) \textrm{ is } \widetilde{\msf S}_{\infty}\textrm{-invariant}\right\}.
    \end{equation*}
    If $\msf M_n\to\msf M$ and $\mathrm{Law}(\msf M_n)(B)\in\{0,1\}$ for all $n$, then passing to a subsequence we may assume $\mathrm{Law}(\msf M_n)(B)=1$ for all $n$ or $\mathrm{Law}(\msf M_n)(B)=0$ for all $n$. Applying Lemma~\ref{lem:weak_conv_measure_01_sets} to either $B$ or its complement, respectively, yields $\mathrm{Law}(\msf M)(B)\in\{0,1\}$. 
    We thus conclude that $\mc M$ is weakly closed, and hence compact.
    
    \paragraph{Equivalence of statements 2 and 3:} The equivalence of statements 2 and 3 follows from Proposition~\ref{prop:comparing_metrics}, since $\lim_n \msf G_k[\msf M_n]=\msf G_k[\msf M]$ weakly for a given $k\in\NN$ if and only if $\lim_nW_1(\msf G_k[\msf M_n],\msf G_k[\msf M])=0$. 
    %Thus, it remains to prove the equivalence of statements 1 and 2.

    \paragraph{Equivalence of statements 3 and 4:} The equivalence of these two statements follows by the proof of Proposition~\ref{prop:quotient-part1}. Specifically, the sequence $(\msf G_k[\msf M_n])$ converges weakly to $\msf G_k[\msf M]$ if and only if the moments of $(\msf G_k[\msf M_n])$ converge to those of $\msf G_k[\msf M]$, and these moments are precisely the quotient densities of grapheurs defined in~\eqref{eq:grapheur_density}.
    
    \paragraph{Equivalence of statements 1 and 2:} Endow $\RR^2$ with the $\ell_{\infty}$ norm, endow $\mc P([0,1]^2)$ with the Wasserstein-1 metric with respect to this norm, and endow $\mc M$ with the Wasserstein-1 metric with respect to the Wasserstein-1 metric on $\mc P([0,1]^2)$, which metrizes weak convergence of random measures~\cite[Thm.~6.9]{villani2008optimal}. Explicitly, this metric is given by
    \begin{equation*}
        \mathrm{dist}_{W_1}(\msf M_1,\msf M_2) = \min_{\substack{\textrm{random } (\msf M_1',\msf M_2')\\ \msf M_i'\overset{d}{=}\msf M_i \textrm{ for } i=1,2} }\mbb E_{(\msf M_1',\msf M_2')}\sup_{\substack{f\colon[0,1]^2\to\RR\\ f(0)=0\\ \textrm{ 1-Lipschitz in } \ell_{\infty}}}|\mbb E_{\msf M_1'}f - \mbb E_{\msf M_2'}f|,
    \end{equation*}
    where the min over couplings is attained~\cite[Thm.~4.1]{villani2008optimal}.
    It therefore suffices to prove that $\lim_nW_{\square}(\msf M_n,\msf M)=0$ if and only if $\lim_n\mathrm{dist}_{W_1}(\msf M_n,\msf M)=0$.
    We proceed to prove this by standard approximations of continuous functions by step functions and vice-versa.

    Suppose that $\lim_n\mathrm{dist}_{W_1}(\msf M_n,\msf M)=0$ and fix any $\epsilon>0$. Let $(\msf M_n',\msf M')$ be a coupling of $(\msf M_n,\msf M)$ attaining $\mathrm{dist}_{W_1}$. 
    For any interval $S\subseteq[0,1]$, define $S_{\epsilon}=(S+[-\epsilon,\epsilon])\cap[0,1]$. Note that 
    \begin{equation*}
        \lambda^2(S_{\epsilon}\times T_{\epsilon}\setminus S\times T) \leq 4\epsilon,\qquad \lambda_D(S_{\epsilon}\times T_{\epsilon}\setminus S\times T)\leq 2\sqrt{2}\epsilon.
    \end{equation*}
    We can find an $\epsilon^{-1}$-Lipschitz $f_{\epsilon}\colon[0,1]^2\to[0,1]$ satisfying $f|_{S\times T}=1$ and $f|_{[0,1]^2\setminus S_{\epsilon}\times T_{\epsilon}}=0$. For example, set $f_{\epsilon}(x)=\frac{\mathrm{dist}(x,S_{\epsilon}\times T_{\epsilon})}{\mathrm{dist}(x,S\times T) + \mathrm{dist}(x,S_{\epsilon}\times T_{\epsilon})}$ where $\mathrm{dist}$ is with respect to the $\ell_{\infty}$ norm.
    For any $\msf M\in\mc M$, we have
    \begin{equation*}
        \mbb E|\msf M(S\times T)-\mbb E_{\msf M}f_{\epsilon}|\leq \mbb E\msf M(S_{\epsilon}\times T_{\epsilon}\setminus S\times T)\leq 4\epsilon,
    \end{equation*}
    by Lemma~\ref{lem:means}, hence
    \begin{equation*}
        |\mbb E\msf M_n'(S\times T)-\mbb E\msf M'(S\times T)|\leq 8\epsilon + |\mbb E_{\msf M_{n}'}f_{\epsilon}-\mbb E_{\msf M'}f_{\epsilon}|\leq 8\epsilon + \frac{1}{\epsilon}\mathrm{dist}_{W_1}(\msf M_n,\msf M).
    \end{equation*}
    We conclude that $W_{\square}(\msf M_n,\msf M)\leq 8\epsilon + \epsilon^{-1}\mathrm{dist}_{W_1}(\msf M_n,\msf M)$, so taking $n\to\infty$ and then $\epsilon\to0$ we conclude that $\lim_nW_{\square}(\msf M_n,\msf M)=0$.

    Conversely, suppose $\lim_nW_{\square}(\msf M_n,\msf M)=0$. Once again, we note that the infimum over couplings in~\eqref{eq:dual_cut_metric} is attained, since it consists of minimizing a lower-semicontinuous function of the coupling over the compact space of all possible couplings. For each $n$, let $(\msf M',\msf M_n')$ be a coupling attaining $W_{\square}(\msf M_n,\msf M)$. For each $N\in\NN$, define the intervals $I_i^{(N)}=[(i-1)/N,i/N)$ for $i\in\{1,\ldots,N-1\}$ and $I_N^{(N)}=[1-1/N,1]$. For any 1-Lipschitz $f\colon[0,1]\to\RR$ with $f(0)=0$, let 
    \begin{equation*}
        f_N(x)=\sum_{i,j=1}^Nf_{i,j}^{(N)}\mathbbm{1}_{I_i^{(N)}\times I_j^{(N)}}(x),
    \end{equation*}
    where $f_{i,j}^{(N)}=f(\frac{i}{N}-\frac{1}{2N},\frac{j}{N}-\frac{1}{2N})$ is the value of $f$ on the center of the squares $I_i^{(N)}\times I_j^{(N)}$. 
    Since $f$ is 1-Lipschitz in $\ell_{\infty}$ and the diameter of these squares is $1/2N$ in $\ell_{\infty}$, we have $|f(x)-f_N(x)|\leq \frac{1}{2N}$ for all $x\in[0,1]^2$. Furthermore, since $f(0)=0$ we have $|f_{i,j}^{(N)}|\leq 1$ for all $i,j$.
    Therefore, for any $\msf M\in\mc M$ we have
    \begin{equation*}
        |\mbb E_{\msf M}f - \mbb E_{\msf M}f_N|\leq \frac{1}{2N},
    \end{equation*}
    and for our particular coupling $(\msf M',\msf M_n')$, we have
    \begin{equation*}
        |\mbb E_{\msf M'}f-\mbb E_{\msf M_n'}f|\leq \frac{1}{N} + |\mbb E_{\msf M'}f_N-\mbb E_{\msf M_n'}f_N|\leq \frac{1}{N}+\sum_{i,j=1}^N|\msf M'(I_i^{(N)}\times I_j^{(N)})-\msf M_n'(I_i^{(N)}\times I_j^{(N)})|.
    \end{equation*}
    Since this inequality holds for any such function $f$, we conclude that
    \begin{equation*}
        \mathrm{dist}_{W_1}(\msf M_n,\msf M)\leq \frac{1}{N}+\sum_{i,j=1}^N\mbb E|\msf M'(I_i^{(N)}\times I_j^{(N)})-\msf M_n'(I_i^{(N)}\times I_j^{(N)})| \leq \frac{1}{N} + N^2W_{\square}(\msf M_n,\msf M).
    \end{equation*}
    Letting $n\to\infty$ and then $N\to\infty$, we conclude that $\lim_n\mathrm{dist}_{W_1}(\msf M_n,\msf M)=0$.
\end{proof}

\section{Edge Sampling and Property Testing}\label{sec:sampling}
In this section, we consider the problem of testing properties of large graphs by randomly sampling small summaries using our framework.  For properties specified by graph parameters that are continuous with respect to quotient convergence, it would seem that random quotients are natural candidates for obtaining small summaries of large graphs.  However, random quotients require access to the entire graph and are therefore impractical for large graphs.  Moreover, we will see in Section~\ref{sec:eqp_models} that random quotients do not concentrate very well, further obstructing property testing based on these summaries.

To remedy this situation, we present in this section an edge-based sampling procedure that yields a practical method for obtaining small summaries approximating large graphs of any size in the $W_\square$-metric with universal rates.  This sampling procedure is based on an equivalent view of quotient convergence in terms of convergence of edge-exchangeable random graph models, the equivalence being furnished by the classical de Finetti's theorem.  We describe this perspective on quotient convergence in Section~\ref{sec:sampling_edges} and an edge-based analog of the Szemer\'edi regularity lemma in Section~\ref{sec:szemeredy}.  We leverage these results to present a practical approach for property testing, which we formally define and study in Section~\ref{sec:prop_test}.  We conclude in Section~\ref{sec:missing_proofs_sampling} by presenting proofs from Section~\ref{sec:szemeredy} that are omitted from notational clarity as well as the proof of Proposition~\ref{prop:structure_result} from Section~\ref{sec:structure_of_lims}.

% While random quotients would be natural candidates for these summaries, they depend on the entire graph and are therefore impractical for large graphs. We seek random summaries that depend on only a few randomly-chosen vertices in the graph, similarly to subgraph sampling in the graphon and graphex literature. 
% Furthermore, we will see in Section~\ref{sec:eqp_models} below that large random quotients do not concentrate very well, further obstructing property testing based on them. 

% Instead of random quotients, we present an edge-based sampling procedure capable of approximating a graph of any size in $W_{\square}$ metric with universal rates. This sampling procedure is based on an equivalent view of quotient convergence in terms of convergence of edge-exchangeable random graph models, the equivalence being furnished by the classical de Finetti's theorem.

%\VC{I have questions about the organization of each of the subsections below.  All are minor.  But let's discuss before I make my pass.}

\subsection{Edge Sampling View of Quotient Convergence}\label{sec:sampling_edges}
We begin by describing a procedure to obtain finite random graphs from a grapheur by sampling edges with replacement.  We then reinterpret quotient convergence in terms of convergence of the resulting random edge-exchangeable graphs.

Given a grapheur $\msf M$, we can obtain a sequence of weighted random graphs $(\msf G^{(n)}[\msf M])_n$ where $\msf G^{(n)}[\msf M]$ has $n$ edges as follows.  We sample a measure $\mu\sim\msf M$ on $[0,1]^2$, followed by sampling $n$ edges $(x_1,y_1),\ldots,(x_n,y_n)\overset{iid}{\sim}\mu$.  Next we construct the weighted graph $\msf G^{(n)}[\msf M]$ on $N=|\{x_i\}\cup\{y_i\}|$ vertices by arbitrarily enumerating $\{t_1,\ldots,t_N\}=\{x_i\}\cup\{y_i\}$ and setting the edge weights to:
\begin{equation*}
(\msf G^{(n)}[\msf M])_{i,j}=|\{k\in[n]: (x_k,y_k)=(t_i,t_j)\}|/n
\end{equation*}
for $1 \leq i,j \leq N$.  Note that $|V(\msf G^{(n)}[\msf M])|\leq 2n$.
% \begin{enumerate}
%     \item Sample a measure $\mu\sim\msf M$ on $[0,1]^2$;
%     \item Sample $n$ edges $(x_1,y_1),\ldots,(x_n,y_n)\overset{iid}{\sim}\mu$;
%     \item Construct the weighted graph $\msf G^{(n)}[\msf M]$ on $N$ vertices by arbitrarily enumerating $\{t_1,\ldots,t_N\}=\{x_i\}\cup\{y_i\}$ and setting the edge weights to $\msf G^{(n)}[\msf M]_{i,j}=|\{k\in[n]: (x_k,y_k)=(t_i,t_j)\}|/n$. 

% Note that $|V(\msf G^{(n)}[\msf M])|\leq 2n$.

% \end{enumerate}
If the grapheur $\msf M$ is the limit of undirected graphs, we can further replace $\msf G^{(n)}[\msf M]$ by its symmetrization $\frac{\msf G^{(n)}[\msf M]+(\msf G^{(n)}[\msf M])^\top}{2}$ to obtain an undirected random graph; see Remark~\ref{rmk:undirected_graphs} below.  When $\msf M=\msf M_G$ is the grapheur corresponding to a finite graph $G$ as in~\eqref{eq:step_measure}, the above sampling procedure amounts to sampling with replacement the edges of $G$ proportionally to their weight, and randomly labeling the resulting graph.

For a general grapheur $\msf M$ described by parameters $(E,\sigma,\varsigma,\theta,\vartheta)$, edge sampling from $\msf M$ as above corresponds to first sampling random locations $(T_i)_i$ and then sampling edges as follows.
\begin{itemize}
    \item With probability $E_{i,j}$, sample the edge $(T_i,T_j)$. 
    
    \item With probability $\sigma_i$, we sample a new vertex $T'\sim\mathrm{Unif}([0,1])$ and sample the edge $(T_i,T')$. 
    With probability $\varsigma_i$, we analogously sample $(T',T_i)$.
    
    \item With probability $\theta$ we sample two new vertices $T_1',T_2'\sim\mathrm{Unif}([0,1])$ and add the edge $(T_1',T_2')$. Similarly, with probability $\vartheta$ we add the edge $(T'',T'')$ for a new location $T''\sim\mathrm{Unif}([0,1])$.
\end{itemize}
After repeating the above procedure $n$ times, we get a random exchangeable multigraph containing $n$ edges (possibly with repetition), and normalize the resulting graph to have total edge weight one. 
%\VC{The above construction is confusing -- let's discuss.}

%By interpreting these edge-exchangeable random graph models as limit objects in a metric space, we obtain a new convergence result for them.

We now observe that convergence of edge-sampled graphs is equivalent to quotient convergence.
\begin{proposition}\label{prop:edge_sample_view}
    We have $\lim_n\msf M_n=\msf M$ in $\mc M$ if and only if $(\msf G^{(k)}[\msf M_n])_n\subseteq\Delta^{2k\times 2k}$ converges weakly to $\msf G^{(k)}[\msf M]$ for all $k\in\NN$.
\end{proposition}
\begin{proof}
    The grapheurs $(\msf M_n)$ converge weakly if and only if the random exchangeable arrays 
    \begin{equation}\label{eq:random_edge_array}
        \begin{bmatrix} x_1^{(n)} & \cdots & x_k^{(n)}\\ y_1^{(n)} & \cdots & y_k^{(n)}\end{bmatrix}\in[0,1]^{2\times k},
    \end{equation}
    obtained by first sampling $\mu_n\sim\msf M_n$ and then sampling $(x_i,y_i)\overset{iid}{\sim}\mu_n$ converge weakly by~\cite[Thm.~1.1]{kallenberg1973characterization}, since the distributions of~\eqref{eq:random_edge_array} are precisely the finite-dimensional distributions of $\msf M_n$. The exchangeable array~\eqref{eq:random_edge_array} is precisely the list of edges of the multigraph $k\msf G^{(k)}[\msf M_n]$ listed in arbitrary order after assigning its vertices iid labels in $\mathrm{Unif}([0,1])$. Formally, we have $\msf M_{\msf G^{(k)}[\msf M_n]}\overset{d}{=}\frac{1}{k}\sum_{i=1}^k\delta_{(x_i^{(n)},y_i^{(n)})}$ for all $n,k\in\NN$. Thus, the sequence~\eqref{eq:random_edge_array} converges weakly for each $k\in\NN$ precisely when the sequence of random graphs $(\msf G^{(k)}[\msf M_n])_n$ converges weakly. 
\end{proof}
In particular, a sequence of graphs $(G_n\in\Delta^{n\times n})$ quotient-converges precisely when the sequence of random graphs $(\msf G^{(k)}[\msf M_{G_n}])$---obtained by independently sampling $k$ edges from each $G_n$ proportionally to their edge weights and randomly labeling the vertices---converges weakly for each $k\in\NN$. 

\begin{remark}[Edge-exchangeable random graphs]\label{rmk:edge_exch}
    The random graphs we obtain by sampling edges from a grapheur appeared previously in the literature as edge-exchangeable random multigraph models~\cite{janson2018edge}. As explained in~\cite[Rmk.~4.4]{janson2018edge}, every edge-exchangeable random multigraph model is obtained from a random exchangeable measure on $[0,1]^2$ as above after labelling the vertices by iid labels from $\mathrm{Unif}([0,1])$. 
    Nevertheless, to our knowledge graph limits based on these edge-exchangeable random graphs as in Proposition~\ref{prop:edge_sample_view} have not been previously studied in the literature, see also Section~\ref{sec:related_work}.

    As noted in~\cite{janson2018edge}, if we repeat the above edge sampling process infinitely-many times, each edge $(T_i,T_j)$ corresponding to a positive edge weight $E_{i,j}>0$ will be repeated infinitely-many times. In contrast, all the other edges sampled above will only appear once. In particular, the vertices $T',T_1',T_2'$ sampled in this process will only appear in one edge. 
    Such vertices are called ``blips'' in~\cite{borgs2019sampling,janson2018edge}, and the isolated edges and self-loops are called ``dust''. 
    %\VC{I'm still a bit puzzled as to the point of this remark -- are your edge exchangeable random graphs more general or a particular subclass than what was done previously?  Is there a succinct way to describe the distinction?} \TODO{point to related work and mention that our view of convergence of these is new}
\end{remark}

\subsection{Edge-Based Szemer\'edi Regularity}\label{sec:szemeredy}
We now provide bounds on the rate of convergence for $W_{\square}(\msf M_{\msf G^{(n)}[\msf M]},\msf M)$ to zero as $n$ grows large, both in expectation and with high probability.  In bounding the quantity $W_{\square}(\msf M_{\msf G^{(n)}[\msf M]},\msf M)$, it is important to distinguish between the randomness in the grapheur $\msf M_{\msf G^{(n)}[\msf M]}$ coming from the random atoms $T_i$ and randomness coming from $\msf G^{(n)}[\msf M]$. In our next result, we fix a realization of $\msf G^{(n)}[\msf M]$, compute the distance $W_{\square}(\msf M_{\msf G^{(n)}[\msf M]}, \msf M)$ depending on this realization (and hence random), and seek to bound both its expectation and its tails with respect to the randomness in $\msf G^{(n)}[\msf M]$.

In fact, our results provide convergence rates for the larger metric $\overline{W_{\square}}(\msf M,\msf N)$ defined by
\begin{equation*}
    W_{\square}(\msf M,\msf N)\leq \overline{W_{\square}}(\msf M,\msf N) = \inf_{\substack{\textrm{random } (\msf M_1',\msf M_2')\\ \msf M_i'\overset{d}{=}\msf M_i \textrm{ for } i=1,2} }\ \mbb E\left[\sup_{\substack{S,T\subseteq[0,1]\\ \textrm{intervals}}}|\msf M_1'(S\times T)-\msf M_2'(S\times T)|\right],
\end{equation*} 
where the expectation is taken with respect to the coupling $(\msf M_1',\msf M_2')$ of $\msf M_1$ and $\msf M_2$.
Note that $\overline{W_{\square}}$ differs from $W_{\square}$ in~\eqref{eq:dual_cut_metric} by moving the supremum over axis-aligned rectangles into the expectation. 

\begin{theorem}\label{thm:convergence_of_samples_v2}
    For any grapheur $\msf M\in \mc M$ and each $n\in\NN$, let $\msf G^{(n)}[\msf M]$ be a random graph obtained by sampling $n$ edges from $\msf M$. Then
    \begin{equation*}
        \mbb E_{\msf G^{(n)}[\msf M]}\left[W_{\square}(\msf M, \msf M_{\msf G^{(n)}[\msf M]})\right] \leq \mbb E_{\msf G^{(n)}[\msf M]}\left[\overline{W_{\square}}(\msf M, \msf M_{\msf G^{(n)}[\msf M]})\right] \leq \frac{174}{\sqrt{n}}.
    \end{equation*}
    Moreover, $W_{\square}(\msf M, \msf M_{\msf G^{(n)}[\msf M]}) \leq \overline{W_{\square}}(\msf M, \msf M_{\msf G^{(n)}[\msf M]})\leq \frac{174+\epsilon}{\sqrt{n}}$ with probability at least $1-\exp(-2\epsilon^2)$ over the random graph $\msf G^{(n)}[\msf M]$.
\end{theorem}
The proof of Theorem~\ref{thm:convergence_of_samples_v2} relies on two key lemmas. The first involves the construction of a particular coupling of $\msf M$, $\msf M_{\msf G^{(n)}[\msf M]}$, and $\msf G^{(n)}[\msf M]$. 
\begin{lemma}\label{lem:empirical_coupling}
    There is a coupling $(\msf M',\msf M'_{\msf G^{(n)}}, \msf G^{(n)})$ of $\msf M$, $\msf M_{\msf G^{(n)}[\msf M]}$, and $\msf G^{(n)}[\msf M]$ satisfying the following two properties. First, in this coupling $\msf M'_{\msf G^{(n)}}$ is obtained by sampling $\mu\sim\msf M'$, then sampling $n$ iid points from $\mu$, and finally forming the associated empirical measure. Second, conditioned on $\msf G^{(n)}=G$, the pair $(\msf M',\msf M'_{\msf G^{(n)}})$ is a coupling of $\msf M$ and $\msf M_{G}$.
    %For $\msf M\in\mc M$ and $n\in\NN$, fix a realization $\msf G^{(n)}[\msf M]$ of the random graph obtained by sampling $n$ edges from $\msf M$. There exists a coupling $(\msf M',\msf M_{\msf G^{(n)}[\msf M]}')$ of $\msf M$ and $\msf M_{\msf G^{(n)}[\msf M]}$ in which $\msf M_{\msf G^{(n)}[\msf M]}'$ is equal in distribution to an empirical measure obtained by sampling $n$ iid points from $\msf M'$ almost surely. 
    %\VC{What does it mean to say `empirical measure obtained by sampling $n$ iid points from $\msf M'$ almost surely' when $\msf M'$ is a random measure?}
\end{lemma}
The second step in the proof of Theorem~\ref{thm:convergence_of_samples_v2} involves using Vapnik-Chervonenkis theory to uniformly bound the discrepancy between a measure and its empirical sample on axis-aligned rectangles.
\begin{lemma}\label{lem:VC_of_AAR}
    For any measure $\mu\in\mc P([0,1]^2)$, let $\mu_n=\frac{1}{n}\sum_{i=1}^n\delta_{X_i}$ be the random empirical measure obtained by sampling $X_1,\ldots,X_n\overset{iid}{\sim}\mu$. Then
    \begin{equation*}
        \mbb E_{\mu_n}\sup_{\substack{S,T\subseteq[0,1]\\ \textrm{intervals}}}|\mu(S\times T)-\mu_n(S\times T)|\leq \frac{174}{\sqrt{n}}.
    \end{equation*}
\end{lemma}
The proofs of these two lemmas are straightforward but technical, and we defer them to Section~\ref{sec:missing_proofs_sampling} for clarity of exposition. Using the lemmas, we can now prove Theorem~\ref{thm:convergence_of_samples_v2}.
\begin{proof}[Proof (Theorem~\ref{thm:convergence_of_samples_v2}).]
    
    %Finally, note that for any measure-preserving bijection $\sigma\colon[0,1]\to[0,1]$, we have $(\msf M'\circ(\sigma,\sigma), \msf M'_{G^{(n)}}\circ(\sigma,\sigma))\overset{d}{=}(\msf M',\msf M'_{G^{(n)}})$ due to the random locations used in our above construction.

    Let $\mathcal{AAR}=\{[a,b]\times[c,d]: a,b,c,d\in[0,1]\}$ be the class of axis-aligned rectangles contained in $[0,1]^2$, and let $(\msf M',\msf M'_{\msf G^{(n)}},\msf G^{(n)})$ be the coupling from Lemma~\ref{lem:empirical_coupling}.
    Define 
    \begin{equation*}
        f(\msf G^{(n)}[\msf M]) = f((x_1,y_1),\ldots,(x_n,y_n)) = \overline{W_{\square}}(\msf M_{\msf G^{(n)}[\msf M]},\msf M),%\mbb E_{(\msf M',\msf M_{\msf G^{(n)}[\msf M]}')}\sup_{R\in\mc{AAR}}|\msf M'(R)-\msf M'_{\msf G^{(n)}[\msf M]}(R)|,
    \end{equation*}
    which is a function of the edges $(x_1,y_1),\ldots,(x_n,y_n)$ sampled iid from $\mu\sim\msf M$ to create $\msf G^{(n)}[\msf M]$. 
    First note that
    \begin{equation*}
        \mbb E_{\msf G^{(n)}[\msf M]}f(\msf G^{(n)}[\msf M]) = \mbb E_{\msf G^{(n)}}f(\msf G^{(n)}) \leq \mbb E_{(\msf M',\msf M'_{\msf G^{(n)}},\msf G^{(n)})}\sup_{R\in\mc{AAR}}|\msf M'(R)-\msf M'_{\msf G^{(n)}}(R)|\leq \frac{174}{\sqrt{n}},
    \end{equation*}
    where the equality follows from $\msf G^{(n)}[\msf M]\overset{d}{=}\msf G^{(n)}$, the first inequality follows by considering the coupling $(\msf M',\msf M'_{\msf G^{(n)}})|\msf G^{(n)}$ of $\msf M$ and $\msf M_{\msf G^{(n)}}$,
    and the last inequality follows by Lemma~\ref{lem:VC_of_AAR} and the fact that $\msf M'_{\msf G^{(n)}}$ is an empirical measure obtained from $n$ iid samples from $\msf M'$. 
    
    Second, note that $f$ is $1/n$-Lipschitz in each $(x_i,y_i)$. Indeed, if $\msf G^{(n,i)}[\msf M]$ differs from $\msf G^{(n)}[\msf M]$ only in the $i$th location $(x_i,y_i)$, and we glue the two couplings $(\msf M'',\msf M''_{\msf G^{(n)}[\msf M]})$ and $(\msf M'', \msf M''_{\msf G^{(n,i)}[\msf M]})$ attaining the corresponding $\overline{W_{\square}}$ distances as in~\cite[\S1]{villani2008optimal}, then 
    \begin{equation*}
        |f(\msf G^{(n,i)}[\msf M])-f(\msf G^{(n)}[\msf M])|\leq \mbb E_{(\msf M''_{\msf G^{(n)}[\msf M]},\msf M''_{\msf G^{(n,i)}[\msf M]})}\sup_{R\in\mc{AAR}}|\msf M''_{\msf G^{(n)}[\msf M]}(R)-\msf M''_{\msf G^{(n,i)}[\msf M]}(R)| \leq \frac{1}{n},
    \end{equation*}
    since the only difference between $\msf M''_{\msf G^{(n,i)}[\msf M]}$ and $\msf M''_{\msf G^{(n)}[\msf M]}$ is a weight of $1/n$ placed at one of their atoms.
    Since $(x_1,y_1),\ldots,(x_n,y_n)$ are sampled iid from a fixed measure $\mu$ (after it was sampled once from $\msf M$), the bounded-difference concentration inequality~\cite[Cor.~2.21]{wainwright2019high} implies that $f(\msf G^{(n)}[\msf M])\leq \mbb Ef(\msf G^{(n)}[\msf M])+\epsilon/\sqrt{n}$ with probability at least $1-\exp(-2\epsilon^2)$, proving the theorem.
\end{proof}

Using Theorem~\ref{thm:convergence_of_samples_v2}, we can now prove Theorem~\ref{thm:sampling_intro} from Section~\ref{sec:intro}.
\begin{proof}[Proof (Theorem~\ref{thm:sampling_intro}).]
    The first claim follows from Theorem~\ref{thm:convergence_of_samples_v2}. The second claim follows from the fact that $\mbb EW_{\square}(\msf M,\msf M_{\msf G^{(n)}[\msf M]})\leq\frac{174}{\sqrt{n}}$, which implies that some realization $G^{(n)}$ of $\msf G^{(n)}[\msf M]$, satisfies $W_{\square}(\msf M,\msf M_{G^{(n)}})\leq \frac{174}{\sqrt{n}}$. Setting $n=k(\epsilon)=\lceil(174/\epsilon)^2\rceil$, we obtain the second claim.
\end{proof}

We remark that $\overline{W_{\square}}$ also metrizes weak converge in $\mc M$, as can be shown by combining Theorem~\ref{thm:convergence_of_samples_v2} with Proposition~\ref{prop:edge_sample_view}. We opted to use the $W_{\square}$ metric in the rest of the paper since it is easier to relate to random quotients as in Proposition~\ref{prop:comparing_metrics}.
We conclude this section by remarking that the above discussion applies to undirected random graphs sampled from symmetric grapheurs.
\begin{remark}[Undirected graphs]\label{rmk:undirected_graphs}
    If $\msf M$ is the limit of undirected graphs, then by Corollary~\ref{cor:undirected_graphs} it is symmetric, i.e., we have $\msf M(S\times T)=\msf M(T\times S)$ for any measurable $S,T\subseteq[0,1]$. In that case, we can sample undirected graphs from $\msf M$ by first sampling $\msf G^{(n)}[\msf M]$ as above and symmetrizing it to get $\msf G_{\mathrm{sym}}^{(n)}[\msf M]=\frac{\msf G^{(n)}[\msf M] + (\msf G^{(n)}[\msf M])^\top}{2}$. For any $S,T\subseteq[0,1]$, we have
    \begin{equation*}\begin{aligned}
        |\msf M(S\times T)-\msf M_{\msf G_{\mathrm{sym}}^{(n)}[\msf M]}(S\times T)| &= \frac{1}{2}|\msf M(S\times T)+\msf M(T\times S) - \msf M_{\msf G^{(n)}[\msf M]}(S\times T)-\msf M_{\msf G^{(n)}[\msf M]}(T\times S)|\\&\leq \max\{|\msf M(S\times T)-\msf M_{\msf G^{(n)}[\msf M]}(S\times T)|, |\msf M(T\times S)-\msf M_{\msf G^{(n)}[\msf M]}(T\times S)|\},
    \end{aligned}\end{equation*}
    hence $W_{\square}(\msf M, \msf M_{\msf G_{\mathrm{sym}}^{(n)}[\msf M]})\leq W_{\square}(\msf M, \msf M_{\msf G^{(n)}[\msf M]})$ and similarly for $\overline{W_{\square}}$. Thus, all the results of this section apply to the undirected random graph $\msf G^{(n)}_{\mathrm{sym}}[\msf M]$.
\end{remark}

%Having defined and studied a local sampling procedure that allows us to summarize large graphs by small ones, we proceed to apply this procedure to property testing.

\subsection{Property Testing}\label{sec:prop_test}
We apply the preceding edge-based sampling procedure to test properties of large graphs based on small summaries.  We begin by formally defining the class of graph parameters we can test using our framework. Here a \emph{graph parameter} is a function $f\colon\bigsqcup_{n\in\NN}\Delta^{n\times n}\to\RR$ on graphs of all sizes that is permutation-invariant, in the sense that $f(\pi_n G_n)=f(G_n)$ for any permutation $\pi_n\in\msf S_n$ of the vertex labels.
\begin{definition}[Quotient-testable graph parameters]
    A \emph{quotient-testable graph parameter} is a graph parameter $f\colon\bigsqcup_{n\in\NN}\Delta^{n\times n}\to\RR$ satisfying the following properties:
    \begin{enumerate}[align=left, font=\emph]
        \item[(Isolate-indifference)] $f(G_N)=f(G_n)$ whenever $G_N \in \Delta^{N \times N}$ differs from $G_n \in \Delta^{n \times n}$ by isolated vertices;
        \item[(Continuity)] The sequence $(f(G_n))$ converges whenever $(G_n \in \Delta^{n \times n})$ is quotient-convergent.
    \end{enumerate}
\end{definition}
%In other words, a quotient-testable graph parameter is simply one that extends to a continuous function on the limit space $\mc M$. 
We proceed to give several equivalent characterizations of quotient-testable parameters.
\begin{proposition}[Testability as continuity]\label{prop:testability}
    The following are equivalent for a graph parameter $f$.
    \begin{enumerate}
        \item The parameter $f$ is quotient-testable.
        \item There is a function $\bar f\colon\mc M\to\RR$ satisfying $\bar f(\msf M_G)=f(G)$ for all $G\in\bigsqcup_n\Delta^{n\times n}$. Such $\bar f$ is then unique and continuous.
        \item For any $\epsilon>0$ there is a $\delta=\delta(\epsilon)>0$ such that if $W_{\square}(\msf M_{G_n},\msf M_{G_m})\leq\delta$ then $|f(G_n)-f(G_m)|\leq\epsilon$ for any two graphs $G_n,G_m$ of any two sizes.
        %\item There exists $L>0$ such that $|f(G_n)-f(G_m)|\leq L \cdot W_{\square}(\msf M_{G_n},\msf M_{G_m})$ for any two graphs $G_n,G_m$ of any two sizes.
        %\item For every $\epsilon>0$ there exists $k\in\NN$ such that $|f(G_n)-\mbb Ef(\rho(F_{k,n})G_n)|\leq k^2\epsilon$ for any $G_n\in\Delta^{n\times n}$ and any $n\in\NN$.
    \end{enumerate}
\end{proposition}
\begin{proof}
    We prove these statements imply each other in order, with the last statement implying the first.
    
    Suppose $f$ is quotient-testable. Define $\msf M_{<\infty}=\{\msf M_G:G\in\Delta^{n\times n} \textrm{ for some } n\}\to\RR$ to be the grapheurs in $\msf M$ arising from finite graphs and define $\bar f\colon\msf M_{<\infty}\to\RR$ by $\bar f(\msf M_G)=f(G)$. This is well-defined because $f$ is permutation-invariant and isolate-indifferent, as $\msf M_G=\msf M_{G'}$ if and only if $G$ and $G'$ differ by isolated vertices and node relabeling~\cite[Prop.~3]{kallenberg1990exchangeable}. Further, $\bar f$ is continuous by continuity of $f$, hence it extends uniquely to a continuous function on the closure $\mc M$ of $\mc M_{<\infty}$. This proves statement 2.

    Suppose $f$ extends to a continuous $\bar f\colon\mc M\to\RR$. Since $\mc M$ is compact and this compact topology is metrized by the $W_{\square}$ metric by Proposition~\ref{prop:M0_compact}, the function $\bar f$ is uniformly continuous in the $W_{\square}$ metric.
    %$L$-Lipschitz continuous in $W_{\square}$ for some $L>0$. 
    This proves statement 3.

    Finally, assume statement 3 holds. Since $W_{\square}(\msf M_{G_N},\msf M_{G_n})=0$ if $G_n$ and $G_N$ differ by isolated vertices, we conclude that $f$ is isolate-indifferent. Since quotient-convergent sequences $(G_n)_n$ are Cauchy in the $W_{\square}$-metric by Theorem~\ref{thm:limits_formal}, we conclude that $(f(G_n))_n\subseteq\RR$ is Cauchy and hence converges.
\end{proof}
%\new{In particular, if the extension $\bar f$ of $f$ is $L$-Lipschitz in $W_{\square}$ metric for some $L>0$, then we can take $\delta(\epsilon)=\epsilon/L$ in Proposition~\ref{prop:testability}(3).  In the sequel, we will see that many parameters of interest are indeed Lipschitz.} \VC{Are these two sentences even needed here?  We say much the same thing in the paragraph beginning ``Many graph parameters ...''.  I would delete the above two sentences.}

As a first set of examples, quotient densities and homomorphism numbers are quotient-testable.  More generally, we show next that a large class of polynomial graph parameters are testable, because all such polynomial parameters are functions of fixed-dimensional random quotients of their input graphs.
\begin{proposition}[Polynomial parameters]\label{prop:polys_are_testable}
    Suppose $f$ is an isolate-indifferent graph parameter such that $f|_{\Delta^{n\times n}}$ is given by a polynomial function for each $n$, and that all these polynomials have degree $d\in\NN$. Then $f$ is quotient-testable.  
    
    Further, a graph parameter $f$ is quotient-testable if and only if for any $\epsilon>0$ there exists a polynomial quotient-testable parameter $p$ satisfying $|f(G)-p(G)|\leq\epsilon$ for all graphs $G$ of all sizes.
\end{proposition}
\begin{proof}
    For the first claim, it was shown in~\cite[Thm.~5.2]{levin2025deFin} that there exists a polynomial $q_k$ over $\RR^{k\times k}$ satisfying $f(G_n)=\mbb Eq_k(\rho(F_{k,n})G_n)$ for all $G_n\in\Delta^{n\times n}$ and all $n\in\NN$.  Since $q_k$ is $L$-Lipschitz on $\Delta^{k\times k}$ for some $L>0$, we have by Proposition~\ref{prop:comparing_metrics} that
    \begin{equation*}\begin{aligned}
        |f(G_n)-f(G_m)| &= |\mbb Eq_k(\rho(F_{k,n}G_n))-\mbb Eq_k(\rho(F_{k,m})G_m)|\leq LW_1(\rho(F_{k,n})G_n,\rho(F_{k,m})G_m)\\&\leq Lk^2W_{\square}(\msf M_{G_n},\msf M_{G_m}),
    \end{aligned}\end{equation*}
    for any two graphs $G_n,G_m$ of any two sizes. Thus, the function $f$ is testable by Proposition~\ref{prop:testability}(3), as can be seen by setting $\delta(\epsilon)=\frac{\epsilon}{Lk^2}$.

    For the second claim, if for any $\epsilon>0$ there is a testable parameter $p$ satisfying $|f(G)-p(G)|\leq\epsilon$ for all $G$, then $f$ is a uniform limit of $W_{\square}$-continuous functions and hence is $W_{\square}$-continuous itself. By Proposition~\ref{prop:testability}, the parameter $f$ is therefore testable.
    Conversely, if $f$ is testable then it extends to a continuous $\bar f\colon\mc M\to\RR$ by Proposition~\ref{prop:testability} again.
    Note that the collection of polynomial testable parameters is an algebra, contains constant functions, and separates points by Corollary~\ref{cor:homs_determine_limits}. Since $\mc M$ is compact, Stone--Weierestrass implies that $\bar f$ is the uniform limit of polynomial testable parameters.
\end{proof}
We combine the insights of Sections~\ref{sec:sampling_edges}-\ref{sec:szemeredy} along with the above characterization of quotient-testable parameters to show that these can indeed be tested by edge sampling.
\begin{theorem}\label{thm:edge_testing}
    Consider a graph parameter $f\colon\bigsqcup_n\Delta^{n\times n}\to\RR$ whose restrictions $f|_{\Delta^{n\times n}}$ are continuous (in the usual topology on $\Delta^{n\times n}$) for each $n$. Then $f$ is quotient-testable if and only if for every $\epsilon>0$, there exists $k\in\NN$ such that for any finite graph $G$ of any size we have $|f(G)-f(\msf G^{(k)}[\msf M_G])|\leq\epsilon$ with probability at least $1-\epsilon$. In this case, we can take $k(\epsilon)=\left\lceil\left(\frac{174+\sqrt{\log(\epsilon^{-1/2})}}{\delta(\epsilon)}\right)^2\right\rceil$ where $\delta(\epsilon)>0$ is as in Proposition~\ref{prop:testability}(3).
    % \VC{I'm still a bit puzzled by the statement of this assertion.  In particular, if there exists a $k \in  \mbb N$, then do all $k' \geq k$ also work?} \TODO{yes}
\end{theorem}
\begin{proof}
    Suppose $f$ is quotient-testable. For any $\epsilon>0$, let $\delta(\epsilon)$ be as in Proposition~\ref{prop:testability}(3) and set $k=k(\epsilon)$ as in the theorem statement. Then, noting that $\epsilon=e^{-2\widetilde\epsilon^2}$ where $\widetilde\epsilon=\sqrt{\log(\epsilon^{-1/2})}$, with probability at least $1-\epsilon$ we have 
    $W_{\square}(\msf M_G,\msf M_{\msf G^{(k)}})\leq \frac{174+\widetilde\epsilon}{\sqrt{k}}\leq\delta(\epsilon)$ and hence
    $|f(G)-f(\msf G^{(k)}[\msf M_G])|\leq \epsilon$ by Theorem~\ref{thm:convergence_of_samples_v2}. 

    We turn to proving the converse. If $G_1$ and $G_2$ differ by isolated vertices, then $\msf M_{G_1}=\msf M_{G_2}$ and hence $\msf G^{(k)}[\msf M_{G_1}]\overset{d}{=}\msf G^{(k)}[\msf M_{G_2}]$. 
    Therefore, for any $\epsilon>0$ let $k\in\NN$ be as in the theorem statement and let $\msf G^{(k)}$ be a sample from this common distribution. Then with probability $1-2\epsilon$ we have $|f(G_1)-f(\msf G^{(k)})|\leq\epsilon$ and $|f(\msf G^{(k)})-f(G_2)|\leq \epsilon$, so in particular there is a realization $G^{(k)}$ of $\msf G^{(k)}$ satisfying both inequalities. We then have $|f(G_1)-f(G_2)|\leq |f(G_1)-f(G^{(k)})|+|f(G^{(k)})-f(G_2)|\leq 2\epsilon$.
    Since $\epsilon>0$ was arbitrary, we conclude that $f$ is isolate-indifferent. 
    A similar argument shows that $|f|$ is bounded. Indeed, since $|f|$ restricted to the compact set $\Delta^{2k\times 2k}$ is continuous, it is bounded by some $B_{2k}>0$. 
    For any other graph $G$ of any size, some realization $G^{(k)}$ of $\msf G^{(k)}[\msf M_G]$ satisfies $|f(G)-f(G^{(k)})|\leq\epsilon$, and hence $|f(G)|\leq B_{2k}+\epsilon$ for all graphs $G$.
    Denoting $B=\sup_G|f(G)|$, we then have $|f(G)-\mbb Ef(\msf G^{(k)}[\msf M_G])|\leq \mbb E|f(G)-f(\msf G^{(k)}[\msf M_G])| \leq \epsilon(1-\epsilon)+2B\epsilon$. 
    Finally, if $(G_n)$ is quotient-convergent, then $(\msf G^{(k)}[\msf M_{G_n}])_n$ is weakly-convergent, and by continuity of $f|_{\Delta^{2k\times 2k}}$, the sequence $(\mbb Ef(\msf G^{(k)}[\msf M_{G_n}]))_n$ is convergent. Therefore, we have
    \begin{equation*}
        |f(G_n)-f(G_m)|\leq 2\epsilon (1-\epsilon+2B) + |\mbb Ef(\msf G^{(k)}[\msf M_{G_n}])-\mbb Ef(\msf G^{(k)}[\msf M_{G_m}])|\leq 3\epsilon (1-\epsilon+2B)
    \end{equation*}
    for all large $n,m$, showing that $(f(G_n))$ is convergent.
\end{proof}
%\TODO{comment on dropping isolate indifference; comment on lack of uniform tail bounds on random quotients with forward ref}
We remark that the graphon analog of the characterization of quotient-testable parameters in Theorem~\ref{thm:edge_testing} was used as the definition of testable parameters in~\cite[Def.~2.11]{convergent_seqs1}.

The upshot of this result is that a continuous graph parameter being quotient-testable is equivalent to the parameter being testable using edge sampling.   In particular, Theorem~\ref{thm:edge_testing} gives a simple recipe for estimating the value of a testable parameter $f(G)$ on an arbitrarily-large graph $G$: sample $k$ edges iid from $G$ proportionally to their edge weights, and compute the parameter $f(G^{(k)})$ of the resulting small graph $G^{(k)}$. The number of edges needed is given explicitly as $k=\left\lceil\left(\frac{174+\sqrt{\log(\epsilon^{-1/2})}}{\delta(\epsilon)}\right)^2\right\rceil$ in terms of the desired accuracy $\epsilon$ and the modulus of continuity $\delta(\epsilon)$ of $f$ in the $W_{\square}$-metric. When $f$ is $L$-Lipschitz continuous in the $W_{\square}$-metric, we can set $\delta(\epsilon)=\epsilon/L$.

Many graph parameters of interest are indeed Lipschitz-continuous, including all polynomial parameters as the proof of Proposition~\ref{prop:polys_are_testable} demonstrates. Thus, the key quantity that is needed to bound the required number of edges to test the values of such parameters is their Lipschitz constants. We can explicitly bound the Lipschitz constants of quotient densities. Specifically, for a multigraph $H\in\NN^{k\times k}$ we have
\begin{equation*}
    |t_Q(H;\msf M_1)-t_Q(H;\msf M_2)| = |\mbb E\msf G_k[\msf M_1]^H - \mbb E\msf G_k[\msf M_2]^H|\leq k^2\|H\|_{\infty}W_{\square}(\msf M_1,\msf M_2),
\end{equation*}
since the polynomial $G\mapsto G^H$ is $\|H\|_{\infty}$-Lipschitz with respect to the $\ell_1$ norm on $\Delta^{k\times k}$.
By using~\eqref{eq:tQ_in_inj} and~\eqref{eq:hom_in_inj}, we can similarly derive Lipschitz constants for homomorphism numbers.  In the next example, we derive the Lipschitz constant of a particular graph parameter that can be used to test the presence of hubs.
% \VC{This paragraph feels a bit lacking in context after the proposition.  Is the main idea simply that quotient densites have known Lipschitz constants, and any quotient-testable graph parameter involving `elementary' quotient densities has a Lipschitz constant that be similarly characterized?} \TODO{I moved it here and added some context}
\begin{example}[Testing for hubs]\label{ex:hub_test}
    Consider the graph parameter 
    \begin{equation*}
        f(G) = \mathrm{hom}(S_2^{(\mathrm{out})};G)+\mathrm{hom}(S_2^{(\mathrm{in})};G)=\|G\mathbbm{1}\|_2^2+\|G^\top\mathbbm{1}\|_2^2,
    \end{equation*}
    where $S_2^{(\mathrm{out})},S_2^{(\mathrm{in})}\in\NN^{3\times 3}$ are outward- and inward-pointing stars with two leaves, respectively. This is an extension to weighted and directed graphs of the so-called first Zagreb index, originally introduced for chemical networks in~\cite{GUTMAN1972535}. Define the polynomial $q_2$ on $2\times 2$ matrices by
    \begin{equation*}
        q_2(G)=(G_{1,1}+G_{1,2}-G_{2,1}-G_{2,2})^2 + (G_{1,1}+G_{2,1}-G_{1,2}-G_{2,2})^2,
    \end{equation*}
    and observe that $f(G_n)=\mbb Eq_2(\rho(F_{2,n})G_n)$ for any $G_n\in\Delta^{n\times n}$ and any $n\in\NN$. Indeed, more generally for any grapheur $\msf M$ with random locations $(T_i)$, let $R_i\overset{iid}{\sim}\mathrm{Unif}(\{\pm 1\})$ be iid Rademacher random variables and note that
    \begin{equation*}\begin{aligned}
        \mbb Eq_2(\msf G_2[\msf M])&=\mbb E\left[\sum_{i\in\NN}(E\mathbbm{1}+\sigma)_i\left(\mathbbm{1}_{T_i<1/2} - \mathbbm{1}_{T_i>1/2}\right)\right]^2+\mbb E\left[\sum_{i\in\NN}(E^\top\mathbbm{1}+\varsigma)_i\left(\mathbbm{1}_{T_i<1/2} - \mathbbm{1}_{T_i>1/2}\right)\right]^2\\
        &= \mbb E\left[\sum_{i\in\NN}(E\mathbbm{1}+\sigma)_iR_i\right]^2+\mbb E\left[\sum_{i\in\NN}(E^\top\mathbbm{1}+\varsigma)_iR_i\right]^2=\|E\mathbbm{1}+\sigma\|_2^2+\|E^\top\mathbbm{1}+\varsigma\|_2^2,
    \end{aligned}\end{equation*}
    where $\mathbbm{1}_{T_i<1/2}=1$ if $T_i<1/2$ and zero otherwise and similarly for $\mathbbm{1}_{T_i>1/2}$, and we note that $(\mathbbm{1}_{T_i<1/2} - \mathbbm{1}_{T_i>1/2})_{i\in\NN}\overset{d}{=}(R_i)_{i\in\NN}$.
    We conclude that $f$ can be extended to all grapheurs $\msf M$ by $\bar f(\msf M)=\mbb Eq_2(\msf G_2[\msf M])$. 
    This extension is Lipschitz-continuous. Explicitly, since $|\partial_{G_{i,j}}q_2(G)|=2|G_{1,1}-G_{2,2}|$ if $i=j$ and $2|G_{i,j}-G_{j,i}|$ if $i\neq j$, we conclude that $q_2$ is 2-Lipschitz on $\Delta^{2\times 2}$ with respect to $\|\cdot\|_1$, and hence that
    \begin{equation*}
        |\bar f(\msf M_1)-\bar f(\msf M_2)| = |\mbb Eq_2(\msf G_2[\msf M_1])-\mbb Eq_2(\msf G_2[\msf M_2])|\leq 2W_1(\msf G_2[\msf M_1],\msf G_2[\msf M_2])\leq 8W_{\square}(\msf M_1,\msf M_2),
    \end{equation*}
    for any two grapheurs $\msf M_1,\msf M_2\in\mc M$ by Proposition~\ref{prop:comparing_metrics}. Finally, observe that $\bar f(\msf M)=0$ if and only if $E=0$ and $\sigma=\varsigma=0$, hence if and only if $\msf M$ has no hubs. In particular, $W_{\square}(\msf M_G,\msf M)\geq f(G)/8$ for any grapheur $\msf M$ with no hubs.

    We can now use the above parameter to test for the existence of significant hubs in a graph. 
    Specifically, suppose we are given an arbitrarily-large graph $G$,
    and we sample $k=\left\lceil 64\left(\frac{174+\sqrt{\log(\epsilon^{-1/2})}}{\epsilon}\right)^2\right\rceil$ edges from $G$ to obtain a graph $G^{(k)}$. Then with probability $1-\epsilon$, we have $W_{\square}(\msf M_G,\msf M)\geq \frac{f(G^{(k)})-\epsilon}{8}$ for any grapheur $\msf M$ with no hubs. In this sense, large values of $f(G^{(k)})$ indicate the likely presence of significant hubs in $G$.
    %and suppose that there exists an $\epsilon>0$ such that either one of two alternatives is true. Either there is a grapheur $\msf M$ with no hubs such that $W_{\square}(\msf M_G,\msf M)\leq\epsilon$; or $W_{\square}(\msf M_G,\msf M)\geq1-\epsilon$ for any grapheur $\msf M$ with no hubs. Sampling $k=64\lceil(1+174/\epsilon)^2\rceil$ edges from $G$ to obtain a graph $G^{(k)}$, we then have either $f(G^{(k)})\leq 9\epsilon$ with probability $1-2e^{-2\epsilon^2}$, or \TODO{bla}, respectively under the above alternatives. 
    %Thus, the value of $f(G^{(k)})$ allows us to distinguish between the two alternatives with high probability for \VC{what $\epsilon$?}.
\end{example}
We observe that testing for hubs also has an appealing interpretation in terms of quotients.  A grapheur $\msf M$ has no hubs if and only if there is a $\theta\in[0,1]$ satisfying $\msf G_m[\msf M]=\theta\frac{1}{m^2}\mathbbm{1}\mathbbm{1}^\top+(1-\theta)\frac{1}{m}I_m$ (a deterministic matrix) for all $m\in\NN$. 
Thus, testing for hubs also amounts to testing whether the random quotients of a graph $G$ deviate from a uniform distribution of edge weights, i.e., quotients of the above form. See Figures~\ref{fig:gold_coast}-\ref{fig:Celegans} for example. The parameter of Example~\ref{ex:hub_test} yields a lower bound on this deviation, as by Proposition~\ref{prop:comparing_metrics} we have
\begin{equation*}
    \min_{\theta\in[0,1]}\mbb E\left\|\theta\tfrac{1}{m^2}\mathbbm{1}\mathbbm{1}^\top+(1-\theta)\tfrac{1}{m}I_m-\msf G_m[\msf M_G]\right\|_1 \geq \frac{f(G)}{8} - \frac{4}{m}.
\end{equation*}
For example, the values of the parameter $f$ on the full graphs in Figures~\ref{fig:gold_coast}-\ref{fig:Celegans} are $0.0073$ and $0.010$, respectively. 
Estimating this parameter by sampling $k=10^3$ edges, we obtain values of $0.0084$ and $0.012$, respectively.
For comparison, the value of this parameter on an Erd\H{o}s--R\'enyi random graph with edge probability $1/2$ on 1,500 vertices is $0.001$. 

We end this section by giving two further examples of quotient-testable parameters, each quantifying different structural properties.
\begin{example}[Global Clustering Coefficient]
    One extension of the global clustering coefficient to weighted \emph{undirected} graphs, a global analog of the coefficient proposed in~\cite[\S7]{zhang2005general} for gene co-expression networks, is 
    \begin{equation}\label{eq:global_cluster}
        f(G)=\frac{\mathrm{inj}(K_3;G)}{\mathrm{inf}(P_2;G)\max_{i,j}G_{i,j}},
    \end{equation}
    where $\mathrm{inj}(H;G)$ denotes the injective homomorphism number and is defined as in~\eqref{eq:hom_nums} but with the sum restricted to injective maps, $K_3$ is an arbitrarily-directed triangle graph, and $P_2$ is the directed path on three vertices. We proceed to show that~\eqref{eq:global_cluster} is quotient-testable. Proposition~\ref{prop:polys_are_testable} shows that $\mathrm{inj}(K_3;G)$ and $\mathrm{inj}(P_2;G)$ extend to continuous functions on all of $\mc M$. Similarly, we note that $G\mapsto \max_{i,j}G_{i,j}$ extends to such a continuous function, since if $(G_n)$ quotient-converges to a grapheur described by $(E,\sigma,\varsigma,\theta,\vartheta)$, then $\max_{i,j}(G_n)_{i,j}\to\max_{i,j}E_{i,j}$ by Proposition~\ref{prop:structure_result}. Finally, note that $f(G)\leq 1$ for all graphs $G$ of any size, proving that $f$ extends to a continuous function on all of $\mc M$ and hence is quotient-testable by Proposition~\ref{prop:testability}.
\end{example}

\begin{example}[Katz Centrality]
    Another measure of importance, or influence, of a node, originally introduced in~\cite{Katz_1953} for social networks, is its Katz centrality. Specifically, for a parameter $\alpha\in(0,1)$, we consider the sum of the Katz centralities of all the nodes 
    \begin{equation*}
        f(G) = \mathbbm{1}^\top\Big[(I-\alpha G)^{-1}-I\Big]\mathbbm{1}=\sum_{k\geq 1}\alpha^k(\mathbbm{1}^\top G^k\mathbbm{1}),
    \end{equation*}
    which is another parameter quantifying the presence of hubs in a network. Once again, this parameter is quotient-testable. Indeed, if $(G_n)$ quotient-converges, then $\mathbbm{1}^\top G_n^k\mathbbm{1}$ converges for each $k$ by Proposition~\ref{prop:polys_are_testable}, hence $f(G_n)$ converges as well for any fixed $\alpha\in(0,1)$.
\end{example}

\subsection{Missing Proofs from Section~\ref{sec:sampling}}\label{sec:missing_proofs_sampling}
In this section we give all the proofs missing from Section~\ref{sec:sampling}.

\subsubsection*{Proofs from Section~\ref{sec:szemeredy}}
In this section, we prove the two lemmas used in Section~\ref{sec:szemeredy} to prove Theorem~\ref{thm:convergence_of_samples_v2}. We begin by proving Lemma~\ref{lem:empirical_coupling} by constructing an explicit coupling between $\msf M$, $\msf M_{\msf G^{(n)}[\msf M]}$, and $\msf G^{(n)}[\msf M]$ with the claimed properties.
\begin{proof}[Proof (Lemma~\ref{lem:empirical_coupling}).]
    Sample $\mu\sim\msf M$ and set $\msf M'=\mu$. Next, sample $(x_i,y_i)\overset{iid}{\sim}\mu$, let $\{t_1,\ldots,t_N\}=\{x_i\}\cup\{y_i\}$, form $\msf G^{(n)}$ as in Section~\ref{sec:sampling_edges}, and set 
    \begin{equation}\label{eq:Mgn_couple}
        \msf M'_{\msf G^{(n)}}=\frac{1}{n}\sum_{i=1}^n\delta_{(x_i,y_i)} = \sum_{i,j=1}^N(\msf G^{(n)})_{i,j}\delta_{(t_i,t_j)}.    
    \end{equation}
    We proceed to argue that this coupling satisfies all the claimed properties.
    
    By construction, we have $\msf M'\overset{d}{=}\msf M$ and $\msf G^{(n)}\overset{d}{=}\msf G^{(n)}[\msf M]$. Next, because $\msf M$ is exchangeable, we claim that $t_1,\ldots,t_N\overset{iid}{\sim}\mathrm{Unif}([0,1])$. 
    Indeed, $(x_i,y_i)\overset{d}{=}(\sigma(x_i),\sigma(y_i))$ for all $\sigma\in S_{[0,1]}$, hence $(x_i,y_i)\sim \theta\lambda^2+(1-\theta)\lambda_D$. Therefore, for each $i$ either $x_i=y_i\sim\mathrm{Unif}([0,1])$ or $x_i,y_i\overset{iid}{\sim}\lambda^2$. Combining this with the fact that the $(x_i,y_i)$ are independent for different $i$, we obtain the claim.
    Therefore, $\msf M_{\msf G^{(n)}[\msf M]}\overset{d}{=}\frac{1}{n}\sum_{i=1}^n\delta_{(x_i,y_i)}=\msf M'_{\msf G^{(n)}}$, the latter being an empirical measure drawn from $\mu\sim\msf M$. 
    Finally, we have $\msf M'|\msf G^{(n)}\overset{d}{=}\msf M$ and $\msf M_{\msf G^{(n)}}'|\msf G^{(n)}\overset{d}{=}\msf M_{\msf G^{(n)}}$ as can be seen from~\eqref{eq:Mgn_couple}, as $\msf G^{(n)}$ only depends on the multiplicities of $\{(t_i,t_j)\}$ rather than their locations, and the multiplicities and locations of a random exchangeable measure are independent of each other~\cite[Thm.~2]{kallenberg1990exchangeable}.
\end{proof}
We now turn to proving Lemma~\ref{lem:VC_of_AAR}, combining standard results from VC theory.
\begin{proof}[Proof (Lemma~\ref{lem:VC_of_AAR}).]
    Fix any $\mu \in\mc P([0,1]^2)$. By the proof of~\cite[Thm.~4.10]{wainwright2019high}, we have
    \begin{equation*}
        \underset{\substack{\mu^{(n)}\\ \textrm{emp. meas.}\sim \mu}}{\mbb E}\sup_{R\in\mc{AAR}}|\mu_n(R)-\mu(R)|\leq 2R_n(\mc F_{\mc{AAR}}),
    \end{equation*}
    where $R_n(\mc F)=\underset{\substack{X_i\overset{iid}{\sim}\mu\\\varepsilon_i\overset{iid}{\sim}\mathrm{Unif}(\{\pm 1\})}}{\mbb E}\sup_{f\in\mc F}\left|\frac{1}{n}\sum_{i=1}^n\varepsilon_if(X_i)\right|$ is the Rademacher complexity of a class of functions $\mc F$ and we set $\mc F_{\mc{AAR}}=\{\mathbbm{1}_R:R\in\mc{AAR}\}$. We further have by~\cite[Ex.~5.24]{wainwright2019high} that
    \begin{equation*}
        R_n(\mc F) \leq \frac{24}{\sqrt{n}}\int_0^{\mathrm{diam}(\mc F)}\mbb E\sqrt{\log N(t;\mc F;\|\cdot\|_{\mu_n})}\, dt,
    \end{equation*}
    where $N(t;\mc F;\|\cdot\|_{\mu_n})$ is the $t$-covering number of $\mc F$ in the (random) metric $\|f-g\|_{\mu_n}=\sqrt{\mbb E_{\mu_n}(f-g)^2}$ defined by the empirical measure $\mu_n$. If $\mc F$ is class of $\{0,1\}$-valued functions, then $\mathrm{diam}(\mc F)\leq 1$. 
    Finally, we have by~\cite[Cor.~1]{HAUSSLER1995217} that 
    \begin{equation*}
        N(t;\mc F_{\mc{AAR}};\|\cdot\|_{\mu})\leq e(V+1)\left(\frac{2e}{t}\right)^{V},
    \end{equation*}
    for any measure $\mu$, where $V=\mathrm{VC}(\mc{AAR})=4$ is the VC-dimension of the class of axis-aligned rectangles. Combining the above results, we conclude that
    \begin{equation*}
       \underset{\substack{\mu^{(n)}\\ \textrm{emp. meas.}\sim \mu}}{\mbb E}\sup_{R\in\mc{AAR}}|\mu_n(R)-\mu(R)|\leq \frac{48}{\sqrt{n}}\int_0^1\sqrt{\log(5e) + 4\log(2e/t)}\, dt\leq \frac{174}{\sqrt{n}},
    \end{equation*}
    where we numeriucally evaluated the last integral to 12 digits of accuracy to get $3.624000076562\leq 3.625$.
\end{proof}
\subsubsection*{Proof of Proposition~\ref{prop:structure_result} from Section~\ref{sec:structure_of_lims}}
Using Theorem~\ref{thm:convergence_of_samples_v2} and its proof, we can now prove Proposition~\ref{prop:structure_result} by comparing both the quotient-convergent sequence $(G_n)$ and its limit $\msf M$ to their edge-sampled random graphs.
\begin{proof}[Proof (Proposition~\ref{prop:structure_result}).]
    Since $(G_n)$ quotient-converges to $\msf M$, we have $W_1(\msf G^{(k)}[\msf M_{G_n}],\msf G^{(k)}[\msf M])\to0$ as $n\to\infty$ for each $k\in\NN$. Pick $k\geq (174/\epsilon)^2$, sample $(G_n^{(k)},G_{\msf M}^{(k)})$ from a coupling of $\msf G^{(k)}[\msf M_{G_n}]$ and $\msf G^{(k)}[\msf M]$ attaining $W_1(\msf G^{(k)}[\msf M_{G_n}],\msf G^{(k)}[\msf M])$, 
    and note that with probability $1-4e^{-2\epsilon^2}-\frac{W_1(\msf G^{(k)}[\msf M_{G_n}],\msf G^{(k)}[\msf M])}{\epsilon}$ we have 
    \begin{equation}\label{eq:all_good}
        \sup_{R\in\mc{AAR}}|\mu_{G_n^{(k)}}(R)-\mu_{G_n}(R)|\leq\epsilon,\quad \sup_{R\in\mc{AAR}}|\mu_{G_{\msf M}^{(k)}}(R)-\mu_{\msf M}(R)|\leq\epsilon,\quad \textrm{and}\quad \|G_n^{(k)} - G_{\msf M}^{(k)}\|_1\leq\epsilon,
    \end{equation}
    for $(\mu_{G_{\msf M}^{(k)}},\mu_{\msf M})\sim(\msf M_{G_{\msf M}^{(k)}}',\msf M')$ and $(\mu_{G_n^{(k)}},\mu_{G_n})\sim(\msf M_{G_n^{(k)}}',\msf M_{G_n}')$ sampled from the coupling of Lemma~\ref{lem:empirical_coupling} conditioned on $G_{\msf M}^{(k)}$ and $G_n^{(k)}$, respectively.
    Picking $n$ large enough, the above occurs with positive probability, so we can find deterministic $G_n^{(k)}$ and $G_{\msf M}^{(k)}$ satisfying~\eqref{eq:all_good} for all large $n$. 
    Explicitly, by Lemma~\ref{lem:empirical_coupling} there are locations $t_i\in[0,1]$ and a permutation $\pi_n\in\msf S_n$ satisfying
    \begin{equation*}
        \mu_{G_n}=\sum_{i,j=1}^n(G_n)_{i,j}\delta_{(t_i,t_j)},\quad \mu_{G_n^{(k)}}=\sum_{i,j=1}^{n}(\pi_nG_n^{(k)})_{i,j}\delta_{(t_i,t_j)}.
    \end{equation*}
    By considering rectangles in~\eqref{eq:all_good} centered at each of the atoms of $\mu_{G_n}$ and $\mu_{G_n^{(k)}}$, and around rows and columns of these two measures, we conclude that
    \begin{equation}\label{eq:all_good_1}
        \|G_n - \pi_n G_n^{(k)}\|_{\infty}\leq\epsilon,\quad \|G_n\mathbbm{1} - \pi_n G_n^{(k)}\mathbbm{1}\|_{\infty}\leq\epsilon,\quad \textrm{and}\quad \|G_n^\top\mathbbm{1} - \pi_n (G_n^{(k)})^\top\mathbbm{1}\|_{\infty}\leq\epsilon.
    \end{equation}
    Similarly, we can find a permutation $\widetilde\pi_{M}\in\msf S_M$ for large enough $M\in\NN$ and atoms $(t_i)_{i\in\NN}$ and $(t_i')_{i=1}^{M-m}$ for some $m\leq M$ satisfying
    \begin{equation*}\begin{aligned}
        &\mu_{\msf M}=\sum_{i,j\in\NN}E_{i,j}\delta_{(t_i,t_j)} + \sum_{i\in\NN}\sigma_i(\delta_{t_i}\otimes\lambda) + \sum_{i\in\NN}\varsigma_i(\lambda\otimes\delta_{t_i}) + \theta\lambda^2+\vartheta\lambda_D,\\ 
        &\mu_{G_{\msf M}^{(k)}}=\sum_{i,j=1}^m(\widetilde\pi_{M}G_{\msf M}^{(k)})_{i,j}\delta_{(t_i,t_j)} + \sum_{i=1}^m\sum_{j=m+1}^M(\widetilde\pi_{M}G_{\msf M}^{(k)})_{i,j}\delta_{(t_i,t_j')} + \sum_{i=m+1}^M\sum_{j=1}^m(\widetilde\pi_{M}G_{\msf M}^{(k)})_{i,j}\delta_{(t_i',t_j)}\\&\qquad\qquad + \sum_{i,j=m+1}^M(\widetilde\pi_{M}G_{\msf M}^{(k)})_{i,j}\delta_{(t_i',t_j')}.
    \end{aligned}\end{equation*}
    Here we have separated the edges sampled from each summand in $\msf M$.
    Moreover, with probability 1 we have $(\widetilde\pi_{M}G_{\msf M}^{(k)})_{i,j}\in\{0,1/k\}$ if $i>m$ or $j>m$.
    Note that $k^{-1}\leq(\epsilon/174)^2\leq\epsilon$ for all small $\epsilon>0$.
    By considering similar rectangles in~\eqref{eq:all_good}, we therefore
    \begin{equation}\label{eq:all_good_2}
        \|E-\widetilde\pi_MG_{\msf M}^{(k)}\|_{\infty}\leq\epsilon,\quad \|E\mathbbm{1}+\sigma - \widetilde\pi_M G_{\msf M}^{(k)}\mathbbm{1}\|_{\infty}\leq \epsilon,\quad \textrm{and}\quad \|E^\top\mathbbm{1}+\varsigma - (\widetilde\pi_M G_{\msf M}^{(k)})^\top\mathbbm{1}\|_{\infty}\leq\epsilon.
    \end{equation}
    Finally, since $\|G_n^{(k)}-G_{\msf M}^{(k)}\|_1\leq \epsilon$, we conclude that
    \begin{equation}\label{eq:all_good_3}
        \|G_n^{(k)}-G_{\msf M}^{(k)}\|_{\infty}\leq\epsilon,\quad \|G_n^{(k)}\mathbbm{1}-G_{\msf M}^{(k)}\mathbbm{1}\|_{\infty}\leq\epsilon,\quad \textrm{and}\quad \|(G_n^{(k)})^\top\mathbbm{1} - (G_{\msf M}^{(k)})^\top\mathbbm{1}\|_{\infty}\leq\epsilon.
    \end{equation}
    Combining~\eqref{eq:all_good_1},~\eqref{eq:all_good_2}, and~\eqref{eq:all_good_3}, we obtain
    \begin{equation*}
        \|\widetilde\pi_M\pi_n^{-1}G_n - E\|_{\infty}\leq 3\epsilon,\quad \|\widetilde\pi_M\pi_n^{-1}G_n\mathbbm{1} - (E\mathbbm{1}+\sigma)\|_{\infty}\leq 3\epsilon,\quad \textrm{and}\quad \|\widetilde\pi_M\pi_n^{-1}G_n^\top\mathbbm{1} - (E^\top\mathbbm{1}+\varsigma)\|_{\infty}\leq3\epsilon.
    \end{equation*}
    Since $\epsilon>0$ was arbitrary, our first three claims are proved. For the last claim, observe that $\lim_n\mathrm{Tr}(G_n)=\lim_n\mathrm{hom}([1],G_n)$ exists as it is the limit of the homomorphism numbers of a single self-loop. By Corollary~\ref{cor:homs_determine_limits} and the proof of Proposition~\ref{prop:hom_conv}, the limit of the homomorphism numbers is the same for any sequence of graphs quotient-converging to $\msf M$. By considering the sequence of graphs~\eqref{eq:particular_sequence} derived from $\msf M$, we conclude that $\lim_n\mathrm{Tr}(G_n)=\mathrm{Tr}(E)+\vartheta$, proving the last claim.
\end{proof}

\section{Equipartition-Consistent Random Graph Models}\label{sec:eqp_models}

Each grapheur $\msf M$ yields a sequence of random graphs $(\msf G_k[\msf M])$ via the construction~\eqref{eq:graph_from_measure}. In Section~\ref{sec:eqp_for_real}, we show that these graphs are related to each other via equipartitions, and hence form an equipartition-consistent random graph model (Definition~\ref{def:eqp_models}). We then prove that realizations of these random graphs $(\msf G_k[\msf M])$ converge back to the grapheur $\msf M$ in a suitable fashion. Finally, in Section~\ref{sec:eqp_proof} we prove Theorem~\ref{thm:eqp_models}, stating that any equipartition-consistent random graph model is a mixture of ones obtained from grapheurs as above.
% In this section, we study equipartition-consistent random graph models (Definition~\ref{def:eqp The quotients $(\msf G_k[\msf M])$ of a grapheur $\msf M$ are related to each other by equipartitions. In this section, we  that any sequence of random graph models related by equipartitions across graph sizes is given by a mixture of grapheurs.  Such `equipartition-consistent' models may be viewed as dual analogs of the consistent random graph models from the literature on dense graph limits~\cite{lovasz2012random}.  In Section~\ref{sec:eqp_for_real} we present the main definitions and results, with the more technical proofs given in Section~\ref{sec:eqp_proof}.

% The random graphs in this sequence are related across dimensions by equipartitions: If $D_{n,nk}\colon[nk]\to[n]$ is an equipartition, meaning that all its fibers have the same size, then $\rho(D_{n,nk})\msf G_{nk}^{(\msf M)}\overset{d}{=}\msf G_n[\msf M]$. We call such a sequence of random graphs \emph{equipartition-consistent}, and prove that conversely, any equipartition-consistent random graph model is a mixture of ones obtained from limit objects $\msf M\in\mc M$. Equipartition-consistent models are our dual analogs to consistent random graph models from the literature on dense graph limits~\cite{lovasz2012random}.

\subsection{Equipartition Consistency}\label{sec:eqp_for_real}
We begin by arguing that $(\msf G_k[\msf M])$ is equipartition-consistent. Note that $d_{k,nk}\circ F_{nk,m}\overset{d}{=}F_{k,m}$ for any $k,n,m\in\NN$ by~\cite[Lemma~4.13]{levin2025deFin}, since each $i\in[m]$ is mapped independently by $F_{nk,m}$ to one of the fibers of $d_{k,nk}$ with the same probability. Therefore, for any $G\in\Delta^{n\times n}$ the sequence of distributions $(\mathrm{Law}(\rho(F_{k,n})G_n))_k$ is indeed an equipartition-consistent model. 
If $(G_m)$ is a sequence of finite graphs converging to a grapheur $\msf M$, then $\rho(F_{k,m})G_m\to \msf G_k[\msf M]$ and $\rho(F_{nk,m})G_m\to\msf G_{nk}[\msf M]$ weakly, hence
$$\mathrm{Law}(\rho(d_{k,nk})\msf G_{nk}[\msf M])=\lim_m\mathrm{Law}(\rho(d_{k,nk}\circ F_{nk,m})G_m) = \lim_m\mathrm{Law}(\rho(F_{k,m})G_m)=\mathrm{Law}(\msf G_k[\msf M]),$$
showing that $(\mathrm{Law}(\msf G_k[\msf M]))_k$ is equipartition-consistent for any grapheur $\msf M$.

We now ask whether samples from an equipartition-consistent model $(\msf G_k[\msf M])$ obtained from a grapheur $\msf M\in\mc M$ converge back to $\msf M$. The answer depends on how this sampling is performed.
%We begin by proving almost sure convergence of samples drawn from a particular coupling of the random graphs $(\msf G_k^{(\msf M)})$.
\begin{proposition}\label{prop:eqp_as_convergence}
    For any grapheur $\msf M\in\mc M$, denote by $(\mu_k=\mathrm{Law}(\msf G_k[\msf M]))_k$ its associated equipartition-consistent random graph model.
    \begin{enumerate}
        \item Fix a grapheur $\msf M\in\mc M$, draw $\mu\sim\msf M$, and let $\msf Q_k=(\mu(I_i^{(n)}\times I_j^{(n)}))_{i,j\in[k]}$ for each $k\in\NN$. Then $\mathrm{Law}(\msf Q_k)=\mu_k$ for all $k\in\NN$, and $(\msf Q_k)$ is quotient-convergent to $\msf M$ almost surely.

        \item Suppose $\msf M=\delta_{(T_1,T_2)}$ is the grapheur associated to a single edge. Draw $\msf R_k\sim\mu_k$ independently for different $k$. Then $\limsup_kW_{\square}(\msf M_{\msf R_k},\msf M)\geq 1/4$ almost surely, so $(\msf R_k)$ does not converge to $\msf M$.
    \end{enumerate}
\end{proposition}
\begin{proof}
    For the first claim, fix any $\epsilon>0$ and let $N\in\NN$ be sufficiently large so that $\sum_{i>N\textrm{ or } j>N}E_{i,j} + \sum_{i>N}(\sigma_i+\varsigma_i)\leq\epsilon$. Drawing $\mu\sim\msf M$ amounts to sampling locations $T_i\overset{iid}{\sim}\mathrm{Unif}([0,1])$ in~\eqref{eq:kallenberg_char}, and these samples are distinct almost surely. Therefore, there exists $K\geq N\in\NN$ such that $T_1,\ldots,T_N$ lie in distinct intervals $(I_i^{(k)})_{i\in[k]}$ for all $k\geq K$. Therefore, for all $k\geq K$ we can find a permutation $\pi_k\in\msf S_k$ such that $\pi_k\msf Q_k = \rho(\iota_{k,N})(E_{i,j}+\sigma_i/k+\varsigma_j/k)_{i,j\in[N]} + N_k + \theta\frac{1}{k^2}\mathbbm{1}_k\mathbbm{1}_k^\top + \vartheta\frac{1}{k}I_k$ where $\|N_k\|_1\leq \epsilon$. Using the coupling $(\msf M_{\msf Q_k}',\msf M')$ obtained by drawing random locations $(T_i)$, setting $\msf M'$ to~\eqref{eq:kallenberg_char}, and setting $\msf M'_{\msf Q_k}=\sum_{i,j=1}^N(\pi_k \msf Q_k)_{i,j}\delta_{(T_i,T_j)}$, we conclude that $\limsup_{k\to\infty}W_{\square}(\msf M_{\msf Q_k},\msf M)\leq \epsilon$. 
    Since $\epsilon>0$ was arbitrary, we obtain the claimed convergence.

    For the second claim, note that with probability $1/k$ we have $T_1,T_2\in I_i^{(k)}$ for the same $i\in[k]$, in which case $\msf R_k=\diag(0,\ldots,1,\ldots,0)$ and $W_{\square}(\msf M_{\msf R_k},\msf M)\geq1/4$ as can be seen by setting $S=[0,1/2]$ and $T=[1/2,1]$ in~\eqref{eq:dual_cut_metric}. Since these events are independent and the sum of their probabilities $\sum_k k^{-1}$ diverges, Borel--Cantelli implies that $\limsup_k W_{\square}(\msf M_{\msf R_k},\msf M)\geq1/4$ almost surely.
\end{proof}
% We remark that the use of a specific coupling in the above proposition is necessary.
% \begin{remark}[Couplings of equipartition-consistent models]\label{rmk:couplings_needed}
%     Proposition~\ref{prop:eqp_as_convergence} shows that sampling growing-sized graphs from a particular coupling of $(G_k^{(\msf M)})$ yields a quotient-convergent sequence of graphs almost surely. The use of a specific coupling here is necessary. For example, consider $\msf M=\delta_{(T_1,T_2)}$, the random measure associated to a single directed edge. Sampling $G_n\sim \msf G_n[\msf M]$ independently for different $n$ yields a sequence that does \emph{not} quotient-converge almost surely. Indeed, with probability $1/n$ we have $T_1,T_2\in I_i^{(n)}$ for some $i\in[n]$, in which case $G_n=\diag(0,\ldots,1,\ldots,0)$ and $W_{\square}(G_n,\msf M)\geq1/4$ as can be seen by setting $S=[0,1/2]$ and $T=[1/2,1]$. Since these events are independent and the sum of their probabilities $\sum_n n^{-1}$ diverges, Borel--Cantelli implies that $\limsup_n W_{\square}(G_n,\msf M)\geq1/4$ almost surely.
% \end{remark}

In words, the quotients must be sampled in a manner that is related across different sizes for the samples to converge to the underlying grapheur; if we instead sample random quotients independently for each size, the resulting sequence almost surely does not converge to the underlying grapheur.
In contrast, if we sample growing-sized random graphs from a graphon independently for different sizes, then the resulting sequence does converge back to the graphon almost surely as a consequence of Szemer\'edi's regularity lemma~\cite[Lemma~11.8]{lovasz2012large}. 
Proposition~\ref{prop:eqp_as_convergence}(2) thus implies that the analog of Theorem~\ref{thm:sampling_intro} fails for random quotients. More precisely, for any rate of convergence $(R(k)\in\RR_{\geq0})_{n\in\NN}$ with $R(k)\to 0$, we have $\sum_{k\in\NN}\mathbb{P}[W_{\square}(\msf M,\msf G_k[\msf M])>R(k)]=\infty$ by Borel--Cantelli.
%Thus, no concentration inequality akin to Theorem~\ref{thm:convergence_of_samples_v2} can hold for random quotient sampling.
%The proof of Proposition~\ref{prop:eqp_as_convergence} suggests that the rate of convergence of samples from $(\msf G_k^{(\msf M)})$ back to $\msf M$ depends on the rate of decay of tails of the degree sequence $d_i=A_{i,i}+\sum_{j\neq i}(A_{i,j}+A_{j,i}) + b_i+b_i'$ of $\msf M$. 
We leave the precise analysis of convergence rates of random quotients for future work.   

\subsection{Proof of Theorem~\ref{thm:eqp_models}}\label{sec:eqp_proof}
Our goal in this section is to prove Theorem~\ref{thm:eqp_models}, showing that every equipartition-consistent random graph model arises from a unique random exchangeable measure. We deduce the existence of the desired random measure from results on inverse limits of random histograms~\cite{orbanz2011projective}, and proving that this random measure is exchangeable proceeds via standard approximation arguments. To apply the results of~\cite{orbanz2011projective}, we require the following lemma on the sequence of means of an equipartition-consistent model.
\begin{lemma}\label{lem:means_of_eqp}
    For any equipartition-consistent random graph model $(\mu_k)$, there exists $\theta\in[0,1]$ satisfying $\mbb E_{X\sim\mu_k}X = \theta\frac{1}{k^2}\mathbbm{1}_k\mathbbm{1}_k^\top + (1-\theta)\frac{1}{k}I_k$ for all $k\in\NN$.
\end{lemma}
\begin{proof}
    Note that $M_n=\mbb E_{X\sim\mu_n}X\in\mbb S^n$ is a deterministic matrix invariant under simultaneous permutations of its rows and columns, because $\mu_n$ is exchangeable. Any such matrix has the same on-diagonal and off-diagonal entries, hence there exist $\theta_n,\vartheta_n\in\RR$ satisfying $M_n=\theta_n\frac{1}{n}\mathbbm{1}_{n\times n}+\vartheta_n\frac{1}{n}I_n$. Since $\mathbbm{1}^\top X\mathbbm{1}=1$ almost surely, we similarly have $\mathbbm{1}^\top M_n\mathbbm{1}=\theta_n+\vartheta_n=1$. Also, since $X\geq0$ almost surely, we must have $\theta_n\in[0,1]$. Finally, since $\rho(d_{n,N})X_N\overset{d}{=}X_n$ whenever $n|N$, we must have
    \begin{equation*}
        M_n = \rho(d_{n,N})M_N = \theta_N\tfrac{1}{n^2}\mathbbm{1}_{n\times n}+(1-\theta_N)\tfrac{1}{n}I_n.
    \end{equation*}
    Since $\frac{1}{n^2}\mathbbm{1}_{n\times n}$ and $\frac{1}{n}I_n$ are linearly independent for all $n\geq2$, we must have $\theta_N=\theta_n$ whenever $n|N$, and therefore for all $n,N\in\NN$. This proves the claim.
\end{proof}
Combining the above lemma with~\cite{orbanz2011projective}, we are ready to prove Theorem~\ref{thm:eqp_models}.
\begin{proof}[Proof (Theorem~\ref{thm:eqp_models}).]
    Uniqueness is clear since squares of the form $I_i^{(k)}\times I_j^{(k)}$ generate the Borel $\sigma$-algebra on $[0,1]^2$. To prove existence we appeal to~\cite[Thm.~1.1]{orbanz2011projective}, which states that there exists a random probability measure $\msf M$ on $[0,1]^2$ (not guaranteed to be exchangeable) satisfying $\msf M(I_i^{(k)}\times I_j^{(k)})\sim\mu_k$ for all $k$ if there is a determinisitic measure $\mu\in\mc P([0,1]^2)$ satisfying $(\mu(I_i^{(k)}\times I_j^{(k)}))_{i,j\in[k]} =\mbb E_{X\sim\mu_k}X$ for all $k\in\NN$. Such a deterministic measure exists by Lemma~\ref{lem:means_of_eqp}, since 
    \begin{equation*}
        ((\theta \lambda^2+(1-\theta)\lambda_D)(I_i^{(k)}\times I_j^{(k)}))_{i,j\in[k]}=\theta \tfrac{1}{k^2}\mathbbm{1}_{k\times k} + (1-\theta)\tfrac{1}{k}I_k.
    \end{equation*} 
    We proceed to prove that the above $\msf M$ is exchangeable. To do so, it suffices to prove that for any continuous function $f\colon[0,1)^2\to\RR$ and any measure-preserving bijection $\sigma$, we have $\mbb E_{\msf M}f\overset{d}{=}\mbb E_{\msf M}f\circ(\sigma,\sigma)$. 
    By~\cite[Prop.~9.1]{kallenberg2005probabilistic}, it further suffices to prove this equality for $\sigma\in\msf S_m$, a permutation of the $m$ intervals $I_1^{(m)},\ldots, I_m^{(m)}$ acting by $\sigma((i-1)/m + x)=(\sigma(i)-1)/m+x$ for $x\in[0,1/m)$ for $i\in[m]$. 

    For any continuous function $f\colon[0,1]^2\to\RR$, let $f_n=\sum_{i,j=1}^nf_{i,j}^{(n)}\mathbbm{1}_{I_i^{(n)}\times I_j^{(n)}}$ where $f_{i,j}=f(\frac{i}{n}-\frac{1}{2n}, \frac{j}{n}-\frac{1}{2n})$ is the value of $f$ on the centers of the squares $I_i^{(n)}\times I_j^{(n)}$, whose diameter in $\ell_{\infty}$ is $1/2n$. Since $f$ is continuous on a compact set, it is uniformly continuous, hence for any $\epsilon>0$ we have $|f(x)-f_n(x)|\leq \epsilon$ for all $x\in[0,1)^2$ and all large $n$.
    Note that
    \begin{equation*}
        \mbb E_{\msf M}f_n = \sum_{i,j=1}^nf_{i,j}^{(n)}\mc M(I_i^{(n)}\times I_j^{(n)}) = \sum_{i,j=1}^nf_{i,j}^{(n)}(\msf G_n[\msf M])_{i,j}.
    \end{equation*}
    Therefore, for any $m,k\in\NN$ and $\sigma_m\in\msf S_m$, we view $\sigma_m$ as an interval permutation and as a permutation on $[mk]$ letters permuting consecutive intervals of length $k$, in which case we get
    \begin{equation*}
        \mbb E_{\msf M}f_{mk}\circ(\sigma_m,\sigma_m) = \sum_{i,j=1}^mf_{i,j}^{(mk)}(G_{mk}^{(\msf M)})_{\sigma^{-1}(i),\sigma^{-1}(j)}\overset{d}{=}\sum_{i,j=1}^mf_{i,j}^{(mk)}(G_{mk}^{(\msf M)})_{i,j} = \mbb E_{\msf M}f_{mk}.
    \end{equation*}
    Since this holds for all $k\in\NN$, we have
    \begin{equation*}\begin{aligned}
        W_1(\mbb E_{\msf M}f,\mbb E_{\msf M}f\circ(\sigma_m,\sigma_m))\leq &W_1(\mbb E_{\msf M}f,\mbb E_{\msf M}f_{mk}) + W_1(\mbb E_{\msf M}f_{mk}, \mbb E_{\msf M}f_{mk}\circ(\sigma_m,\sigma_m))\\ &+ W_1(\mbb E_{\msf M}f_{mk}\circ(\sigma_m,\sigma_m), \mbb E_{\msf M}f\circ(\sigma_m,\sigma_m))\leq 2\epsilon,
    \end{aligned}\end{equation*}
    for all large $k$,
    where the last inequality follows since $\msf M$ is supported on $[0,1)^2=\bigcup_{i,j}I_i^{(n)}\times I_j^{(n)}$. Since $\epsilon>0$ was arbitrary, this proves the exchangeability of $\msf M$.
\end{proof}

\section{Conclusions and Future Directions}\label{sec:conclusions}
We have introduced a new notion of limits of growing graphs based on convergence of their fixed-size random quotients. We showed that this notion of convergence can be equivalently characterized in terms of convergence of graph parameters such as homomorphism numbers and quotient densities. 
We characterized limits of quotient-convergent graph sequences via grapheurs, which are certain random exchangeable measures on $[0,1]^2$. The limiting grapheur of a quotient-convergent sequence of graphs describes the asymptotic distribution of edge weight in the graphs. 
We then presented another equivalent view of quotient convergence via edge sampling, allowing us to prove an edge-based analog of Szemer\'edi's regularity lemma and to test properties of arbitrarily-large graphs by sampling edges from them. The number of edges needed is indpendent of the size of the graph, only depending on the desired accuracy and the Lipschitz constant of the graph parameter in question. 
Finally, we studied equipartition-consistent random graph models, showing that any such model corresponds to the sequence of random quotients associated to a mixture of grapheurs. 
We conclude by mentioning several questions suggested by our work.
\begin{enumerate}[font=\textbf, align=left]
    %\item[(Random graphs via quotients)] What are the properties of random graphs obtained by first taking a quotient of a grapheur and then sampling its edges? What if we first sample edges and then randomly quotient? Both random graph models could be used to do statistically test properties by assuming that a given graph is sampled from the above distributions. \VC{This seems a bit confusing and doesn't read like it is conceptual.  One option is to get rid of it entirely.  Another option is to phrase things differently, maybe as follows:  ``We present two approaches for summarizing large graphs, one based on taking random quotients and the other based on edge sampling.  How are these related to each other?  In particular, what if one first takes a quotient of a large graph and then subsamples edges, or vice versa?  Both of these may be fruitful approaches for property testing, and it would be of interest to investigate the relative merits.}

    \item[(Optimal Szemer\'edi rate)] Is the rate of $O(1/\sqrt{n})$ for approximating a grapheur by a graph on $n$ edges optimal? 

    % \item[(Rates for random quotients)] Is there a Szemer\'edi-type regularity lemma for random quotients, i.e., does $\mbb E_{\msf G_k[\msf M]}W_{\square}(\msf M_{\msf G_k[\msf M]},\msf M)$ converge to zero at a universal rate independent of $\msf M$? If not, how does this rate depend on $\msf M$?

    \item[(Extending graphon duality)] Is Conjecture~\ref{conj:bicont_pair} on the bi-continuity of our duality pairing true? Can the duality with graphons in Section~\ref{sec:dual_graphons} be extended to limits of non-simple graphs? Can we relate the cut metric and our $W_{\square}$-metric via this duality? 

    \item[(Inequalities in quotient densities)] Corollary~\ref{cor:inequalities_in_quotients} shows that valid inequalities in quotient densities (or equivalently, homomorphism numbers) that hold for graphs of all sizes correspond to inequalities over grapheurs. Can we exploit the structure of the space of grapheurs to prove such inequalities, in analogy with the use of graphon theory to prove homomorphism density inequalities?
\end{enumerate}

\section*{Acknowledgements}
Eitan Levin and Venkat Chandrasekaran were supported in part by AFOSR grant FA9550-23-1-0070 and by NSF grant DMS-2502377.

\bibliographystyle{unsrt}
\bibliography{free_cvx_refs}

\end{document}